\documentclass[11pt]{amsart}
\usepackage{amsfonts,amssymb, amscd, latexsym, graphicx, psfrag, color, float, xspace}
\usepackage[all]{xy}

\usepackage{todonotes}
\usepackage{hyperref}
\usepackage{mathrsfs}

\newtheorem{dummy}{dummy}[section]
\newtheorem{lemma}[dummy]{Lemma}
\newtheorem{theorem}[dummy]{Theorem}
\newtheorem*{untheorem}{Main Theorem}

\newtheorem{proposition}[dummy]{Proposition}
\theoremstyle{definition}
\newtheorem{definition}[dummy]{Definition}
\newtheorem{notation}[dummy]{Notation}
\newtheorem{example}[dummy]{Example}
\newtheorem{remark}[dummy]{Remark}

\newcommand{\bA}{\mathbb{A}}
\newcommand{\bC}{\mathbb{C}}

\newcommand{\bG}{\mathbb{G}}
\newcommand{\bH}{\mathbb{H}}

\newcommand{\bN}{\mathbb{N}}

\newcommand{\bQ}{\mathbb{Q}}
\newcommand{\bR}{\mathbb{R}}
\newcommand{\bZ}{\mathbb{Z}}
\newcommand{\bS}{\mathbb{S}}

\newcommand{\cA}{\mathcal{A}}

\newcommand{\cC}{\mathcal{C}}

\newcommand{\cI}{\mathcal{I}}
\newcommand{\cK}{\mathcal{K}}
\newcommand{\cL}{\mathcal{L}}

\newcommand{\cN}{\mathcal{N}}
\newcommand{\cO}{\mathcal{O}}

\newcommand{\Pic}{\mathrm{Pic}}

\newcommand{\wt}[1]{\widetilde{#1}}

\newcommand{\Spec}{\mathrm{Spec}\,}

\newcommand{\Hom}{\mathrm{Hom}}

\renewcommand{\mod}{\mathrm{mod}}

\newcommand{\Div}{\mathrm{Div}}

\renewcommand{\log}{\mathrm{log}}

\newcommand{\gp}{\mathrm{gp}}
\newcommand{\an}{\mathrm{an}}

\renewcommand{\top}{\mathrm{top}}

\newcommand{\klog}{\mathrm{klog}}
\newcommand{\ket}{\mathrm{ket}}

\newcommand{\irs}[1]{\hspace{-3pt}\radice[\infty]{#1}}

\newcommand{\FP}{\mathrm{FP}}

\newcommand{\id}{\mathrm{id}}
\renewcommand{\wt}{\mathrm{wt}}

\newcommand{\Coh}{\mathrm{Coh}}

\renewcommand{\Im}{\mathrm{Im}}

\newcommand{\Par}{\mathrm{Par}}

\newcommand{\wti}{\widetilde}

\newcommand{\Qcoh}{\mathrm{Qcoh}}

\newcommand\radice[2][\relax]{\hspace{-1.5pt}\sqrt[\uproot{2}#1]{#2}}

\newcommand{\Preo}{(\mathrm{PreOrd})}
\newcommand{\op}{\mathrm{op}}

\newcommand{\res}{\mathrm{Res}}
\newcommand{\ind}{\mathrm{Ind}}

\newcommand{\Mod}{\mathrm{Mod}}

\newcommand{\Sym}{\mathrm{Sym}}

\begin{document}

\title[Parabolic sheaves on the Kato-Nakayama space]{Parabolic sheaves with real weights  \\ as  sheaves on the Kato-Nakayama space}

\author[Mattia Talpo]{Mattia Talpo}
\address{Department of Mathematics\\
Simon Fraser University\\
8888 University Drive\\
Burnaby BC\\
V5A 1S6 Canada, and Pacific Institute for the Mathematical Sciences\\ 4176-2207 Main Mall \\ Vancouver BC\\ V6T 1Z4 Canada}

\email{\href{mailto:mtalpo@sfu.ca}{mtalpo@sfu.ca}}



\maketitle

\begin{abstract}
We define quasi-coherent parabolic sheaves with real weights on a fine saturated log analytic space, and explain how to interpret them as quasi-coherent sheaves of modules on its Kato-Nakayama space. This recovers the description as sheaves on root stacks of \cite{borne-vistoli} and \cite{TV} for rational weights, but also includes the case of arbitrary real weights.
\end{abstract}

\setcounter{tocdepth}{1}
\tableofcontents

\section{Introduction}

The aim of the this paper is to present a correspondence between parabolic sheaves with real weights on a fine saturated log analytic space, and certain sheaves of modules on its Kato-Nakayama space. This was inspired by the corresponding equivalence for rational weights and root stacks \cite{borne-vistoli, TV}, and by the analogy between the infinite root stack and the ``profinite completion'' of the Kato-Nakayama space \cite{knvsroot, TVnew}.

Parabolic bundles were first defined by Mehta and Seshadri on curves in the '80s \cite{metha-seshadri}, and then studied in increasingly more general situations by several authors \cite{maruyama, mochizuki, iyer-simpson1, borne}, until Borne and Vistoli \cite{borne-vistoli} connected them to logarithmic structures, and gave a general definition (for rational weights with bounded denominator) on a coherent log scheme. They also constructed an equivalence of abelian categories between parabolic sheaves with weights in a fixed Kummer extension, and quasi-coherent sheaves on the corresponding stack of roots. Versions of this correspondence were earlier investigated by Biswas \cite{biswas} and Borne \cite{borne}, and it was further generalized to arbitrary rational weights by the author and Vistoli in \cite{TV}.

Assume for simplicity in this introduction that $X$ is a scheme of finite type over $\bC$, whose log structure is determined by a single effective Cartier divisor $D\subseteq X$, or, equivalently, by the line bundle with section $(L,s):=(\cO_X(D),1_D)$ (this data also gives a log analytic space, by analytifying $X$ and $D$).
In the algebraic setting, parabolic sheaves with weights in the group $\frac{1}{n}\bZ$ are sequences of quasi-coherent sheaves $\{E_{a}\}$ on $X$ for $a\in \frac{1}{n}\bZ$, with: a system of compatible maps $E_a\to E_b$ every time $b\geq a$, isomorphisms $E_{a+1}\cong E_a\otimes_{\cO_X} L$ for every $a$, and such that $E_a\to E_{a+1}\cong E_a\otimes_{\cO_X} L$ coincides with multiplication by $s\in \Gamma(L)$.
Clearly, such an object is completely determined by its restriction to the segment $[0,1]$, i.e. the diagram
\vspace{.2cm}
$$
\xymatrix@R=1em{
E_0\ar[r] & E_{\frac{1}{n}}\ar[r] & \cdots \ar[r] &E_{\frac{n-1}{n}}\ar[r] & E_1\cong E_0\otimes_{\cO_X} L.
}
$$
For a general fine saturated log scheme, a parabolic sheaf is also a system of sheaves with maps, indexed by a constructible sheaf of (possibly higher-rank) lattices.

The root stack $\radice[n]{X}$ parametrizes roots of the pair $(L,s)$, i.e. a morphism $T\to \radice[n]{X}$ corresponds to a map $f\colon T\to X$ and a pair $(N,t)$ on $T$ consisting of a line bundle with a section, with an isomorphism $(N,t)^{\otimes n}\cong f^*(L,s)$. There is a coarse moduli space morphism $\pi\colon \radice[n]{X}\to X$, which is an isomorphism outside of $D$. Points in the preimage of $D$ have a non-trivial stabilizer, the group of $n$-th roots of unity $\mu_n$. As for parabolic sheaves, the definition can be generalized to fine saturated log schemes. The main result of \cite{borne-vistoli}, in this particular case, says that there is an equivalence of abelian categories between parabolic sheaves with weights in $\frac{1}{n}\bZ$ and quasi-coherent sheaves on $\radice[n]{X}$.

The functor $\Phi_n\colon \Qcoh(\radice[n]{X})\to \Par(X,\frac{1}{n}\bZ)$ is easily described as follows: for a given $F\in \Qcoh(\radice[n]{X})$, one sets $\Phi_n(F)_{\frac{1}{n}k}:=\pi_*(F\otimes_{\cO_{\radice[n]{X}}} \cN^{\otimes k})$, where $\cN$ is the universal ``root line bundle'' on $\radice[n]{X}$. Moreover, for $\frac{1}{n}k \leq \frac{1}{n}k'$, there is a natural map $\cN^{\otimes k}\to \cN^{\otimes k'}$ given by the appropriate power of the global section $t$ of $\cN$, that induces a morphism $\Phi_n(F)_{\frac{1}{n}k}\to \Phi_n(F)_{\frac{1}{n}k'}$. The projection formula for $\pi$ assures that the other properties in the definition of a parabolic sheaf are satisfied. Heuristically, the presence of the non-trivial stabilizers $\mu_n$ along the divisor (and its action on fibers of sheaves) allows to encode the different pieces of the parabolic sheaves in a single sheaf on the root stack.

If we allow the index of the root to vary, these equivalences are compatible with the natural projections $\radice[m]{X}\to \radice[n]{X}$ for $n\mid m$, and in fact there is an analogous statement at the limit, on the infinite root stack $\radice[\infty]{X}=\varprojlim_n \radice[n]{X}$ \cite[Theorem 7.3]{TV}. This ``stacky'' point of view allows to treat parabolic sheaves as ``plain'' quasi-coherent sheaves on a slightly more complicated object, and has been useful in several instances (see for example \cite{iyer-simpson},  \cite{biswas-dhillon} and  \cite{talpo}).

In the original definition of Mehta and Seshadri, as well as in later instances, parabolic sheaves are allowed to have arbitrary real weights. In the situation of a scheme $X$ with a divisor $D$ as above, a parabolic sheaf with real weights is going to be a system of indexed sheaves as in the rational case, but the index group is the set of real numbers $\bR$. Finitely presented sheaves (appropriately defined) will be still determined by finitely many sheaves $E_r$ for $r\in [0,1]$ and the maps between them, but for general quasi-coherent sheaves, this is not the case.

As can certainly be expected, irrational weights are hard to handle in a purely algebraic manner. In this paper, we extend to real weights the correspondence with sheaves on root stacks, but using the Kato-Nakayama space instead. This forces us to work over the complex numbers.

Recall that the Kato-Nakayama space $X_\log$ is a topological space with a continuous proper projection $X_\log\to X$, where $X$ is now a (fine saturated) log analytic space. Morally, this construction replaces the log structure of $X$ with non-trivial topology in $X_\log$. For example, in the situation above, where the log scheme is determined by a single smooth divisor $D\subseteq X$ in a smooth analytic space $X$, the space $X_\log$ is the ``real oriented blowup'' of $D$ in $X$.

The use of the Kato-Nakayama space $X_\log$ is heuristically justified by the fact that the infinite root stack is a sort of ``profinite algebraic incarnation'' of the former: there is a morphism $X_\log\to \radice[\infty]{X}_\top$ to the topological realization of $\radice[\infty]{X}$, which is a ``profinite equivalence'' \cite[Theorem 6.4]{knvsroot}. The fiber of $X_\log\to X$ over a point $x$ can be identified with a real torus $(\bS^1)^r$, and the fiber of $\radice[\infty]{X}_\top\to X$ with $B\widehat{\bZ}^r$, where $r$ is the ``rank of the log structure'' at $x$. Thinking of $\bS^1$ as $B\bZ$, the morphism between the fibers $B\bZ^r\to B\widehat{\bZ}^r\cong \widehat{B\bZ^r}$ is the map to the profinite completion.
Morally, while the profinite monodromy (i.e. stabilizer group) in the fibers of $\radice[\infty]{X}\to X$ can only allow for rational weights in the parabolic sheaves, the fibers of $X_\log\to X$ have ``monodromy'' (i.e. fundamental group) with elements of infinite order, and the $\bS^1$s in the fibers can also encode real weights.

Assume that we are still in the simple situation of a log structure given by a divisor $D\subseteq X$ outlined above, and fix a submonoid $\Lambda$ of $\bR_+$, the non-negative real numbers, containing $\bN$. In order to make the heuristic of the previous paragraph precise, we adapt a procedure of Ogus \cite{ogus} to construct on $X_\log$ a sheaf of rings $\cO_\Lambda$, that extends the pullback of $\cO_X$ by adding sections of the form $f^\lambda$, where $f$ is a local equation of $D$ and $\lambda\in \Lambda$ (if $\Lambda=\frac{1}{n}\bN$, we are extracting $n$-th roots, in analogy with root stacks). The intuition for why this can be done, is that passing to $X_\log$ somewhat corresponds to extracting a logarithm of these local sections $f$, and if we have a logarithm we can also define $f^\alpha=\exp(\alpha\, \log(f))$ for any $\alpha \in \bR_+$.

After tensoring $\cO_\Lambda$ over the pullback of $\cO_X$ with the ``structure sheaf'' $\cO_X^\log$ of $X_\log$ (see \cite[Section 1]{illusie-kato-nakayama}), we obtain a sheaf of rings $\cO_\Lambda^\log$ on $X_\log$, that allows us to encode parabolic sheaves with weights in $\Lambda$ as quasi-coherent sheaves. The following is our main result.

\begin{untheorem}[Theorem \ref{thm:main}]
Let $X$ be a fine saturated log analytic space with log structure $\alpha\colon M\to \cO_X$, and $\Lambda$ a $\overline{M}^\gp$-saturated quasi-coherent sheaf of monoids, with $\overline{M}\subseteq \Lambda\subseteq \overline{M}_\bR=\overline{M}\otimes \bR_+$ (here, as usual, $\overline{M}$ denotes the sheaf of monoids $M/\cO_X^\times$).

Then we have an exact equivalence of categories
$$
\Qcoh(\cO_\Lambda^\log)\cong \Par(X,\Lambda)
$$
between quasi-coherent sheaves of $\cO_\Lambda^\log$-modules on $X_\log$ and quasi-coherent parabolic sheaves on $X$ with weights in $\Lambda$.
\end{untheorem}

We remark that quasi-coherence in this setting is a less transparent condition than in the algebraic case (see Remark \ref{rmk:qcoh} and the discussion in (\ref{sec:sheaves.of.modules})). The equivalence restricts to finitely presented sheaves on both sides, that are perhaps more natural objects. Moreover, this equivalence is compatible with the ones for root stacks of \cite{borne-vistoli} and \cite{TV} via the natural maps $X_\log\to \radice[n]{X}_\top$, as we verify in (\ref{sec:comparison.root}).

We plan to make use of this equivalence in future work, in at least a couple of directions. First, there are probably interesting interactions between these parabolic structures and integrable logarithmic connections, through Ogus' version of the Riemann-Hilbert correspondence \cite{ogus}. The sheaves of rings on $X_\log$ that he uses are closely related to the ones we use, and in fact his work on this subject was a fundamental inspiration. Second, the point of view advocated in this paper might be useful to study moduli spaces of parabolic sheaves with arbitrary real weights, and in particular for questions related to the variations of the weights. In moduli problems where there is a stability parameter, very often one has a wall-and-chamber decomposition of the space of possible parameters, and the moduli spaces undergo interesting transformations as the parameter crosses a wall. In the setting of parabolic sheaves, some versions of these questions have been investigated in \cite{boden, thaddeus}. Finally, it would be interesting to investigate whether parabolic sheaves with real weights exhibit some property of invariance under some simple kinds of log blowups, as the ones with rational weights do \cite[Proposition 3.9]{mckay}.

\subsection*{Outline}

Let us describe the contents of each section of the paper. We being by briefly recalling some basics about log schemes and log analytic spaces, and the construction of root stacks and Kato-Nakayama spaces in Section \ref{sec:preliminaries}. In Section \ref{sec:parabolic.sheaves} we extend the definition of parabolic sheaves on a log scheme of \cite{borne-vistoli} to the case of arbitrary real weights. We also include a brief reminder about the proof of the correspondence with quasi-coherent sheaves on root stacks, that we will adapt to the different context when proving Theorem \ref{thm:main}.
We then proceed in Section \ref{sec:sheaves.on.log} to describe how to equip the Kato-Nakayama space $X_\log$ of a fine saturated log analytic space $X$ with several sheaves of rings (depending on the monoid encoding the weights), and we discuss quasi-coherent and finitely presented sheaves on $X_\log$. Finally, Section \ref{sec:correspondence} contains the proof of our main result. We also describe how the correspondence with sheaves on the Kato-Nakayama space is related to the one on root stacks, via the natural map between the two objects.

\subsection*{Acknowledgements}

I am grateful to Niels Borne and Angelo Vistoli for allowing me to include in this work the basics on parabolic sheaves with real weights, that they had partly worked out in a preliminary version of \cite{borne-vistoli}. I am also happy to thank Clemens Koppensteiner for useful conversations, and the anonymous referee for several helpful comments and corrections.

This work was supported by the University of British Columbia, the Pacific Institute for the Mathematical Sciences and Simon Fraser University.

\subsection*{Notations and conventions}

All monoids will be commutative. The terminology ``{toric}'' for a monoid will mean fine, saturated and sharp (and hence torsion-free). If $P$ is a monoid, we will denote by $P^\gp$ the associated group, and by $P^+=P\setminus \{0\}$. If $P$ is a monoid and $S\subseteq P$ is a subset, $\langle S\rangle\subseteq P$ will denote the ideal of $P$ generated by $S$ (recall that $I\subseteq P$ is an ideal if $i+p\in I$ for every $i\in I$ and $p\in P$). We denote by $\bR_+$ the commutative monoid of non-negative real numbers, where the operation is addition, and by $\bR_{\geq 0}$ the monoid with the same underlying set, but where the operation is multiplication. If $P$ is a monoid and $X$ is a topological monoid, we will denote by $X(P)$ the topological monoid given by $\Hom(P,X)$.

We will typically use the same symbol for a locally finite type scheme over $\bC$, its associated complex analytic space and the underlying topological space of the latter (i.e. the set of closed points of the scheme), occasionally adding a subscript ``$\an$'' for analytifications. If $P$ is a finitely generated monoid, we will denote by $\bC(P)$ the complex analytic space $(\Spec \bC[P])_\an$.

Sheaves and stacks on a complex analytic space $X$ will always be sheaves and stacks on the classical analytic site.
A quasi-coherent sheaf on a complex analytic space $X$ will be a sheaf of $\cO_X$-modules that can locally be written as a filtered colimit of coherent sheaves, as in \cite[Section 2.1]{conrad}.
If $(T,\cO_T)$ is a ringed space, we will denote by $\Mod(\cO_T)$ the category of sheaves of $\cO_T$-modules on $T$, and by $\Mod_{\cO_T}$ the stack over (the classical site of) $T$, of sheaves of $\cO_T$-modules. A sheaf of $\cO_T$-modules will be called {finitely presented} if locally on $T$ it is the cokernel of a morphism of free $\cO_T$-modules of finite rank.

If $T$ is a topological space and $G$ is a topological monoid, or group, etc. we denote by $G_T$ the sheaf of continuous functions towards $G$ on opens of $T$ (with the induced structure of a sheaf of monoids, or groups, etc.). If $S$ is a set, the locally constant sheaf with fiber $S$ on the space $T$ will be denoted by $\underline{S}_T$. This will also have the induced structure, if $S$ is a monoid, or group, etc.

\section{Preliminaries}\label{sec:preliminaries}

In this section we briefly recall the basics of log schemes and log analytic spaces, root stacks and the Kato-Nakayama space. For more details, we refer the reader to \cite[Appendix]{knvsroot} and references therein.

\subsection{Log schemes and analytic spaces}

A log scheme is a scheme $X$ equipped with a sheaf of monoids $M$ on its small \'etale site, and a homomorphism $\alpha\colon M\to \cO_X$ (where $\cO_X$ is equipped with multiplication), that induces an isomorphism $\alpha|_{\alpha^{-1}\cO_X^\times}\colon \alpha^{-1}\cO_X^\times\cong \cO_X^\times$. Assuming that $M$ is a sheaf of integral monoids, this additional data is equivalent to a ``Deligne-Faltings structure'' (abbreviated by DF from now on), i.e. a symmetric monoidal functor $L\colon A\to \Div_X$ with trivial kernel (meaning that if $L(a)$ is an invertible object, then $a=0$), where $A$ is a sheaf of sharp monoids on $X$ and $\Div_X$ is the stack of line bundles with a section on the small \'etale site of $X$. Given a log structure $\alpha\colon M\to \cO_X$, the functor $L$ is obtained by modding out in the stacky sense by the action of $\cO_X^\times$ (so in particular the sheaf $A$ is the quotient $\overline{M}=M/\cO_X^\times$).

A morphism of log schemes $f\colon X\to Y$ is a morphism of schemes, together with a homomorphism of sheaves of monoids $f^{-1}M_Y\to M_X$ that is compatible with the maps to the structure sheaves. A morphism of log schemes is strict if this last homomorphism is an isomorphism (i.e. the log structure of $X$ is obtained by that of $Y$ by pullback). There is an analogous description of morphisms using DF structures.

If $P$ is a toric monoid, then the scheme $\Spec \bZ[P]$ (or $k[P]$ if we are working over a field $k$)  has a canonical log structure, determined by the homomorphism of monoids $P\to \bZ[P]$ in the following manner. Starting from the induced morphism of sheaves of monoids $a\colon \underline{P}_X \to \cO_{\Spec \bZ[P]}$, one obtains a log structure by forming the pushout $M=\underline{P}_X \oplus_{a^{-1}\cO_{\Spec \bZ[P]}^\times} \cO_{\Spec \bZ[P]}^\times$ in the category of sheaves of monoid on $X$, and considering the induced homomorphism of sheaves of monoids $\alpha\colon M\to \cO_{\Spec \bZ[P]}$. More generally, a Kato chart for the log scheme $X$ is a homomorphism of monoids $P\to \cO(X)$ that induces the log structure $\alpha\colon M\to \cO_X$ via the procedure just outlined. Equivalently, it is a strict morphism $X\to \Spec \bZ[P]$.

We will work with fine saturated log schemes, those for which, locally for the \'etale topology, we can find charts as above with $P$ integral, finitely generated and saturated (one can moreover take it to be sharp - this follows from example from \cite[Proposition 2.1]{Ols}). All of this also applies word for word to complex analytic spaces, more generally to ringed spaces or even just spaces equipped with a sheaf of monoids \cite{logstr}, where instead of the \'etale topology we use the ``classical'' topology. In particular, a log structure on a locally finite type scheme $X$ over $\bC$ induces a log structure on the analytification $X_\an$ (see for example \cite[Section 2.5]{TVnew}).

\subsection{Root stacks}\label{sec:root.stacks}

Let $X$ be a fine saturated log scheme or log analytic space. Assume that $\overline{M}\to B$ is a system of denominators, in the language of \cite{borne-vistoli}, i.e. it is an injective map of Kummer type (every section of $B$ locally has a multiple in $\overline{M}$), and $B$ has local charts by finitely generated monoids (a homomorphism of monoids $Q\to B(X)$ with $Q$ finitely generated, such that the induced morphism of sheaves $f\colon \underline{Q}_X\to B$ is a cokernel of sheaves of monoids, i.e. $B\cong \underline{Q}_X/\ker f$). A typical example is the inclusion $\overline{M}\to \frac{1}{n}\overline{M}$.

The root stack with respect to $B$, denoted $\radice[B]{X}$, is the stack over $X$ parametrizing liftings of $L\colon \overline{M}\to \Div_X$ to a symmetric monoidal functor $B\to \Div_X$.
It is a tame algebraic stack (Deligne-Mumford in characteristic $0$), with a proper quasi-finite coarse moduli space morphism $\radice[B]{X}\to X$. Roughly, $\radice[B]{X}$ is the stack obtained by extracting roots out of sections of $\overline{M}$, with respect to indices dictated by the sections of the sheaf of monoids $B$ (for instance if $B=\frac{1}{n}\overline{M}$, we are extracting $n$-th roots of all sections of $\overline{M}$).

Locally where $\overline{M}\to B$ has a chart $P\to Q$ (meaning that this is a Kummer homomorphism, $P\to \overline{M}(X)$ and $Q\to B(X)$ are charts, and the obvious square commutes), and the chart for $\overline{M}$ is a Kato chart (i.e. it lifts to $P\to M(X)$), the root stack $\radice[B]{X}$ is isomorphic to the quotient stack $$[\underline{\Spec}_X (\cO_X[P^\gp]\otimes_{\bZ[P]}\bZ[Q])/\widehat{Q}],$$ where the group $\widehat{Q}=\Hom(Q^\gp,\bG_m)$ acts on $$\underline{\Spec}_X (\cO_X[P^\gp]\otimes_{\bZ[P]}\bZ[Q])\cong \underline{\Spec}_X (\cO_X[P^\gp])\times_{\Spec \bZ[P]}\Spec \bZ[Q]$$ via the natural grading of the second factor, and via the homomorphism $\widehat{Q}\to \widehat{P}$ on the first factor \cite[Remark 4.14]{borne-vistoli}. In particular, quasi-coherent sheaves on $\radice[B]{X}$ can be identified with quasi-coherent sheaves of $\cO_X[P^\gp]\otimes_{\bZ[P]}\bZ[Q]$-modules on $X$ that have a $Q^\gp$-grading, compatible with the module structure.

This quotient presentation is more convenient to describe the correspondence with parabolic sheaves, but there is a perhaps simpler one, where the group that we quotient by is finite. Precisely, in presence of a Kato chart as above, there is an isomorphism
\begin{equation}\label{eq:root.stack}
\radice[B]{X}\cong [(X\times_{\Spec \bZ[P]}\Spec \bZ[Q])/\mu_{Q/P}],
\end{equation}
where $\mu_{Q/P}$ is the Cartier dual of the quotient $Q^\gp/P^\gp$, acting on $\Spec \bZ[Q]$ in the natural manner.

For $n\in \bN$, denote $\radice[\frac{1}{n}\overline{M}]{X}$ by $\radice[n]{X}$. These root stacks form an inverse system: if $n\mid m$ there is a natural map $\radice[m]{X}\to \radice[n]{X}$. The inverse limit is the infinite root stack $\irs{X}:=\varprojlim_n \radice[n]{X}$.

\subsection{The Kato-Nakayama space}\label{sec:kn}

Let $X$ be a log analytic space. Its Kato-Nakayama space $X_\log$ is a topological space (in good cases, a manifold with corners), defined as follows. As a set, elements of $X_\log$ are pairs $(x,\phi)$ consisting of a point $x\in X$ and a homomorphism of groups $\phi\colon M_x^\gp\to \bS^1$ such that $\phi(f)=\frac{f(x)}{|f(x)|}$ for every $f\in \cO_{X,x}^\times\subseteq M_x^\gp $.

If $X=\bC(P)=(\Spec\bC[P])_\an$ for a fine monoid $P$, then the space $X_\log$ can be identified with $\Hom(P,\bR_{\geq 0}\times \bS^1)$. More generally, if the log analytic space $X$ has a Kato chart $X\to \bC(P)$, then $X_\log$ can be identified with a closed subset of the topological space $X\times \Hom(P^\gp,\bS^1)$ (where $\Hom(P^\gp,\bS^1)$ has its natural topology), and we can equip it with the induced topology. This can be shown to be independent of the particular Kato chart that we choose, and we obtain a topology on the set $X_\log$ for a general $X$.

The natural projection $\tau\colon X_\log\to X$ that sends $(x,\phi)$ to $x$ is continuous and proper. The fiber over a point $x\in X$ can be identified with the space $\Hom(\overline{M}_x^\gp,\bS^1)$, which is non-canonically isomorphic to a real torus $(\bS^1)^r$, where $r$ is the rank of the (finitely generated) free abelian group $\overline{M}_x^\gp$. If $X=\bC(P)$, the map $X_\log\to X$ is identified with $\Hom(P,\bR_{\geq 0}\times \bS^1)\to \bC(P)=\Hom(P,\bC)$, sending a homomorphism $P\to \bR_{\geq 0}\times \bS^1$ to the composite with $\bR_{\geq 0}\times \bS^1\to \bC$, defined as $(r,a)\mapsto r\cdot a$. If the log structure of $X$ is given by a normal crossings divisor $D\subseteq X$, then the space $X_\log$ is the ``real oriented blowup'' of $X$ along $D$.

In the following we will also make us of a covering space $\wti{X}_\log\to X_\log$, that can be constructed in presence of a Kato chart $X\to \bC(P)$. For $\bC(P)$ itself, this is defined as $\wti{\bC(P)}_\log:=\Hom(P,\bH)$, where $\bH$ is the ``closed complex half-plane'' $\bR_{\geq 0}\times \bR\subseteq \bC$ (note that usually $\bH$ denotes the open half-plane), and the map $\wti{\bC(P)}_\log\to \bC(P)_\log=\Hom(P,\bR_{\geq 0}\times \bS^1)$ is given by composing with the map $\bH\to \bR_{\geq 0}\times \bS^1$ described as $(x,y)\mapsto (x, e^{iy})$. For a general $X$, the map $\wti{X}_\log\to X_\log$ is obtained by base change along the Kato chart $X\to \bC(P)$.

In both cases, $\wti{X}_\log\to X_\log$ is a covering space, with group of deck transformations given by $\bZ(P):=\Hom(P,\bZ)$ (or, more precisely, $\Hom(P,\bZ(1))$ where $\bZ(1)=2\pi i\bZ$ - we will systematically omit these ``Tate twists'' in the notation). Note that this can be non-canonically identified with $\bZ^n$, via $\Hom(P,\bZ)=\Hom(P^\gp, \bZ)$ and the fact that $P^\gp\cong \bZ^n$ for some $n$. This covering space should be thought of as an ``atlas'' of $X_\log$, the analogue of the scheme $X\times_{\Spec\bZ[P]}\Spec \bZ[Q]$ in the local description (\ref{eq:root.stack}) for root stacks.

In fact, if $Q=\frac{1}{n}P$, and $P\to Q= \frac{1}{n}P$ is the inclusion, the group $\mu_{Q/P}$ is isomorphic to $\mu_n(P):=\Hom(P,\bZ/n\bZ)\cong \mu_n^r$, where $r$ is the rank of $P^\gp$, and the group of deck transformations $\bZ(P)=\Hom(P,\bZ)$ of the cover $\wti{X}_\log\to X_\log$ naturally maps to $\mu_n(P)$. There is also a canonical map $\wti{X}_\log\to X\times_{\bC(P)}\bC(\frac{1}{n}P)$ that is $(\bZ(P)\to \mu_n(P))$-equivariant, and this gives a canonical morphism from $X_\log$ to the root stack $\radice[n]{X}$ (more precisely, to the underlying topological stack). If $X=\bC(P)$, the map $\wti{\bC(P)}_\log=\Hom(P,\bH)\to \bC(\frac{1}{n}P)=\Hom(\frac{1}{n}P,\bC)$ is given by composing $\frac{1}{n}P\cong P \to \bH$ with $\bH\to \bH$ given by $(x,y)\mapsto (\radice[n]{x}, y/n)$ (this steps ``compensates'' the identification $\frac{1}{n}P\cong P$), and then with $\bH\to \bC$ given by $(x,y)\mapsto x\cdot e^{iy}$.

The construction of this map can be globalized (see \cite{knvsroot} and \cite{TVnew}), so for every fine saturated log analytic space $X$ and every $n$ (including $n=\infty$) there is a canonical morphism of topological stacks $\phi_n\colon X_\log\to \radice[n]{X}_\top$. The rough idea here is that on $X_\log$ we have logarithms of sections of $\overline{M}$, so in particular we also have $n$-th roots of such sections, for any $n$, since $(\exp(\frac{1}{n}\log (z)))^n=z$.

On $X_\log$ there is a sheaf of rings $\cO_X^\log$, that is generated over $\tau^{-1}\cO_X$ by formal logarithms of sections of the sheaf $M$. Its precise definition will be recalled later (Section \ref{sec:rings.on.kn}).

\section{Parabolic sheaves with real weights}\label{sec:parabolic.sheaves}

For this section, $X$ will be either a fine saturated log scheme or log analytic space.

\subsection{Sheaves of weights}

As recalled in (\ref{sec:root.stacks}), to define root stacks and parabolic sheaves with finitely generated weights, one considers an injective map of Kummer type $\overline{M} \to B$ with $B$ a coherent  sheaf of monoids (i.e. admitting local charts by finitely generated monoids).
The root stack $\radice[B]{X}$ parametrizes extensions $N\colon B\to \Div_X$ of the DF structure of $X$, and parabolic sheaves are cartesian functors $E\colon B^\wt\to \Qcoh_X$, where $B^\wt$ is the ``category of weights'' associated with $B$: its objects are sections of $B^\gp$, and an arrow $s\to t$ is a section $b$ of $B$ such that $t=s+b$ (see \cite[Section 5]{borne-vistoli}).
Note that since $\overline{M}\to B$ is Kummer, we can see $B$ as a subsheaf of $B_\bQ\cong \overline{M}_\bQ$ (for a monoid $P$, we denote by $P_\bQ$ the positive rational cone spanned by $P$ in $P^\gp\otimes_\bZ \bQ$, and we use the same notation for the analogous construction on sheaves of monoids). In the limit, when we consider the infinite root stack $\irs{X}$, the sheaf $B$ is $\overline{M}_\bQ$ itself.

Here we want to generalize these concepts to the case where we have weights in a sheaf of monoids $\Lambda$ with $\overline{M} \subseteq \Lambda\subseteq \overline{M}_\bR$, where $\overline{M}_\bR=$``$\overline{M} \otimes \bR_+$'' is the positive \emph{real} cone spanned by $\overline{M}$ in $\overline{M}^\gp\otimes_\bZ\bR$. In other words, sections of $\overline{M}_\bR$ are sums of sections of the form $m\otimes r$ in $\overline{M}^\gp\otimes_\bZ\bR$, with $m\in \overline{M}$ and $r\in \bR_+$. Note that $\Lambda$ might well not be finitely generated as a monoid (this already happens for $\overline{M}_\bQ$). The concept of \emph{parabolic sheaves with real weights} that we will define is a generalization of earlier definitions (for example \cite{metha-seshadri, maruyama, iyer-simpson}).

Since finitely presented parabolic sheaves will be defined as the ones obtained by applying an induction functor via a ``fine sub-system of weights'' (see Definition \ref{def:fine.charts.weight} below) that does not have to be a sheaf of monoids, we will discuss sheaves of weights in a more general context, just requiring that they be sheaves of pre-ordered sets, with an action of the sheaf $\overline{M}^\gp$. Part of what follows is taken from a section in a preliminary version of \cite{borne-vistoli}, that was removed in the final version. I am grateful to A. Vistoli and N. Borne for allowing me to include this material here.

Recall that a pre-ordered set is a set $W$ equipped with a reflexive and transitive relation, that we denote by $\leq$ (or $\leq_W$ when we have to specify $W$). Pre-ordered sets form a category $\Preo$, where morphisms $W\to W'$ are increasing functions $f\colon W\to W'$, i.e. such that $f(w)\leq_{W'} f(w')$ if $w\leq_W w'$. Pre-ordered sets can also be seen as small categories with at most one morphism (in each direction) between any two objects.

\begin{definition}
Let $P$ be an integral monoid. A \emph{weight system} for $P$ is a pre-ordered set $W$ with an action of $P^\gp$, that we denote by $(p,w)\mapsto p+w$, such that: (a) if $w\leq w'$, then for every $p\in P^\gp$ we have $p+w\leq p+w'$, and (b) for every $p\in P$ we have $p+w\geq w$.
\end{definition}

Let $\cC$ be a site, that we will later specify to be the small \'etale site of a scheme, or the classical site of a complex analytic space.

\begin{definition}
A \emph{pre-sheaf of pre-ordered sets} on $\cC$ is a functor $W\colon \cC^\op\to \Preo$. A \emph{sheaf of pre-ordered sets} on $\cC$ is a pre-sheaf of pre-ordered set, which is furthermore a sheaf of sets, and such that if $w, w' \in W(U)$ are such that $f_i^*w\leq f_i^*w'$ for a covering $\{f_i\colon U_i\to U\}$ in $\cC$, then $w\leq w'$ in $W(U)$.
\end{definition}

One can sheafify a pre-sheaf of pre-ordered sets to a sheaf of pre-ordered sets in a unique way.

Let now $X$ be a log scheme or log analytic space, with DF structure $L\colon \overline{M}\to \Div_X$.

\begin{definition}
A \emph{pre-weight system} on $X$ is a pre-sheaf of pre-ordered sets $W$, together with an action of $\overline{M}^\gp$, such that for every $U$ the set $W(U)$ is a weight system for the monoid $\overline{M}(U)$. A pre-weight system is a \emph{weight system} if $W$ is a sheaf of pre-ordered sets.
\end{definition}

\begin{example}\label{example:weight system}
Assume that $B$ is a sheaf of integral monoids containing $\overline{M}$. Then there is a natural partial order on $B^\gp$, by declaring that $b\leq b'$ if and only if there exists $c\in B$ such that $b'=b+c$. Moreover, there is an action of $\overline{M}^\gp\subseteq B^\gp$, given by the monoid operation. We will denote the corresponding weight system by $B^\wt$.

If $\overline{M}\to B$ is a Kummer extension, these weight systems $B^\wt$ are the ones appearing in \cite{borne-vistoli}.
\end{example}

\subsection{Diagrams of $\cO$-modules}\label{sec:diagrams}

Let $X$ be a scheme or analytic space, with a symmetric monoidal functor $L\colon P\to \Div(X)$ for an integral monoid $P$. We denote the image of $p\in P$ by $(L_p,s_p)$. Recall that the functor $L$ can be extended to a functor $P^\gp\to \Pic(X)$ (that we continue to denote by $L$), by sending $p-p'$ to the invertible sheaf $L_p\otimes_{\cO_X} L_{p'}^\vee$ (see \cite[Proposition 5.2]{borne-vistoli}). Denote by $\lambda_{a,b}\colon L_{a+b}\cong  L_a\otimes_{\cO_X} L_b$ and $\epsilon\colon L_0\cong \cO_X$ the isomorphisms that are part of the data of the symmetric monoidal functor $L$ (as in \cite[Definition 2.1]{borne-vistoli}).

Let $W$ be a weight system for the monoid $P$. The following is the straightforward adaptation of \cite[Definition 5.6]{borne-vistoli} to this more general setting.

\begin{definition}\label{def:parabolic.sheaf}
A \emph{diagram of} $\cO$\emph{-modules} $(E,j^E)$ for the above data is a functor $E\colon W\to \Mod(\cO_X)$, denoted $w\mapsto E_w$ and $(w\leq w')\mapsto E_{(w,w')}$, together with an isomorphism
$$
\rho^E_{a,w}\colon E_{a+w}\cong L_a\otimes_{\cO_X} E_w
$$
for every $a\in P^\gp$ and $w\in W$, such that
\begin{itemize}
\item[(a)] for every $p\in P$ the diagram
$$
\xymatrix@C=2cm{
E_w\ar[r]^{E_{(w,p+w)}}\ar[d]_\cong & E_{p+w}\ar[d]^{\rho_{p,w}^E}  \\
\cO_X\otimes_{\cO_X} E_w\ar[r]^{s_p\otimes \id} & L_p\otimes_{\cO_X} E_w
}
$$
commutes,
\item[(b)] for every $w\leq w'$ in $W$ and $a\in P^\gp$, then the diagram
$$
\xymatrix@C=2cm{
E_{a+w}\ar[r]^{\rho_{a,w}^E} \ar[d]_{E_{(a+w,a+w')}} & L_a\otimes_{\cO_X} E_w\ar[d]^{\id\otimes E_{(w,w')}} \\
E_{a+w'}\ar[r]^{\rho_{a,w'}^E}  & L_a\otimes_{\cO_X} E_{w'}
}
$$
commutes,
\item[(c)] for every $a,b\in P^\gp$ and $w\in W$, the diagram
$$
\xymatrix@C=2cm{
E_{a+b+w}\ar[r]^{\rho_{a+b,w}^E} \ar[d]_{\rho_{a,b+w}^E} & L_{a+b} \otimes_{\cO_X} E_w \ar[d]^{\lambda_{a,b} \otimes \id } \\
L_a\otimes_{\cO_X} E_{b+w}\ar[r]^{\id\otimes \rho_{b,w}^E}  & L_a\otimes_{\cO_X} L_b\otimes_{\cO_X} E_{w}
}
$$
commutes, and
\item[(d)] for every $w\in W$ the composite 
$$
E_w=E_{0+w}\xrightarrow{\rho_{0,w}^E} L_0\otimes_{\cO_X} E_w\xrightarrow{\epsilon \otimes \id} \cO_X\otimes_{\cO_X} E_w 
$$
coincides with the natural isomorphism $E_w\cong \cO_X\otimes_{\cO_X} E_w$.
\end{itemize}
\end{definition}

A morphism of diagrams of $\cO$-modules is a natural transformation $\phi\colon E\to E'$, such that for every $a\in P^\gp$ and $w\in W$, the diagram
$$
\xymatrix@C=2cm{
E_{a+w}\ar[r]^{\rho_{a,w}^E}\ar[d]_{\phi_{a+w}} & L_a\otimes_{\cO_X} E_w\ar[d]^{\id\otimes \phi_w} \\
E'_{a+w}\ar[r]^{{\rho}_{a,w}^{E'}} & L_a\otimes_{\cO_X} E'_w
}
$$
commutes. Diagrams of $\cO$-modules on $X$ with respect to $W$ form an abelian category $\Mod(X,W)$. The structure of abelian category is defined ``component-wise''.

We say that a diagram of $\cO$-modules $E$ is \emph{quasi-coherent} if the functor $E$ has values in $\Qcoh(X)$. We denote by $\Qcoh(X,W)$ the full sub-category of $\Mod(X,W)$ of quasi-coherent diagrams of $\cO$-modules. 

Assume now that $X$ is a log scheme or log analytic space, with DF structure given by $L\colon \overline{M}\to \Div_X$, and assume that we have a weight system $W$ for $\overline{M}$ on $X$.

\begin{definition}
A \emph{diagram of} $\cO$\emph{-modules} for the above data is a cartesian functor $E\colon W\to \Mod_{\cO_X}$, with an isomorphism of $\cO_U$-modules
$$
\rho^E_{a,w}\colon E_{a+w}\cong L_a \otimes_{\cO_U} E_w
$$
for every open $U\to X$ (either an \'etale morphism, or an open immersion of analytic spaces), $a\in \overline{M}^\gp(U)$ and $w\in W(U)$, such that

\begin{itemize}
\item[a)] for every open $U\to X$, the restriction $E(U)\colon W(U)\to \Mod(\cO_U)$ is a diagram of $\cO$-modules on $U$ with weights in $W(U)$, and
\item[b)] for every arrow $f\colon (U\to X) \to (V\to X)$ between opens of $X$, and for every $a\in \overline{M}^\gp(V)$ and $w\in W(V)$, the isomorphism
$$
\rho^{E}_{f^*a,f^*w}\colon E_{f^*(a+w)}=E_{f^*a+f^*w}\cong L_{f^*a}\otimes_{\cO_U} E_{f^*w}
$$
coincides with the pullback of $\rho^E_{a,w}\colon E_{a+w}\cong L_a\otimes_{\cO_V} E_w$.
\end{itemize}
\end{definition}

Sometimes we will refer to the sheaves $E_w$ as \emph{pieces} of the diagram of $\cO$-modules $E$.
Similarly to the previous case, there is an abelian category $\Mod(X,W)$ of diagrams of $\cO_X$-modules on $X$ with weights in $W$, and a full subcategory $\Qcoh(X,W)\subseteq \Mod(X,W)$ of quasi-coherent diagrams.

These definitions coincide with to the ones of \cite[Section 5.2]{borne-vistoli}, if the weight system is given by a Kummer extension of sheaves of monoids $\overline{M}\to B$.

\subsubsection{Functoriality}\label{sec:functoriality}

Let $X$ be a log scheme and $W,W'$ two weight systems, with an injective $\overline{M}^\gp$-equivariant map $j\colon W\to W'$. We will call such a map an \emph{embedding} of weight systems.

In this situation, we can define two adjoint functors between diagrams of $\cO$-modules, that we call \emph{restriction} $$\res_{W}^{W'}\colon \Mod(X,W')\to \Mod(X,W)$$ and \emph{induction} $$\ind_W^{W'}\colon \Mod(X,W)\to \Mod(X,W').$$
Restriction is simply defined by restricting a diagram of $\cO$-modules $E\colon W'\to \Mod_{\cO_X}$ and the isomorphisms $\rho^E$ along the embedding $W\to W'$. Note that this operation sends quasi-coherent diagrams to quasi-coherent diagrams.

Induction is more complicated to describe. Assume that $E\colon W\to \Mod_{\cO_X}$ is a diagram of $\cO$-modules, and let $U\to X$ be an open, and $w'\in W'(U)$. We want to define a sheaf of $\cO_U$-modules $\wti{E}_{w'}$.

For an open $f\colon V\to U$, consider the subset of $W(V)$ given by 
$$W_{w'}(V)=\{w\in W(V)\mid w\leq_{W'} f^*w'\}$$
with the induced pre-order. We have a functor from $W_{w'}(V)$ to the category of abelian groups, sending $w$ to $E_w(V)$. Define
$$
\wti{E}_{w'}^{\mathrm{pre}}(V)=\varinjlim_{w\in W_{w'}(V)} E_w(V).
$$
Moreover, given a further open $g\colon V'\to V$, we have a morphism of pre-ordered sets $g^*\colon W_{w'}(V)\to W_{w'}(V')$. For every $w\in W_{w'}(V)$ we have an isomorphism $g^*E_w\cong E_{g^*w}$, and these induce a homomorphism
$$
E_w(V)\to E_{g^*w}(V')\to \varinjlim_{w''\in W_{w'}(V')}E_{w''}(V')=\wti{E}_{w'}^{\mathrm{pre}}(V').
$$
By taking the colimit we obtain a homomorphism $\wti{E}_{w'}^{\mathrm{pre}}(V)\to \wti{E}_{w'}^{\mathrm{pre}}(V')$. This makes $\wti{E}_{w'}^{\mathrm{pre}}$ into a pre-sheaf of $\cO_U$-modules. Let $\wti{E}_{w'}$ be the associated sheaf.

We define a cartesian functor $\wti{E}\colon W'\to \Mod_{\cO_X}$ sending $w'$ to $\wti{E}_{w'}$. Note that if $w'\leq w''$ in $W'$, then there is an inclusion of pre-ordered sets $W_{w'}(V)\to W_{w''}(V)$ for every open $V\to U$, and this induces a homomorphism of $\cO_U$-modules $E_{w'}\to E_{w''}$. Since pullbacks respect direct limit and sheafifications, it also follows that the functor is cartesian. Moreover the isomorphisms $\rho^E_{a,w}$ induce isomorphisms $\rho^{\wti{E}}_{a,w'}\colon \wti{E}_{a+w'}\cong L_a\otimes_{\cO_X} \wti{E}_{w'}$ by taking colimits.

It is straightforward now to define the functor $\ind_W^{W'}$ sending $E$ to $\wti{E}$. Moreover, one also easily checks that $\ind_W^{W'}$ is left adjoint to $\res_W^{W'}$, and fully faithful (equivalently, the unit of the adjunction $\id\to \res_W^{W'}\circ \ind_W^{W'}$ is an isomorphism). We record this in the following proposition.

\begin{proposition}\label{prop:res.ind}
The restriction functor $\res_W^{W'}$ has a left adjoint $\ind_W^{W'}$, which is moreover fully faithful. \qed
\end{proposition}

There are obvious (simpler) versions of these constructions for diagrams of $\cO$-modules for a weight system relative to a monoid $P$ and a symmetric monoidal functor $L\colon P\to \Div(X)$.

Using the equivalence between parabolic sheaves and quasi-coherent sheaves on root stacks of \cite{borne-vistoli}, these two functors are identified with pullback and pushforward along the canonical map between the two corresponding root stacks, and this adjunction is the usual one. This is explained for example in \cite[Section 2.2]{talpo} and \cite[Section 7.1]{TV}.

\begin{remark}
While $\res_W^{W'}$ always preserves quasi-coherence, the functor $\ind_W^{W'}$ probably does not, in full generality. One can show that it is true with some additional assumptions on the weight systems, but this fact will not be needed.
\end{remark}

\subsubsection{Local models}
 
We now discuss charts for weight systems. Assume that $X$ is a log scheme or a log analytic space, with DF structure $L\colon \overline{M}\to \Div_X$ with a global chart $P\to \overline{M}(X)$, and that $R$ is a weight system for $P$. Then we obtain an induced weight system $W$ for $\overline{M}$ as follows. Call $K\subseteq \underline{P}_X$ the kernel of the map to $\overline{M}$, and consider the (sheaf) quotient $W=\underline{R}_X/K$. This is a weight system for $\overline{M}$, for the pre-order defined by $w\leq w'$ in $W(U)$ if there exists a covering $\{f_i\colon U_i\to U\}$ and $w_i, w_i' \in \underline{R}_X(U_i)$ such that $w_i\leq w_i'$ for every $i$, and $f_i^*w=w_i$, $f_i^*w'=w_i'$ for every $i$. There is a natural map $R\to W(X)$ of pre-ordered sets.
 
\begin{definition}\label{def:fine.charts.weight}
In the situation we just described, we say that the pair $(P\to \overline{M}(X), R\to W(X))$ is a \emph{chart} for the weight system $W$ on $X$.

A chart is said to be \emph{fine} if $P$ is finitely generated, and $R$ is the union of a finite number of orbits for the action of $P^\gp$.
\end{definition}
 
 Sometimes we will refer to a chart for a weight system just via the morphism $R\to W(X)$.
 
\begin{example}
Consider the weight system associated with the monoid $\frac{1}{n}\overline{M}$. If the DF structure has a global chart $P\to \overline{M}(X)$, this weight system also has a global chart, given by $\frac{1}{n}P\to \frac{1}{n}\overline{M}(X)$. This chart is also fine if $P$ is finitely generated, since the pre-ordered set $\frac{1}{n}P^\wt$ has finitely many orbits with respect to the action of $P^\gp$.

Analogously the weight system associated with $\overline{M}_\bQ$ has a chart given by $P_\bQ\to \overline{M}_\bQ(X)$, but this chart is not fine.
\end{example}

\begin{definition}
A weight system $W$ for a log scheme $X$ is said to be \emph{quasi-coherent} if it locally admits charts. It is said to be \emph{fine} if it locally admits fine charts.
\end{definition}

\subsection{Parabolic sheaves}\label{sec:par.sheaves}

Assume now that $\Lambda$ is a sheaf of monoids on the fine saturated logarithmic scheme $X$, such that $\overline{M}\subseteq \Lambda\subseteq \overline{M}_\bR$, and that it is \emph{quasi-coherent}, i.e. it admits local charts. This means that locally on $X$ there is a chart $P\to \overline{M}(X)$ and a monoid $\Lambda_0$ with $P\subseteq \Lambda_0\subseteq P_\bR$, with a chart $\Lambda_0\to \Lambda(X)$ that makes the obvious diagram commute. One can check that this is equivalent to asking that $\Lambda$ be \emph{log constructible} \cite[3.2]{ogus} (briefly, this means that it is locally constant on the stratification associated to the log structure of $X$). 

We will also always assume that $\Lambda\subseteq \overline{M}_\bR$ is \emph{saturated} for the action of $\overline{M}^\gp$ (or $\overline{M}^\gp$-\emph{saturated}, for short). This will mean the following: if $\lambda\in \overline{M}_\bR$ is such that $\lambda+p\in \Lambda$ for some $p\in \overline{M}$, then $\lambda\in \Lambda$. This condition can also be formulated by considering the projection $\pi\colon \overline{M}_\bR^\gp\to \overline{M}_\bR^\gp/\overline{M}^\gp\cong \overline{M}^\gp\otimes_{\bZ}(\bR/\bZ)$, and requiring that $\Lambda=\pi^{-1}\pi(\Lambda)\cap \overline{M}_\bR$. This condition can be checked on charts for $\overline{M}$ and $\Lambda$.

\begin{example}\label{example:saturated.submonoid}
If the log structure has a chart with $P=\bN$, then we have $P_\bR=\bR_+$, and for every number $c\in \bR_+$ we can consider as $\Lambda_0$ the $\bZ$-saturated submonoid of $\bR_+$ generated by $c$. If $c$ is irrational this is not just the submonoid $\{nc \mid n\in \bN\}$ consisting of positive multiples of $c$ (this does not even contain $\bN$), but is the subset $\{nc+m\mid n\in \bN, \, m\in \bZ\ \text{ and } nc+m\geq 0\}$. Notice that as a monoid this is not finitely generated. In fact, in general if a monoid $\Lambda$ as above is finitely generated, then it is necessarily contained in some $\frac{1}{n}\overline{M}$.
\end{example}

Recall that $\Lambda$ determines a weight system $\Lambda^\wt$ for $\overline{M}^\gp$ (see Example \ref{example:weight system}).

\begin{definition}
A \emph{parabolic sheaf} on $X$ with weights in $\Lambda$ is a diagram of $\cO$-modules for the weight system $\Lambda^\wt$.
\end{definition}

The same definition applies in the presence of a Kato chart $P\to \overline{M}(X)$ and a $P^\gp$-saturated monoid $P\subseteq \Lambda_0\subseteq P_\bR$, and gives a notion of parabolic sheaf on $X$ with weights in $\Lambda_0$.

\begin{proposition}\label{prop:parabolic.chart}
Let $X$ be a fine saturated log scheme, and $\Lambda$ be a $\overline{M}^\gp$-saturated sheaf of monoids with $\overline{M}\subseteq \Lambda\subseteq \overline{M}_\bR$, with a global chart $(P\to \overline{M}(X),\Lambda_0\to \Lambda(X))$. Then there is an equivalence of categories $\iota\colon \Mod(X,\Lambda^\wt)\to \Mod(X,\Lambda_0^\wt)$.
\end{proposition}

\begin{proof}
The functor $\iota$ is defined by restricting $E\colon \Lambda^\wt\to \Mod_{\cO_X}$ and the isomorphisms $\rho_{a,w}^E$ along $\Lambda^\wt_0\to \Lambda^\wt(X)$.

The fact that $\iota$ is an equivalence is proven exactly as in \cite[Proposition 5.10]{borne-vistoli} (the only change is that in \cite[Lemma 5.11]{borne-vistoli} one has to consider a section $l$ such that $k\leq l\leq mk$ for some $m\in \bN$).
\end{proof}

\begin{remark}
In the algebraic case, if one wants to work in the \'etale topology then for the statement of the previous proposition to be true one needs to replace $\Mod$ by $\Qcoh$ (the problem is that $\Mod$ is not a stack for the \'etale topology). Since the main focus for this paper is on the analytic case, we will not worry about this.
\end{remark}

As for the case of schemes and finite systems of weights treated in \cite{borne-vistoli}, we want to restrict to a class of ``quasi-coherent'' parabolic sheaves. One natural choice would be to consider quasi-coherent diagrams of $\cO$-modules as in the definition above, but we will do something a little different. Let us define finitely presented sheaves first.

\begin{definition}\label{def:qcoh.par}
A parabolic sheaf $E$ with weights in $\Lambda$ is \emph{finitely presented} if for all $\lambda\in \Lambda^\wt$ the sheaf $E_\lambda$ is a finitely presented sheaf of $\cO_X$-modules on $X$, and locally on $X$ there exists a fine sub-weight system $R\subseteq \Lambda^\wt$ such that $E$ is in the essential image of the induction functor
$$
\ind_{R}^{\Lambda^\wt}\colon \Mod(X,R)\to \Mod(X,\Lambda^\wt).
$$
\end{definition}

Intuitively, the last condition says that locally the parabolic sheaf is completely determined by finitely many of its pieces $E_\lambda$. It is not hard to check that the induction functors of the kind that appear in the definition above preserve quasi-coherence of the diagrams (as defined in (\ref{sec:diagrams})): using Proposition \ref{prop:parabolic.chart} one can reduce to the case of constant monoids, and then the statement reduces to the fact that a finite colimit of quasi-coherent sheaves is quasi-coherent. Note that this assertion may fail for non-finite colimits in the analytic context (see \cite[Remark 2.1.5]{conrad}). 

\begin{example}
In Example \ref{example:saturated.submonoid}, for every number $c\in \bR_+$ we can consider the weight system $R$ given by the subset $\{c+k\mid k\in \bZ\}$ inside the weight system $\Lambda_0^\wt=\{mc+n\mid m\in \bZ, \, n\in \bZ\}$. This $R$ is a fine weights system, since it consists of a single orbit for the action of $\bZ$.
\end{example}

\begin{example}
Continue to assume that the log structure has a global chart with monoid $\bN$ (so the DF structure is given by a line bundle with a section $(L,s)\in \Div(X)$), and take $\Lambda=\bR$. In this case a parabolic sheaf with weights in $\Lambda$ is the assignment of an $\cO_X$-module $E_r$ for each $r\in \bR$, with maps $E_r\to E_{r'}$ when $r\leq r'$, that are compatible with respect to composition, and such that $E_r\to E_{r+1}\cong L\otimes_{\cO_X} E_r$ is identified with multiplication by the section $s$.

For such a sheaf, being finitely presented means that each $E_r$ is a finitely presented sheaf on $X$, and moreover there exist finitely many real numbers $0\leq r_1 < \cdots < r_k <1$, such that for every $r\in \bR$ the sheaf $E_r$ is obtained in this way: consider the largest integer $\lfloor r\rfloor$ which is $\leq r$, and the fractional part $s=\{r\}=r-\lfloor r\rfloor$; then $E_r\cong E_{r_i}\otimes L_{\lfloor r\rfloor}$ via the given map, where $r_i$ is the largest of the fixed numbers that is $\leq s$. 

In other words, the parabolic sheaf $E$ is completely determined by the weights $r_i$, the finitely presented sheaves $E_{r_i}$, and the maps between them.
\end{example}

\begin{remark}
Note that if $X$ is a noetherian (or more generally coherent) scheme or an analytic space, finitely presented sheaves coincide with coherent sheaves. In the next sections we will work with complex analytic spaces or schemes of finite type over $\bC$, so this comment will apply.
\end{remark}

In the category of parabolic sheaves with weights in $\Lambda$, we can form colimits by taking the colimits ``level-wise''.

\begin{definition}
A parabolic sheaf with weights in $\Lambda$ is \emph{quasi-coherent} if locally on $X$ it can be written as a filtered colimit of finitely presented parabolic sheaves with weights in $\Lambda$.
\end{definition}

\begin{remark}\label{rmk:qcoh}
This definition is inspired by the definition of a quasi-coherent sheaf on an analytic space of \cite{conrad}. We opted to use this notion, instead of the perhaps more natural one requiring that all sheaves $E_\lambda$ are quasi-coherent on $X$, for technical convenience. Note that, if $X$ itself is coherent, for a quasi-coherent sheaf in the sense of the definition it is indeed the case that  $E_\lambda$ is quasi-coherent for every $\lambda$ (it is locally a filtered colimit of finitely presented sheaves, which are also coherent if $\cO_X$ itself is), but it is not clear if the two notions would coincide in general.

See also the discussion about quasi-coherent sheaves on $X_\log$ in (\ref{sec:sheaves.of.modules}).
\end{remark}

We will denote by $\Par(X,\Lambda)$ the category of quasi-coherent parabolic sheaves on $X$ with weights in $\Lambda$, and by $\FP\Par(X,\Lambda)$ the full subcategory of finitely presented parabolic sheaves. Moreover, $\Par(X,\bQ)$ will be a shorthand for the category $\Par(X,\overline{M}_\bQ)$, and $\Par(X,\bR)$ for the category $\Par(X,\overline{M}_\bR))$. Note that it is not clear that these are abelian categories, but we can talk about exactness by embedding these categories into the abelian category $\Mod(X,\Lambda^\wt)$ of diagrams of $\cO$-modules for $\Lambda^\wt$.

To conclude this section we briefly note that, over the complex numbers, there is a version for finitely presented parabolic sheaves of the GAGA equivalence, that relates parabolic sheaves on a proper scheme $X$ over $\bC$ and parabolic sheaves on the associated analytic space. Assume that $X$ is a fine saturated log scheme locally of finite type over the complex numbers. Then the analytification $X_\an$ inherits a fine saturated log structure (on its classical site), and we can compare  parabolic sheaves on the two sides. Let $\Lambda$ be a $\overline{M}^\gp$-saturated quasi-coherent sheaf of monoids on $X$ such that $\overline{M}\subseteq \Lambda\subseteq \overline{M}_\bR$, and denote by $\Lambda_\an$ the induced sheaf on the classical site of $X_\an$.

\begin{proposition}\label{prop:gaga}
There is a natural analytification functor
$$
(-)_\an\colon \FP\Par(X,\Lambda)\to \FP\Par(X_\an,\Lambda_\an)
$$
which is exact. If $X$ is proper, this functor is an equivalence of categories.
\end{proposition}

\begin{proof}
Note that since $X$ is noetherian, the pieces of a finitely presented parabolic sheaf are coherent sheaves. The analytification functor is then defined by analytifying all the pieces and the maps of a parabolic sheaf, and all the assertions follow immediately from the classical GAGA theorems.
\end{proof}

\begin{remark}\label{rmk:analytic.vs.algebraic}
The previous proposition assures that if $X$ is a proper scheme over $\bC$, Theorem \ref{thm:main} applies also to ``algebraic'' parabolic sheaves on $X$ (since in that case they are the same as ``analytic'' ones on $X_\an$).

In general, our main result gives a correspondence between {analytic} parabolic sheaves on $X_\an$, and certain sheaves of modules on $X_\log$. If $X$ is a scheme of finite type over $\bC$ which is not proper, one can still ask if {algebraic} parabolic sheaves with real weights on $X$ correspond to some kind of sheaves on $X_\log$. There should indeed be a variant of the constructions that we will describe starting in the next section, giving a correspondence involving algebraic parabolic sheaves, but for simplicity of exposition we will restrict our treatment to analytic sheaves on $X_\an$ (which is anyway the most natural setting, given the use of the Kato-Nakayama space).
\end{remark}

\subsection{Correspondence with sheaves on root stacks}\label{sec:reminder}

Since our proof of Theorem \ref{thm:main} will follow quite closely the one of \cite[Theorem 6.1]{borne-vistoli}, we give a short reminder about the correspondence with sheaves on root stacks.

Let $X$ be a fine saturated log scheme, and consider a system of denominators $\overline{M}\to B$ (i.e. a Kummer extension of sheaves of monoids, admitting local charts). We are going to sketch the construction of the functor $\Phi\colon \Qcoh(\radice[B]{X})\to \Par(X,B)$, and the proof that it is an equivalence.

Recall that $\pi\colon \radice[B]{X}\to X$ carries a universal DF structure $L^B\colon \pi^{-1}B\to \Div_{\radice[B]{X}}$ that extends the pullback $\pi^{*}L\colon \pi^{-1}A\to \Div_{\radice[B]{X}}$. Given a quasi-coherent sheaf $F \in \Qcoh(\radice[B]{X})$ and $b\in B^\wt(U)$ for $U\to X$ \'etale, set
$$
\Phi(F)_b:=\pi_*(F|_{\pi^{-1}U}\otimes_{\cO_{\radice[B]{U}}} L^B_b).
$$
Note that for $b\leq b'$, i.e. $b'=b+c$ with $c\in B(U)$, we have $L^B_{b'}\cong L^B_{b}\otimes_{\cO_X} L^B_{c}$, and hence we obtain a map $\Phi(F)_b\to \Phi(F)_{b'}$, given by multiplication by the section $s_{c}$ of $L^B_{c}$. The projection formula for $\pi$ provides the isomorphisms $\rho_{m,b}^{\Phi(F)}\colon \Phi(F)_{b+m}\cong L_m\otimes_{\cO_X} \Phi(F)_b$ for $b\in B^\wt(U)$ and $m\in \overline{M}(U)$. Easy verifications show that $\Phi(F)$ is a parabolic sheaf on $X$ with weights in $B$.

To prove that this functor is an equivalence, since both categories extend to stacks for the \'etale topology of $X$, one can localize where there is a Kato chart $X\to \Spec \bZ[P]$ and a chart $(Q\to B(X), P\to Q)$ for $\overline{M}\to B$, and construct a quasi-inverse locally. Recall from (\ref{sec:root.stacks}) that in the presence of such charts, quasi-coherent sheaves on $\radice[B]{X}$ can be identified with quasi-coherent sheaves of $\cO_X[P^\gp]\otimes_{\bZ[P]}\bZ[Q]$-modules on $X$ that have a $Q^\gp$-grading compatible with the module structure.

Given a parabolic sheaf $E$ with weights in $Q$, one forms the sheaf $\bigoplus_{q\in Q^\gp} E_q$ on $X$. This has a structure of $\cO_X[P^\gp]\otimes_{\bZ[P]}\bZ[Q]$-module, determined by the maps $E_q\to E_{q'}$ for $q\leq q'$ that are part of the definition of a parabolic sheaf, and a $Q^\gp$-grading that is compatible with the module structure. Hence we obtain a quasi-coherent sheaf on $\radice[B]{X}$, and this gives the desired quasi-inverse. One can also show that if $X$ is noetherian, coherent sheaves on $\radice[B]{X}$ correspond to parabolic sheaves such that each $E_b$ is coherent.

\section{Sheaves on the Kato-Nakayama space}\label{sec:sheaves.on.log}

From now on $X$ will be a fine saturated complex analytic space (which might for example be the analytification of a fine saturated log scheme locally of finite type over $\bC$). We will denote by $\tau\colon X_\log\to X$ the Kato-Nakayama space of $X$ with its natural projection.

\subsection{Indexed algebras}\label{sec:3.1}

Assume for the moment that $(T,\cO_T)$ is an arbitrary ringed space.

We start by describing a construction of a sheaf of $\cO_T$-algebras $\cA$ associated with an extension
\begin{equation}\label{eq:extension}
\xymatrix{
0\ar[r] & \cO_T^\times\ar[r] & M\ar[r]^\pi & \overline{M}\ar[r] & 0
}
\end{equation}
of sheaves of monoids on $T$, which could be associated with a log structure in the case where $T$ is the underlying space of a complex analytic space (but we will also apply this procedure to exact sequences on the Kato-Nakayama space). By an extension of sheaves of monoids, we mean a pair of maps $f\colon P'\to P$ and $g\colon P\to P''$, such that $f$ is an isomorphism onto the submonoid $g^{-1}(0)$ of $P$, and $g$ induces an isomorphism $P/{P'}\cong P/f^{-1}(0) \cong P''$. The following construction is taken from a paper of Lorenzon \cite{lorenzon}, via the work of Ogus \cite{ogus}.

For an open $U\subseteq T$ and a section $a\in \overline{M}(U)$, the sheaf of preimages of $a$ in $M$ is an $\cO_U^\times$-torsor that we denote by $P_a$. This corresponds to a line bundle $N_a$ (i.e. an invertible sheaf of $\cO_U$-modules) on $U$, given by the contracted product $P_a\times^{\cO^\times_U}\cO_U$, where $\cO_U^\times\to \cO_U$ is the natural inclusion.
We define $\cA(U)=\oplus_{a\in \overline{M}(U)}N_a(U)$ as an $\cO_T(U)$-module. The natural restriction maps give a sheaf of $\cO_T$-modules $\cA$. To ease notation, we will succinctly write $\cA=\bigoplus_{a\in \overline{M}}N_a$. Moreover if $a,b \in \overline{M}(U)$, we have a natural map $N_a\otimes_{\cO_T} N_b\to N_{a+b}$, and this gives $\cA$ the structure of a sheaf of $\cO_T$-algebras. There is also a natural morphism of monoids $M\to \cA$ (where the operation on $\cA$ is multiplication), since a section $m$ of $M$ trivializes the torsor $P_{\pi(m)}$.

Assume now that we have a homomorphism of sheaves of monoids $M\to \cO_T$ (where $\cO_T$ is equipped with multiplication), such that the composite $\cO_T^\times\to M\to \cO_T$ coincides with the canonical inclusion. Then every $N_a$ has a canonical co-section  $t_a\colon N_a\to \cO_T$, induced by the map $P_a\subseteq M\to \cO_T$.

\begin{example}\label{ex:not.qc}
Assume we are considering the extension (\ref{eq:extension}) associated with the natural log structure given by the origin on $X=\bA^1=(\Spec \bC[z])_\an$. In this case the sheaf $\cA$ is as follows. If $U\subseteq \bA^1$ does not contain the origin, then $\cA|_U\cong \cO_U$, since in this case the restriction of the extension is trivial (i.e. $\overline{M}=0$). If $U$ does contain the origin, then $\Gamma(U,\overline{M})=\bN$, and $\cA(U)\cong \bigoplus_{n\in \bN} t^n\cdot \cO_X(U)$, where $t^n$ is just a placeholder variable. 

The map $N_a\otimes_{\cO_T} N_b\to N_{a+b}$ is given by the natural isomorphism $$(t^n \cdot \cO_{\bA^1})\otimes_{\cO_{\bA^1}} (t^m\cdot \cO_{\bA^1})\cong t^{n+m}\cdot \cO_{\bA^1},$$
and the co-section $N_a\to \cO_X$ is determined by $t^n\cdot \cO_{\bA^1}\to \cO_{\bA^1}$ sending $t^n\cdot 1$ to $z^n$. A similar description can be given for affine toric varieties $X=\bC(P)=(\Spec \bC[P])_\an$ with the natural log structure.

In particular note that the sheaf $\cA$ is not quasi-coherent, even in the algebraic case (and this is in fact typical). If it were quasi-coherent, it would be the sheaf associated to the $\bC[z]$-module $\Gamma(\bA^1,\cA)=\bigoplus_{n\in \bN} t^n\cdot \bC[z]$, but this is clearly incorrect, since the restriction  of $\cA$ to $\bA^1\setminus \{0\}$ coincides with the structure sheaf $\cO_{\bA^1\setminus \{0\}}$.
\end{example}

Denote by $\Div_{(T,\cO_T)}$ the symmetric monoidal category over $T$ of pairs $(L,s)$ consisting of an $\cO_T$-line bundle $L$ (i.e. a locally free sheaf of $\cO_T$-modules of rank $1$) with a section $s$. From the extension (\ref{eq:extension}) and the previous construction we also obtain a symmetric monoidal functor $\overline{M}\to \Div_{(T,\cO_T)}$ by sending $a\in \overline{M}(U)$ to the dual $L_a=N_a^\vee$ of the line bundle $N_a$ associated to the torsor $P_a$, together with the section $s_a\colon \cO_T\to L_a$ induced by the co-section $t_a$.

\begin{definition}
If $X$ is a log analytic space, and extension (\ref{eq:extension}) comes from the log structure, we will denote the associated sheaf of $\cO_X$-algebras by $\cA_X$.
\end{definition}

\subsection{Extensions on $X_\log$}\label{sec:extensions}

Let $X$ be a fine saturated log analytic space. We will explain how to produce various extensions of the form (\ref{eq:extension}) on the Kato-Nakayama space $X_\log$, besides the one coming from the log structure of $X$. We will use these extensions to produce sheaves of rings $\cO_\Lambda$ on $X_\log$ for a quasi-coherent sheaf of submonoids $\Lambda\subseteq \overline{M}_\bR=\overline{M}\otimes \bR_+$ containing $\overline{M}$. Morally, the sheaf $\cO_\Lambda$ will be generated by the pullback of $\cO_X$ and the products of powers ${m_i}^{\alpha_i}$ where $m_i$ are sections of $\overline{M}$ and $\alpha_i\in \bR_+$, such that $\sum_i m_i \otimes \alpha_i \in \Lambda \subseteq \overline{M}_\bR$ (see the description in (\ref{sec:local.description})).

Recall that the universal object parametrized by the topological space $X_\log$ with the projection $\tau\colon X_\log\to X$ is a homomorphism $c\colon \tau^{-1}M^\gp\to \bS^1_{X_\log}$ of sheaves of abelian groups of $X_\log$, such that $c(f)=f/|f|$ for $f\in \tau^{-1}\cO^\times$. Recall that if $T$ is a topological space and $G$ is a topological monoid, or group, etc. we denote by $G_T$ the sheaf of continuous functions towards $G$ on opens of $T$ (with the induced structure of a sheaf of monoids, groups, etc.).

Consider the sheaf of abelian groups $\cL=\tau^{-1}M^\gp\times_{\bS^1_{X_\log}}  i\bR_{X_\log}$ on $X_\log$, where $i\bR_{X_\log}\to \bS^1_{X_\log}$ is given by the exponential. This sits in a commutative diagram with exact rows 
$$
\xymatrix{
0\ar[r] & 2\pi i \bZ \ar[r]\ar[d]_= &\ar[d]\ar[r] \cL &\ar[r] \tau^{-1}M^\gp\ar[d]^c &0 \\
0 \ar[r] & 2\pi i \bZ \ar[r] & i\bR_{X_\log} \ar[r] &\ar[r]\bS^1_{X_\log} &0.
}
$$
In other words, sections of $\cL$ over an open $U$ are pairs consisting of a section $m$ of $\tau^{-1}M^\gp$ and a continuous function $\theta \colon U\to i\bR$, such that $c(m)=e^{\theta}$ as functions $U\to \bS^1$. If we think of $c$ as assigning a \emph{phase} to every section of $\tau^{-1}M^\gp$ that is not in $\tau^{-1}\cO^\times_X$, then $\cL$ records also the choice of an \emph{angle}, i.e. a pre-image in $i\bR$ of the phase. In this sense, $\cL$ is a sheaf of ``logarithms'' of sections of $M^\gp$. The structure sheaf $\cO_X^\log$ of $X_\log$ is constructed by formally adjoining to $\cO_X$ the sections of the sheaf $\cL$ (see (\ref{sec:rings.on.kn}) for details).

\begin{example}
Recall that the \emph{standard log point} $(\Spec \bC,\bN)$ is the log analytic space given by the analytic space $\Spec \bC$ (i.e. a reduced point), with monoid $\bC^\times\oplus \bN$ and morphism $\bC^\times\oplus \bN\to \bC=\cO_{\Spec \bC}$ described by $(a,n)\mapsto a\cdot 0^n$, where $0^0=1$. In general we denote by $(\Spec \bC,P)$ the log point with monoid $P$, defined in the analogous manner.

The Kato-Nakayama space of the standard log point is $\Hom(\bN,\bS^1)\cong \bS^1$, and the sheaf $\cL$ can be described as follows: on the universal cover $\pi\colon \bR\to \bS^1$, consider the constant sheaf $\underline{\bZ\oplus \bC}_\bR$, and make the group of deck transformations $\bZ$ act on this sheaf by $k\cdot (k', c)=(k', c+2kk'\pi i)$ (so that the sheaf acquires an equivariant structure). The result of descent to $\bS^1$ is precisely $\cL$. The map $2\pi i\bZ\to \pi^{-1}\cL$ can be described as $2\pi i k\mapsto (0,2\pi i k)$, and $\pi^{-1}\cL\to \pi^{-1}\tau^{-1}M^\gp\cong \underline{\bZ\oplus \bC^\times}_{\bR}$ is given by $(k,c)\mapsto (k,\exp(c))$.

\end{example}

There is an injective homomorphism of sheaves of abelian groups $\tau^{-1}\cO_X\to \cL$ defined on an open subset $U$ by sending a holomorphic function $f\in \tau^{-1}\cO_X$ to the pair $(\exp(f),i\, \Im(f))$, 
where $\exp(f)$ is seen as a section of $\tau^{-1}\cO_X^\times\subseteq \tau^{-1}M^\gp$. Moreover, we also have a homomorphism $\cL\to \tau^{-1}\overline{M}^\gp$ given by composing the first projection $\cL\to \tau^{-1}{M}^\gp$ with the quotient map $\tau^{-1}{M}^\gp\to \tau^{-1}\overline{M}^\gp$.

These maps fit into a short exact sequence
\begin{equation}\label{eq:extension2}
\xymatrix{
0\ar[r] & \tau^{-1}\cO_X \ar[r] &\ar[r] \cL &\ar[r] \tau^{-1}\overline{M}^\gp  &0
}
\end{equation}
of sheaves of abelian groups on $X_\log$: if a section $(m,\theta)$ of $\cL$ maps to zero in $\tau^{-1}\overline{M}^\gp$, then $m$ is a section of $\tau^{-1}\cO_X^\times$, and there is a unique ``logarithm'' for this pair, i.e. a section $f\in \tau^{-1}\cO_X$ such that $(\exp(f),i\,\Im(f))=(m,\theta)$.

The following construction is taken from \cite[Section 3.3]{ogus}. Let us tensor (\ref{eq:extension2}) by the constant sheaf $\bC$ (over $\bZ$ - we will omit this from the notation). We obtain
$$
\xymatrix{
0\ar[r] & \tau^{-1}\cO_X\otimes \underline{\bC}_{X_\log} \ar[r] &\ar[r] \cL \otimes \underline{\bC}_{X_\log}&\ar[r] \tau^{-1}\overline{M}^\gp \otimes \underline{\bC}_{X_\log}&0.
}
$$
Now let us consider the map $\tau^{-1}\cO_X\otimes \underline{\bC}_{X_\log}\to \tau^{-1}\cO_X^\times$ defined on generators as $f\otimes c\mapsto e^{c\cdot f}$, and the induced diagram
$$
\xymatrix{
0\ar[r] & \tau^{-1}\cO_X\otimes \underline{\bC}_{X_\log} \ar[r]\ar[d] &\ar[r]\ar[d] \cL \otimes \underline{\bC}_{X_\log}&\ar[r] \ar[d]^{=}\tau^{-1}\overline{M}^\gp \otimes \underline{\bC}_{X_\log}&0\\
0\ar[r] & \tau^{-1}\cO_X^\times \ar[r] &\ar[r] M_\log^\gp &\ar[r] \tau^{-1}\overline{M}^\gp \otimes \underline{\bC}_{X_\log}&0
}
$$
where $M_\log^\gp$ is the pushout of the diagram to its left.

Finally, given a quasi-coherent sheaf of monoids $\Lambda$ on $X$, with $\Lambda \subseteq \overline{M}_\bR=\overline{M}\otimes \bR_+ \subseteq \overline{M}^\gp\otimes \underline{\bC}_{X_\log}$, we can pullback the bottom extension of the last diagram to an extension of sheaves of monoids on $X_\log$
\begin{equation}\label{eq:extension3}
\xymatrix{
0\ar[r] & \tau^{-1}\cO_X^\times \ar[r] &\ar[r] M_{\Lambda} &\ar[r]\tau^{-1}\Lambda &0.
}
\end{equation}

\begin{remark}
In \cite{ogus}, the symbol $\Lambda$ is used in this same context to denote a log constructible sheaf of abelian groups in $\overline{M}^\gp\otimes \underline{\bC}_{X_\log}$, and not a sheaf of monoids.
\end{remark}

\begin{definition}
We will denote by $\cA_\Lambda$ the sheaf of $\tau^{-1}\cO_X$-algebras on $X_\log$, associated with the extension (\ref{eq:extension3}), by the procedure outlined in (\ref{sec:3.1}).
\end{definition}

It is clear that the construction of these extensions, as well as the objects that we are going to describe next, are compatible with strict base change.

\subsection{Sheaves of rings on $X_\log$}\label{sec:rings.on.kn}

Now consider $\Lambda=\overline{M}$, so that $\cA_{\Lambda}=\tau^{-1}\cA_X$, and note that there is a natural homomorphism of $\tau^{-1}\cO_X$-algebras $\tau^{-1}\cA_X\to \tau^{-1}\cO_X$, which is the pullback of a homomorphism $\cA_X\to \cO_X$ on $X$: each homogeneous piece $N_a$ of $\cA_X$ has a morphism of $\cO_X$-modules into $\cO_X$ given by the co-section $s_a\colon N_a\to \cO_X$, and the resulting map $\cA_X\to \cO_X$ is a homomorphism of algebras.

Moreover for every $\Lambda$ we have a natural homomorphism $\tau^{-1}\cA_X\to \cA_{\Lambda}$. We set $$\cO_{\Lambda}:=\cA_{\Lambda}\otimes_{\tau^{-1}\cA_X}\tau^{-1}\cO_X.$$ This is a sheaf of rings on $X_\log$, together with an injective map $\tau^{-1}\cO_X\to \cO_{\Lambda}$.

As anticipated above, $\cO_\Lambda$ should be loosely thought of as $\tau^{-1}\cO_X[\prod_i m_i^{\alpha_i}]$, where $\sum_i m_i\otimes \alpha_i$ are the sections of $\Lambda\subseteq \overline{M}\otimes \bR_+$, and the obvious relations are satisfied, for example, $m_i^1$ is identified with a corresponding local section $f_i\in \tau^{-1}\cO_X$. See Remark \ref{rmk:sheaves.description} below for a more precise statement, in a particular case.

\begin{example}\label{ex:st.log.pt}
Assume that $X$ is the standard log point $(\Spec \bC,\bN)$, and let $\bN \subseteq \Lambda\subseteq \bR_+$ be a $\bZ$-saturated monoid. The algebra $\cA_X$ in this case can be described as the $\bC$-algebra $\bigoplus_{n\in \bN} t^n\cdot \bC$. The morphism $\cA_X\to \cO_X=\bC$ sends $t^0$ to $1$, and $t^n$ to $0$ for $n>0$.

The Kato-Nakayama space $X_\log$ is isomorphic to $\bS^1$, and the sheaf $\tau^{-1}\cA_X$ is the constant sheaf $\underline{\bigoplus_{n\in \bN} t^n\cdot \bC}_{\bS^1}$. As for $\cA_\Lambda$, we have $\cA_\Lambda=\bigoplus_{\lambda\in \Lambda} N_\lambda$ for $\tau^{-1}\bC=\underline{\bC}_{\bS^1}$-invertible sheaves $N_\lambda$.

For $\lambda\notin \bN$, this line bundle will have non-trivial monodromy with respect to the action of the fundamental group $\pi_1(\bS^1)\cong \bZ$, and that can be described as follows. Consider the universal cover $\pi\colon \bR\to \bS^1$, and denote the composite $\bR\to \bS^1\to \Spec \bC$ by $\wti{\tau}$. For every $\lambda$, the pullback $\pi^{-1}N_\lambda$ is locally constant on $\bR$, hence it is constant, $N_\lambda\cong \underline{\bC}_\bR$. Let us formally write $t^\lambda$ for a generator, so that $N_\lambda=t^\lambda\cdot \underline{\bC}_\bR$. The group $\bZ$ of deck transformations of $\pi$ acts on this sheaf (in the sense that the sheaf has a $\bZ$-equivariant structure), by $k\cdot (t^\lambda  c)=e^{2\pi i k \lambda}t^\lambda c$. 
With this notation, we can write $\pi^{-1}\cA_\Lambda=\bigoplus_{\lambda\in \Lambda} t^\lambda\cdot \underline{\bC}_\bR$.

Furthermore, recall that $\cO_\Lambda=\cA_\Lambda\otimes_{\tau^{-1}\cA_X}\tau^{-1}\cO_X$. This has the effect (on the universal cover $\bR$) of identifying $t^n$ with its image in $\underline{\bC}_\bR=\wti{\tau}^{-1}\cO_X$, in the description above. Hence, if $\lambda>1$ and since $\Lambda$ is $\bZ$-saturated, this forces the image of $t^\lambda$ to be $0$, since $t^\lambda=t^{\lambda-1}\cdot t$, and $t$ maps to $0$ in $\bC$. Hence we have $$\pi^{-1}\cO_\Lambda=\bigoplus_{\lambda\in \Lambda\cap [0,1)} S^\lambda\cdot \underline{\bC}_\bR,$$ where multiplication is determined by $S^\lambda\cdot S^{\lambda'}=S^{\lambda+\lambda'}$ if $\lambda+\lambda'<1$, and $0$ otherwise.

Here $S^\lambda$ should be thought of as ``$\,z^\lambda\,$'', where $z$ is the coordinate of the $\bA^1$, of which the standard log point is the origin (i.e. the generator of $\overline{M}=\bN$). 
\end{example}

One of the main points of this paper is that the ringed space $(X_\log,\cO_\Lambda)$ can be seen as a sort of root stack of $X$ with respect to coefficients in $\Lambda$, for any given $\overline{M}^\gp$-saturated quasi-coherent sheaf of monoids $\overline{M}\subseteq \Lambda\subseteq \overline{M}_\bR$. In fact, there is a canonical homomorphism of sheaves of monoids $\alpha_\Lambda\colon M_\Lambda\to \cO_\Lambda$:
as explained in \cite[I.2.3]{lorenzon}, a section $m\in M_\Lambda$ determines a trivialization of the torsor $P_{\overline{m}}$ of preimages in $M_\Lambda$ of $\overline{m}$ (which denotes the image of $m$ in $\tau^{-1}\Lambda$), and consequently a section $e_m \in N_{\overline{m}}\subseteq \cA_\Lambda$ of the associated line bundle. We set $\alpha_\Lambda(m)$ to be the image of $e_m$ in $\cO_\Lambda=\cA_{\Lambda}\otimes_{\tau^{-1}\cA_X}\tau^{-1}\cO_X$.
The map $\alpha_\Lambda$ is not properly a log structure, since the units in $\cO_\Lambda$ are in general bigger than the units in $\tau^{-1}\cO_X$ (already for example for the standard log point), but the induced homomorphism $\mathbf{M}_\Lambda=M_\Lambda\oplus_{\tau^{-1}\cO_X^\times}\cO_\Lambda^\times \to \cO_\Lambda$ is a log structure.

From the following extension (derived from extension (\ref{eq:extension3}))
$$
\xymatrix{
0\ar[r] & \cO_\Lambda^\times \ar[r] &\ar[r] \mathbf{M}_{\Lambda} &\ar[r]\tau^{-1}\Lambda &0
}
$$
together with the homomorphism $\mathbf{M}_\Lambda\to \cO_\Lambda$, as described in (\ref{sec:3.1}) we obtain a symmetric monoidal functor $L^\Lambda\colon \tau^{-1}\Lambda\to \Div_{(X_\log,\cO_\Lambda)}$ from $\tau^{-1}\Lambda$ to the stack over opens of $X_\log$ of $\cO_\Lambda$-invertible sheaves with a global section. If $N_\lambda$ is the invertible sheaf of $\tau^{-1}\cO_X$-modules associated with $\lambda\in \tau^{-1}\Lambda$ via the extension (\ref{eq:extension3}) above, then the invertible sheaf of $\cO_\Lambda$-modules associated with $\lambda$ via the last extension is canonically isomorphic to $N_\lambda\otimes_{\tau^{-1}\cO_X} \cO_\Lambda$, and hence the corresponding sheaf of $\cO_\Lambda$-algebras is simply $$\cA_\Lambda\otimes_{\tau^{-1}\cO_X} \cO_\Lambda=\bigoplus_{\lambda\in \tau^{-1}\Lambda}(N_\lambda\otimes_{\tau^{-1}\cO_X} \cO_\Lambda).$$
These data give a  log structure on $X_\log$ that extends the one of $X$ by adjoining ``real powers'' of the sections of $\overline{M}$ to the structure sheaf. The invertible sheaves $L^\Lambda(\lambda)=(L^\Lambda_\lambda, s_\lambda)\in  \Div_{(X_\log,\cO_\Lambda)}$ for sections $\lambda \in \tau^{-1}\Lambda$, duals of the sheaves $N_\lambda\otimes_{\tau^{-1}\cO_X} \cO_\Lambda$ mentioned above, will be fundamental for the correspondence with parabolic sheaves (strictly speaking, after tensoring them with $\cO_X^\log$ - see below). If the log structure is divisorial, morally $s_\lambda$ should be thought of as $f^\lambda$, where $f$ is a local equation of a branch of the boundary divisor, and the sheaf $L^\Lambda_\lambda$ is to be thought of as $f^{-\lambda}\cdot \cO_\Lambda$.

\begin{example}
If for example $X=\bA^1$ and $\bN\subseteq \Lambda\subseteq \bR_+$, then $\cA_\Lambda$ is isomorphic to $\tau^{-1}\cO_X$ outside $\tau^{-1}(0)$, and $\cA_\Lambda(U)=\bigoplus_{\lambda \in \Lambda} N_\lambda(U)$ for $U$ intersecting $\tau^{-1}(0)$.

Here the sheaf $N_n$ for $n\in \bN$ can be seen as the pullback to $X_\log$ of $t^n\cdot \cO_X$, i.e. of the sheaf $\cO_X$ for $n=0$, and the sheaf defined by $(t^n\cdot \cO_X)(U):=\cO_X(U)$ if $U$ contains the origin and $0$ otherwise, for $n\neq 0$ (so that in this case $t^n$ is ``supported at the origin''). In general $N_\lambda$ can be described, on the universal cover $\bH=\bR_{\geq 0}\times \bR\to \bR_{\geq 0}\times \bS^1=X_\log$, as the sheaf $t^\lambda\cdot \cO_X$ (where $t^\lambda$ is likewise supported on the preimage of the origin, for $\lambda\neq 0$), with the $\bZ$-equivariant structure determined by $k\cdot t^\lambda=e^{2\pi i k \lambda}t^\lambda$.

The sheaf $\cA_\Lambda$ is the direct sum of these line bundles, and the sheaf $\cO_\Lambda=\cA_\Lambda\otimes_{\tau^{-1}\cA_X}\tau^{-1}\cO_X$ can be described on $\bH$ by adjoining to $\cO_X$ sections $S^\lambda$ for $\lambda\in \Lambda$ (supported on the preimage of the origin), on which $\bZ$ acts as above, and with multiplication defined so that $S^\lambda=z^n\cdot S^{\lambda'}$ if $\lambda=n+\lambda'$ and $n \in \bN$. Note that the restriction of this description to the origin in $\bA^1$ recovers the discussion of Example \ref{ex:st.log.pt}.
\end{example}

For reasons that will be explained later (see Remark \ref{rmk:exact}), we need to also tensor the sheaves $\cO_\Lambda$ and the line bundles $L_\lambda^\Lambda$ with the ``structure sheaf'' $\cO_X^\log$ of $X_\log$. We briefly recall the construction of this sheaf; see \cite[Section 3]{KN}, \cite[Section 1]{illusie-kato-nakayama} or \cite[Section 3.3]{ogus} for more details.

The sheaf $\cO_X^\log$ is the universal sheaf of $\tau^{-1}\cO_X$-algebras with a compatible morphism of sheaves of abelian groups $\cL\to \cO_X^\log$, where $\cL$ is the sheaf of abelian groups of (\ref{sec:extensions}). An explicit construction is as the quotient
$$
\cO_X^\log=(\tau^{-1}\cO_X\otimes_\bZ \Sym_{\bZ}\cL)/\cI
$$
where $\cI$ is the sheaf of ideals generated by local sections of the form $f\otimes 1 -1\otimes h(f)$ for $f\in \tau^{-1}\cO_X$, and where $h\colon \tau^{-1}\cO_X\to \cL$ is the map in the exact sequence (\ref{eq:extension2}), defined by
$$h(f)=(\exp(f), i \, \Im(f))\in \cL\cong \tau^{-1}M^\gp \times_{\bS^1_{X_\log}} i\bR_{X_\log}.$$
The stalks of $\cO_X^\log$ can be described as follows: let $y\in X_\log$ be a point with image $x=\tau(y)$, and let $m_1,\hdots, m_r$ be elements of $\cL_y$, whose image under $\cL_y\to (\tau^{-1}\overline{M}^\gp)_y\cong \overline{M}^\gp_x$ is a $\bZ$-basis. Then there is an $\cO_{X,x}$-linear isomorphism $\cO_{X,x}[T_1,\hdots, T_r]\cong \cO^\log_{X,y}$ given by $T_i\mapsto {m_i}$, where we are slightly abusing notation in denoting the image of $m_i$ in $\cO^\log_{X,y}$ by the same symbol. Hence, morally $\cO_X^\log$ should be thought of as the sheaf $\tau^{-1}\cO_X[\log(\overline{m_1}), \hdots, \log(\overline{m_r})]$, where $\overline{m_1},\hdots, \overline{m_r}$ is a local basis of $\overline{M}^\gp$. 

\begin{example}\label{ex:st.log.log}
Assume that $X$ is the standard log point $X=(\Spec \bC, \bN)$. The sheaf $\cO_X^\log$ on the Kato-Nakayama space $\bS^1$ can be described as follows. As usual take the universal cover $\wti{X}_\log\cong \bR\to \bS^1$, and consider the constant sheaf of $\bC$-algebras $\underline{\bC[T]}_\bR$, equipped with the $\bZ$-equivariant structure determined by $k\cdot T=T-2\pi i k$. The result of descent to $X_\log=\bS^1$ is the sheaf $\cO_X^\log$.

In other words, for every point $y\in \bS^1$ we have $\cO^\log_{X,y}\cong \bC[T]$, but by moving around the circle once, $T$ becomes $T-2\pi i$ (this makes sense if we think of $T$ as ``$\,\log \,z\,$'', where $z$ is the coordinate of $\bA^1$, and the standard log point is the origin). For a detailed explanation of the ``minus'' sign in this action, we refer the reader to \cite[Appendix A1]{kato-usui}.

Note that the global sections of $\cO_X^\log$ are only the constants, i.e. $\tau_*\cO_X^\log\cong \cO_X$ in this case. This is true also in general (Proposition \ref{lemma:push.structure}).
\end{example}

\begin{remark}
As mentioned in the introduction of \cite{TVnew}, it is natural to ask if the map of topological stacks $\phi_\infty\colon X_\log\to \radice[\infty]{X}_\top$ (whose construction is recalled briefly in (\ref{sec:kn})) can be promoted in a natural way to a morphism of ringed topological stacks, where we are equipping $X_\log$ with the sheaf $\cO_X^\log$, and $\radice[\infty]{X}_\top$ with its structure sheaf $\cO_\infty$. The idea here would be the $\cO_X^\log$ has logarithms of local sections of $\overline{M}$, so the $n$-th roots of such sections in the sheaf $\cO_\infty$ should have an image in $\cO_X^\log$.

It turns that it is hard to make sense of this if we want a map of rings: if we want to define $\cO_\infty\to \cO_X^\log$ as a sort of logarithm, by sending $t^{\frac{1}{n}}$ to $\frac{1}{n}\log(t)$, then it will not be a homomorphism of rings. A second possibility would be to pass to convergent power series in $\cO_X^\log$, and try to impose that $\exp(\frac{1}{n}\log(t))=\text{ image of } t^\frac{1}{n}$. This also makes little sense in some situations, because an exponential had better be invertible, but $t^{\frac{1}{n}}$ sometimes is nilpotent, for example if $X$ is the standard log point.

Instead of trying to make sense of this, the solution that we adopt in this paper is to adjoin all needed roots to $\cO_X^\log$ (as is done in \cite[Section 4]{illusie-kato-nakayama}, for example), in order to have a map as above.
\end{remark}

We will set $$\cO_\Lambda^\log:=\cO_\Lambda\otimes_{\tau^{-1}\cO_X} \cO_X^\log.$$ This is again a sheaf of rings on $X_\log$, with an injective homomorphism $\tau^{-1}\cO_X\to\cO_\Lambda^\log$. Morally, on top of adding every possible real power $m^\lambda$ of sections of $\overline{M}$ and exponents in $\Lambda$, we are also adding formal logarithms $\log(m)$.

By tensoring with $\cO_X^\log$ we can lift the symmetric monoidal functor $L^\Lambda\colon \tau^{-1}\Lambda\to \Div_{(X_\log,\cO_\Lambda)}$ to a symmetric monoidal functor $\tau^{-1}\Lambda\to \Div_{(X_\log,\cO_\Lambda^\log)}$, that we will keep denoting by the same symbol. The line bundle associated with $\lambda\in \tau^{-1}\Lambda$ via this new symmetric monoidal functor is simply $L^\Lambda_\lambda \otimes_{\tau^{-1}\cO_X}\cO_X^\log$ (and will be denoted again by $L_\lambda^\Lambda$ - we will make no use of the invertible sheaf before tensoring with $\cO_X^\log$).

The functor $L^\Lambda\colon \tau^{-1}\Lambda\to \Div_{(X_\log,\cO_\Lambda^\log)}$ extends the symmetric monoidal functor $L\colon \overline{M}\to \Div_X$ on $X$, so in particular for $\lambda=m\in \overline{M}$ we have $L_m^\Lambda\cong \tau^{-1}L_m\otimes_{\tau^{-1}\cO_X}\cO_\Lambda^\log$ (and the global sections are also identified).
Moreover, as in \cite[Proposition 5.2]{borne-vistoli}, we get an induced symmetric monoidal functor $\tau^{-1}\Lambda^\gp\to \Pic_{(X_\log,\cO_\Lambda^\log)}$, that we will also denote by $L^\Lambda$, by setting $L^\Lambda_{\lambda-\mu}=L^\Lambda_\lambda\otimes_{\cO_\Lambda^\log} (L^\Lambda_\mu)^\vee$.

\begin{remark}
Some version of the construction of $\cO_\Lambda^\log$ has appeared in \cite{illusie-kato-nakayama}. In (\ref{sec:comparison.root}) below we point out that the sheaf of rings $\cO_X^\klog$ that is used in that paper, and is obtained using the Kummer-\'etale site of $X$, is canonically isomorphic to our $\cO^\log_{\overline{M}_\bQ}$. In \cite{ogus}, Ogus uses larger sheaves of rings $\cO_X^\log\otimes_{\tau^{-1}\cO_X} \cA_{\Lambda^\gp}$, that are related to our $\cO_\Lambda^\log$, but not exactly the same.
\end{remark}

\subsection{Local description}\label{sec:local.description}

We will need a local description of some of the constructions that we described up to this point.

Let us suppose that $X$ has a Kato chart $X\to \bC(P)$. Note first of all that with this assumption, every line bundle $L_a$ coming from the DF structure $L\colon P\to \Div(X)$ of $X$ is canonically trivialized. For the log analytic space $\bC(P)$ with its natural log structure we have  $\bC(P)_\log=(\bR_{\geq 0}\times \bS^1)(P)=\Hom(P,\bR_{\geq 0}\times \bS^1)$, and we will also use the universal cover $\wti{\bC(P)}_\log=\bH(P)=\Hom(P,\bH)$ (recall that $\bH=\bR_{\geq 0}\times \bR$). The morphism $\wti{\bC(P)}_\log\to \bC(P)_\log$ is induced by the map $\bH\to \bR_{\geq 0}\times \bS^1$ given by $(x,y)\mapsto (x,e^{iy})$, and it is a $\bZ(P)=\Hom(P,\bZ)$-principal bundle. Note that $\Hom(P,\bZ)=\Hom(P^\gp,\bZ)\cong \bZ^r$ non-canonically, where $r$ is the rank of the free abelian group $P^\gp$.

As for the analytic space $X$, there is a diagram with cartesian squares
$$
\xymatrix{
\widetilde{X}_\log\ar[r]\ar[d] & \bH(P)\ar[d]\\
X_\log\ar[r]\ar[d] & (\bR_{\geq 0}\times \bS^1)(P)\ar[d]\\
X\ar[r] & \bC(P)
}
$$
where $\wti{X}_\log\to X_\log$ is a $\bZ(P)=\Hom(P,\bZ)$-covering space for the action induced on $\wti{X}_\log=X\times_{\bC(P)} \bH(P)$ by the one on the second factor. We will need to consider the analogous constructions of the sheaves $\cA_\Lambda$ and $\cO_\Lambda$ and $\cO_\Lambda^\log$ on the space $\wti{X}_\log$: in this case we will add a tilde to remind ourselves that we are on $\wti{X}_\log$ rather than on $X_\log$. We will denote by $\pi\colon \wti{X}_\log\to X_\log$ the natural map, and by $\wti{\tau}$ the composite $\tau\circ \pi$.

Note that, since the bottom map of the last diagram is strict, every extension of the form
$$
\xymatrix{
0\ar[r] &  \tau^{-1}\cO_X^\times \ar[r] & M_\Lambda \ar[r] & \tau^{-1}\Lambda \ar[r] & 0
}
$$
where $\Lambda$ is pulled back from a quasi-coherent sheaf of monoids on $\bC(P)$, is also pulled back from the analogous extension on $\bC(P)_\log$. The same is true of the sheaves $\cA_\Lambda, \wti{\cA}_\Lambda, \cO_\Lambda, \wti{\cO}_\Lambda, \cO_\Lambda^\log, \wti{\cO}_\Lambda^\log$ and of the Deligne-Faltings structure giving the sheaves $L_\lambda^\Lambda$ and $\wti{L}_\lambda^\Lambda$ (both before and after tensoring with $\cO_X^\log$).

According to this description of $X_\log$ as the quotient of $\widetilde{X}_\log$ for the free action of $\bZ(P)$, we will describe sheaves and maps between sheaves on $X_\log$ as objects on $\wti{X}_\log$ that are $\bZ(P)$-equivariant. 

Assume now that $\overline{M}\subseteq \Lambda\subseteq \overline{M}_\bR$ is a $\overline{M}^\gp$-saturated quasi-coherent sheaf of monoids, together with a global chart $\Lambda_0\to \Lambda(X)$, with $P\subseteq \Lambda_0\subseteq P_\bR$. Let us describe the sheaf $\wti{\cO}_\Lambda$ in terms of $\Lambda_0$.

\begin{notation}\label{notation:variables}
In order to avoid confusion, in this situation and for $p \in P$ we will denote
\begin{itemize}
\item by $x^p$ the element of $\bC[P]$, image of $p$ via $P\to \bC[P]$,
\item by $f_p$ the section of $\cO_X$, image of $p$ via $P\to \cO_X(X)$,
\item by $t^p$ the ``placeholder variable'' in the sheaf $\cA_P=\bigoplus_{p\in P}t^p\cdot \cO_X$ on $X$, and
\item by $s_p$ the global section of the invertible sheaf $L_p$, image of $p$ via $P\to \Div(X)$.
\end{itemize}
\end{notation}

Consider the sheaf $\wti{\cA}_{\Lambda_0}:=\wti{\cA}_P \otimes_{\bC[P]} \bC[\Lambda_0]$, where $\bC[P]\to \wti{\cA}_P$ sends $x^p$ to the element $t^p$ of $\wti{\cA}_P=\bigoplus_{p\in P}t^p\cdot \wti{\tau}^{-1}{\cO}_X$. We also set
$$\wti{\cO}_{\Lambda_0}:=\wti{\cA}_{\Lambda_0}\otimes_{\wti{\cA}_P} \wti{\tau}^{-1}{\cO}_X=\wti{\tau}^{-1}{\cO}_X\otimes_{\bC[P]} \bC[\Lambda_0],$$
where $\wti{\cA}_P\to \wti{\tau}^{-1}{\cO}_X$ sends $t^p$ the the section $f_p\in \wti{\tau}^{-1}{\cO}_X$, and correspondingly $\bC[P]\to \wti{\tau}^{-1}{\cO}_X$ sends $x^p$ to $f_p$. Note that the group $\bZ(P)$ acts on these sheaves, by acting on $\bC[\Lambda_0]$ via $g\cdot t^\lambda=e^{2\pi i g_\bR(\lambda)} t^\lambda$, where here and from now on we denote by $g_\bR\colon P^\gp \otimes_\bZ \bR \to \bR$ the natural linear extension of the homomorphism $g\colon P^\gp\to \bZ$. Here $t^\lambda$ should be read as $e^{2\pi i \lambda\cdot \log(t)}$, and the action of $g$ is by ``translation'' on $\log(t)$.

There are surjective homomorphisms $\wti{\cA}_{\Lambda_0}\to \wti{\cA}_\Lambda$ and $\wti{\cO}_{\Lambda_0}\to \wti{\cO}_\Lambda$, whose kernel is the ideal generated by sections of the form ``$\, t^\lambda-1\,$'' (interpreted in the obvious way in the two sheaves) with $\lambda$ a local section in the kernel of $(\Lambda_0)_X \to \Lambda$.
This gives an explicit description for the sheaves $\wti{\cA}_\Lambda$ and $\wti{\cO}_\Lambda$, similar to the one we obtained above for $\wti{\cA}_{\Lambda_0}$ and $\wti{\cO}_{\Lambda_0}$, where the monoid $\Lambda_0$ is replaced by the sheaf $\Lambda$. Examples \ref{examp:sheaves.stdd.log.pt} and \ref{examp:sheaves.std.log.pt} below give completely explicit descriptions of these sheaves over log points.

\begin{remark}\label{rmk:sheaves.description}
If $X=\bC(P)$ and $\Lambda$ has a chart $P\subseteq \Lambda_0\subseteq P_\bR$, we can describe $\cO_\Lambda$ on $X_\log$ even more concretely, in the style of \cite[(1.1) and (3.2)]{illusie-kato-nakayama}. 

Let $U=(\Spec \bC[P^\gp])_\an\subseteq X$, and note that this embedding lifts to $j\colon U\subseteq X_\log$. Moreover the constant sheaf $\underline{P^\gp}_{X_\log}$ can be seen as a subsheaf of $j_*\cO_U^\times$. Then the sheaf $\cO_\Lambda$ can be identified with the subsheaf of rings of $j_*\cO_U$, generated by $\tau^{-1}\cO_X$ and by local sections of the form $\prod_i p_i^{\alpha_i}$, for $\sum_i p_i\otimes \alpha_i \in \Lambda_0\subseteq P\otimes \bR_+$,
and where the $p_i$ are seen as sections of $\underline{P^\gp}_{X_\log}\subseteq j_*\cO_U^\times$. Here by $p^\alpha$ with $p\in P$ and $\alpha \in \bR_+$ we mean the following: by identifying $p$ as a section of  $j_*\cO_U^\times$, if we are on a small enough open subset we can choose a logarithm, i.e. a function $\log(p) \in j_*\cO_U$ such that $\exp(\log(p))=p$ for the usual exponential map $j_*\cO_U\to j_*\cO_U^\times$. Then $p^\alpha:=\exp(\alpha\cdot\log(p))\in j_*\cO_U$. Of course different choices of $\log(p)$ will give different sections ``of the form $p^\alpha$'', that are related by the action of the monodromy around the boundary.

In the same manner, the sheaf $\cO_\Lambda^\log$ is obtained by adding, on top of the previous sections, also local logarithms $\log(p)$ for sections of $\underline{P^\gp}_{X_\log}$. For $\Lambda=\overline{M}_\bQ$, this coincides with the description of $\cO_X^\klog$ in \cite[3.2]{illusie-kato-nakayama}.

\end{remark}

We can also describe the invertible sheaf $\wti{L}_{\lambda}^\Lambda$ on $\wti{X}_\log$ (before tensoring with $\cO_X^\log$): this is an invertible sheaf of $\wti{\cO}_\Lambda$-modules, which is globally trivial, and we will think about it as $\wti{L}_{\lambda}^\Lambda=t^{-{\lambda}}\cdot \wti{\cO}_\Lambda$. This sheaf also has a natural $\bZ(P)$-equivariant structure, a ``shifted'' version of the one of $\wti{\cO}_\Lambda$: for a section $t^{-\lambda}\cdot s$, an element $g\in \bZ(P)$ acts by $g\cdot (t^{-\lambda}\cdot s)=e^{-2\pi i g_\bR(\lambda)}(t^{-\lambda}\cdot (g\cdot s))$, where $g$ acts on $s\in \wti{\cO}_\Lambda$ as explained above. The sheaf on $X_\log$ obtained by descent is $L^\Lambda_{\lambda}$. The global section $\wti{s}_{\lambda}$, given by the natural map $\wti{\cO}_\Lambda\to \wti{L}_{\lambda}^\Lambda$ (that can be seen as multiplication by $t^\lambda \in \wti{\cO}_\Lambda$), also descends to the global section $s_{\lambda}$ of $L_{\lambda}^\Lambda$. As already noted above, as sheaves of $\wti{\cO}_\Lambda$-modules, we have $\wti{L}_{\lambda}^\Lambda \cong \wti{\cO}_\Lambda$, but the action of $\bZ(P)$ is different, unless $\overline{\lambda}=0$ in $\Lambda^\gp/\overline{M}^\gp$, i.e. $\lambda\in \overline{M}^\gp$. In fact, observe also that if $\lambda$ is a section $m$ of $\overline{M}^\gp$, then $\wti{L}^\Lambda_{\lambda}=t^{-m}\cdot \wti{\cO}_\Lambda\cong \wti{\tau}^{-1}L_m\otimes_{\wti{\tau}^{-1}\cO_X}\wti{\cO}_\Lambda$, where $L\colon \overline{M}^\gp\to \Pic_X$ is the functor associated with the Deligne-Faltings structure of the log analytic space $X$. 

Now let us bring the sheaf $\cO_X^\log$ into the picture. On $\wti{X}_\log$ we can tensor the sheaf $\wti{\cO}_\Lambda$ and the various line bundles $\wti{L}_\lambda^\Lambda$ with $\wti{\cO}_X^\log$ over $\wti{\tau}^{-1}\cO_X$ to obtain $\wti{\cO}_\Lambda^\log$ and the $\wti{\cO}_\Lambda^\log$-line bundles $\wti{L}_\lambda^\Lambda\otimes_{\wti{\tau}^{-1}\cO_X} \wti{\cO}_X^\log$ (along with the induced global sections), that, as in the previous section, we will continue to denote by $\wti{L}_\lambda^\Lambda$.

The local descriptions of these sheaves are obtained from the ones described above, by tensoring with the sheaf $\wti{\cO}_X^\log$ on $\wti{X}_\log$. A description of this latter sheaf is given in \cite[Lemma 3.3.4]{ogus}: if $\cI$ denotes the sheaf of ideals in $\wti{\tau}^{-1}\cO_X\otimes_\bZ \Sym_\bZ P^\gp$ generated by elements of the form $f_p\otimes 1-1\otimes p$ for $p$ a local section of $P$ that maps to a unit in $M$ (via the chart morphism $\underline{P}\to M$), then we have an isomorphism
$$
\wti{\cO}_X^\log\cong (\wti{\tau}^{-1}\cO_X\otimes_\bZ \Sym_\bZ P^\gp)/\cI.
$$
The resulting sheaves have an induced $\bZ(P)$-action, and the results of descent to $X_\log$ are the sheaf $\cO_\Lambda^\log$ and the line bundles $L_\lambda^\Lambda$.

To conclude this discussion, we give a completely explicit description of these sheaves on fibers of the map $\wti{X}_\log\to X$ (or equivalently over log points). These will be important for a few proofs later on.

\begin{example}\label{examp:sheaves.stdd.log.pt}
Let us give an explicit description of the sheaf $\wti{\cO}_X^\log$ on the fibers of $\wti{\tau}\colon \wti{X}_\log\to X$ (generalizing Example \ref{ex:st.log.log}). Fix $x\in X$, and call $P=\overline{M}_x$. Then we can find a Kato chart with monoid $P$ for $X$ around $x$.

Taking the fibers over $x$, we can write $(X_\log)_x\cong \Hom(P^\gp,\bS^1)$ and $(\wti{X}_\log)_x\cong \Hom(P^\gp, \bR)$. If we fix an isomorphism $P^\gp\cong \bZ^r$, we are looking at the universal cover $\bR^r\to (\bS^1)^r$. The sheaf $\wti{\cO}_X^\log$ on $(\wti{X}_\log)_x=\bR^r$ is the constant sheaf $\underline{\bC[T_1,\hdots, T_r]}_{\bR^r}$, and it has a $\bZ^r$-equivariant structure, where the $i$-th standard generator $g_i$ acts on $T_j$ by sending it to $T_j-2\pi i \delta_{ij}$, where $\delta_{ij}$ is the Kronecker delta. By descending this to $(X_\log)_x=(\bS^1)^r$, we obtain the sheaf $\cO_X^\log$.

\end{example}

\begin{example}\label{examp:sheaves.std.log.pt}
Generalizing the previous example, let us also describe the sheaf $\wti{\cO}_\Lambda^\log$ and the line bundles $\wti{L}_\lambda^\Lambda$ on the fibers of $\wti{X}_\log\to X$. By base changing via $x\to X$, we can assume that $X$ is  a log point $X=(\Spec \bC, P)$, where as in the previous example $P=\overline{M}_x$.

Choose elements $p_1,\hdots, p_r\in P$ forming a basis of $P^\gp$ as a $\bZ$-module. The pullback $\wti{L}_\lambda^\Lambda$ to $\wti{X}_\log$ is the constant sheaf on $\wti{X}_\log$ associated with $${t^{-\lambda}\cdot \bC[T_1,\hdots, T_r, S^{\lambda'} \mid \lambda'\in \Lambda^+]/(S^{\lambda'+\lambda''}-S^{\lambda'}\cdot S^{\lambda''} \,, \, S^p \mid \lambda', \lambda''\in \Lambda^+, p\in P^+)}$$
where $T_i$ are independent variables (for $\lambda=0$ this gives a description of the sheaf of rings $\wti{\cO}_\Lambda^\log$). Note that $S^{\lambda'} \cdot S^{\lambda''}=0$ if $\lambda'+\lambda'' \in \langle P^+ \rangle \subseteq \Lambda$, where recall that $P^+=P\setminus \{0\}$ and $\langle - \rangle$ denotes the generated ideal in $\Lambda$. In particular every $S^{\lambda'}$ is nilpotent, since $n \lambda' \in \langle P^+ \rangle$ for $n$ big enough.

This sheaf has a natural $\bZ(P)=\Hom(P,\bZ)$-equivariant structure: denote by $g_i \in \bZ(P)$ the dual of the chosen basis element $p_i$ (i.e. the only element of $\bZ(P)$ such that $g_i(p_j)=\delta_{ij}$), then
\begin{align*}
g_i\cdot t^{-\lambda} & =e^{2\pi i (g_i)_\bR(-\lambda)} t^{-\lambda} \\
g_i \cdot t^{-\lambda}T_j& =e^{2\pi i (g_i)_\bR(-\lambda)} t^{-\lambda}(T_j-2\pi i \delta_{ij}) \\ 
g_i \cdot t^{-\lambda} S^{\lambda'}& =e^{2\pi i (g_i)_\bR(\lambda'-\lambda)}t^{-\lambda}S^{\lambda'}
\end{align*}
where $\delta_{ij}$ is the Kronecker delta, and $(g_i)_\bR$ denotes as usual the extension of $g_i\colon P^\gp\to \bZ$ to a linear map $(g_i)_\bR\colon P^\gp\otimes_\bZ\bR\to \bR$.

The result of descent to $X_\log$ is the line bundle $L_\lambda^\Lambda$ (which is therefore a locally constant sheaf). We can loosely write
$$
L_\lambda^\Lambda= t^{-\lambda}\cdot  \bC[\log(p_i), \prod_j q_j^{\alpha_{j}}]
$$
where $\log(p_i)$ and $\prod_j q_j^{\alpha_{j}}$ for $q_j\in P$ and $\alpha_j \in \bR_+$ denote local sections that are the result of descent respectively of the sections $T_i$ and of the sections $S^\lambda$, with $\lambda=\sum_j q_j\otimes \alpha_{j}\in \Lambda\subseteq P_\bR$  over some open subset of $X_\log$. These sections are the restrictions of the local sections of Remark \ref{rmk:sheaves.description} (the log point $X=(\Spec \bC, P)$ can be identified with the ``vertex'' of the affine toric variety $\Spec \bC[P]$).

We will need a slight variant of this description, where we have a point $\Spec \bC$, and this is regarded as a locally ringed space via a local ring $A$, and equipped with a log structure $\beta\colon A^\times\oplus P\to A$, where $P$ is a toric monoid. The Kato-Nakayama space $X_\log$ in this case topologically is the same as the one of the log point $(\Spec \bC,P)$, but all the sheaves on it are also tensored with $A$ over $\bC$. 

In this situation we have an analogous description of the line bundle $\wti{L}_\lambda^\Lambda$ on $\wti{X}_\log$ as the constant sheaf associated with the $\wti{\cO}_\Lambda^\log$-module 
$$
t^{-\lambda}\cdot A[T_1,\hdots, T_r, S^{\lambda'} \mid \lambda'\in \Lambda^+]/(S^{\lambda'+\lambda''}-S^{\lambda'}\cdot S^{\lambda''} \,, \, S^p-f_p \mid \lambda', \lambda''\in \Lambda^+, p\in P^+)
$$
but notice that this time if $\lambda'+\lambda''=p+\mu$ with $\lambda', \lambda'', \mu\in \Lambda^+$ and $p \in P^+$, then $S^{\lambda'}\cdot S^{\lambda''}=f_p \cdot S^\mu$ (recall from Notation \ref{notation:variables} that $f_p$ denotes the image of $p\in P$ in $A$ via $\beta$).
By descending on $X_\log$, we can loosely write 
$$
L_\lambda^\Lambda= t^{-\lambda}\cdot  A[\log(p_i), \prod_j q_j^{\alpha_{j}}]
$$
where again here $\log(p_i)$ and $ \prod_j q_j^{\alpha_{j}}$ denote local sections.
\end{example}

\subsection{Sheaves of modules on $X_\log$}\label{sec:sheaves.of.modules}

The last ingredient that we need in order to discuss the correspondence with parabolic sheaves is a discussion of quasi-coherent sheaves of modules on $X_\log$ and their properties.

In general if $(T,\cO_T)$ is a ringed space, the usual meaning for ``quasi-coherent sheaf of $\cO_T$-modules'' refers to the existence of local presentations
$$
\xymatrix{
 \cO_T^{\oplus J} \ar[r] &  \cO_T^{\oplus I} \ar[r] & F\ar[r] & 0
}
$$
with possibly infinite index sets $I,J$. In complete generality, it is not clear how well-behaved the category of such sheaves is.

We will instead adopt the terminology of \cite[Section 2.1]{conrad} (as for parabolic sheaves, see Remark \ref{rmk:qcoh}). For quasi-coherence of sheaves of $\cO_\Lambda^\log$-modules on $X_\log$, we will use the following definition. As usual we denote by $\tau\colon X_\log\to X$ the projection, and $\Lambda$ is a $\overline{M}^\gp$-saturated quasi-coherent sheaf of monoids with $\overline{M}\subseteq \Lambda\subseteq \overline{M}_\bR$.

\begin{definition}\label{def:qcoh}
We will say that a sheaf $F$ of $\cO_\Lambda^\log$-modules on $X_\log$ is \emph{finitely presented} if locally on $X$ (i.e. locally on $X_\log$ for open sets of the form $\tau^{-1}U$, with $U\subseteq X$ open) it admits a presentation of the form
$$
\xymatrix{
\bigoplus_{j} L_{\lambda_j}^\Lambda \ar[r] & \bigoplus_{i} L_{\lambda_i}^\Lambda \ar[r] & F\ar[r] & 0
}
$$
for which the index sets for $i$ and $j$ are finite.

We will say that a sheaf $F$ of $\cO_\Lambda^\log$-modules on $X_\log$ is \emph{quasi-coherent} if locally on $X$ it can be written as filtered colimit of finitely presented sheaves.
\end{definition}

Here the sheaves $L_\lambda^\Lambda$ are the line bundles on $X_\log$ of (\ref{sec:rings.on.kn}).

\begin{remark}
We refrain from using the term ``coherent'' for the sheaves that locally admit finite presentations, because already on the infinite root stack, the structure sheaf might not be coherent (see \cite[Example 4.17]{TV}), so that ``finitely presented'' and ``coherent'' are not equivalent notions. We expect the same to happen in this context. Note that, since $\cO_X$ is coherent, on $X$ itself it is indeed true that finite presentation and coherence are equivalent.
\end{remark}

We will denote the category of finitely presented sheaves of  $\cO_\Lambda^\log$-modules by $\FP_\Lambda(X)$, and the category of quasi-coherent sheaves by $\Qcoh_\Lambda(X)$.

\begin{remark}
Some comments about the definition above are in order. 

First of all, note that the line bundles $L_{\lambda}^\Lambda$ are locally isomorphic to $\cO_\Lambda^\log$ on $X_\log$, so a finitely presented sheaf as we defined it will also be finitely presented in the ``standard'' sense (of admitting local presentations as a cokernel of a map between free sheaves of finite rank) on the ringed space $(X_\log,\cO_\Lambda^\log)$. In view of the correspondence with parabolic sheaves, though, we want to restrict to the class that admit local presentations \emph{on opens pulled back from} $X$. Once we choose to do this, using direct sums of the non-trivial line bundles $L_\lambda^\Lambda$ is forced.

This condition on having presentations for a  topology of $X_\log$ that is coarser than the natural one should be compared with the situation of root stacks: the map $\sqrt[n]{X}\to X$ is a homeomorphism on the associated topological spaces, and even though one can localize in the \'etale topology around points of $\sqrt[n]{X}$ where there are non-trivial stabilizers, one can not ``physically'' localize on the fibers (since they are single points!), as one can do on the Kato-Nakayama space.
\end{remark}

We will talk about exact sequences of sheaves in $\Qcoh_\Lambda(X)$ or $\FP_\Lambda(X)$, meaning that the same sequence is exact when viewed in $\Mod(\cO_\Lambda^\log)$.

Let us consider pullback and pushforward along the morphism of ringed spaces $\tau\colon (X_\log,\cO_\Lambda^\log) \to X$. We will omit the sheaf of weights $\Lambda$ from the notation of those functors, since there will be no risk of confusion.

We can define pullback $\tau^*\colon \Mod(\cO_X) \to \Mod(\cO_\Lambda^\log)$ and pushforward $\tau_*\colon \Mod(\cO_\Lambda^\log)\to \Mod(\cO_X)$ as usual. Since $\tau^*\cO_X\cong \cO_\Lambda^\log$ and $\tau^*$ commutes with colimits and is right exact, it is also clear that the pullback functor will restrict nicely to the subcategories of quasi-coherent and finitely presented sheaves, inducing functors $\tau^*\colon \Qcoh(X)\to \Qcoh_\Lambda(X)$ and $\tau^*\colon \FP(X)\to \FP_\Lambda(X)$. It is less clear that the pushforward will behave well. This is what the rest of this section will be about.

The main point here will be showing that the functor $\tau_*\colon \Qcoh_\Lambda(X)\to \Mod(\cO_X)$ is exact. Note that by standard arguments we can define a derived pushforward functor $R\tau_*\colon D^+(\Mod(\cO_\Lambda^\log))\to D^+(\Mod(\cO_X))$.

\begin{proposition}\label{prop:push.exact}
The functor $\tau_*\colon \Qcoh_\Lambda(X)\to \Mod(\cO_X)$ is exact.
\end{proposition}

\begin{proof}
As usual the pushforward is left exact. We will show that for every quasi-coherent sheaf $F\in \Qcoh_\Lambda(X)$ we have $R^1\tau_*F=0$ (where the derived functor is computed in $\Mod(\cO_\Lambda^\log)$), and this will imply the exactness. Since $\tau$ is proper and the spaces are locally compact, $R\tau_*$ commutes with filtered colimits (because it coincides with $R\tau_!$ and therefore is a left adjoint - see for example \cite[Section 3.1]{kashiwara-shapira}), and hence we can assume that $F$ is finitely presented.

In order to show $R^1\tau_*F=0$, let us fix a point $x\in X$ and check that the stalk $(R^1\tau_*F)_x$ is zero. By proper base change (as formulated for example in \cite[Appendix A2]{kato-usui}) via the cartesian diagram
$$
\xymatrix{
(X_\log)_x\cong (x)_\log\ar[r]^>>>>>i\ar[d] & X_\log\ar[d]\\
x\ar[r] & X
}
$$
we have an isomorphism $(R^1\tau_*F)_x=H^1((x)_\log, i^{-1}F)$ (notice the $i^{-1}$, instead of $i^*$). We can therefore assume that $X$ is the log locally ringed space given by the point $\Spec \bC$, but equipped with the local ring $A=\cO_{X,x}$ instead of $\bC$, and with log structure $\beta\colon A^\times \oplus P\to A$ where $P\to A$ is obtained from a chart of the log structure around $x$, with $P=\overline{M}_x$. The space $(x)_\log$ in this case is the Kato-Nakayama space of the log point $(\Spec \bC, P)$, but the corresponding sheaves living on it are also tensored with the local ring $A$, which acts as ``ring of coefficients'' in place of the base field $\bC$.

We will need the following lemma.

\begin{lemma} \label{lemma:loc.free.vanishing.coho}
Let $X$ be the point $\Spec \bC$, equipped with a local ring $A$ and a log structure $\beta\colon A^\times\oplus P\to A$ where $P$ is a toric monoid. Then for every $\lambda \in \Lambda^\gp$ and $m>0$, we have
$H^m(X_\log,L_\lambda^\Lambda)=0$.
\end{lemma}

\begin{proof}
This was proven in \cite[Proposition 3.7]{illusie-kato-nakayama}, \cite[Proposition 4.6]{matsubara} and \cite[Proposition 2.2.10]{kato-usui} in the case where $X$ is a fine saturated log analytic space, $\Lambda=\overline{M}$ and $\lambda=0$, so that there are no ``added roots'', and $L_\lambda^\Lambda=\cO_X^\log$. 

The proof we give is along the same lines of the one of \cite[Proposition 2.2.10]{kato-usui}. Call $p_i\in P$ elements that give a $\bZ$-basis of $P^\gp$. Recall the description of Example \ref{examp:sheaves.std.log.pt}: we can write
$$
L_\lambda^\Lambda \cong t^{-\lambda}\cdot  A[\log(p_i), \prod_j q_j^{\alpha_{j}}]
$$
where both $\log(p_i)$ and $\prod_j q_j^{\alpha_{j}}$ denote local sections. Here the $\prod _j q_j^{\alpha_{j}}$ appearing are exactly the ones for which $\sum_j q_j\otimes \alpha_{j}\in \Lambda^+\setminus \langle P^+\rangle \subseteq P_\bR$, and if $\sum_j q_j\otimes \alpha_j+ \sum_k q_k \otimes {\alpha_{k}}=p+\sum_l q_l\otimes {\alpha_{l}}$ with $p\in P$ then the product $\prod_j q_j^{\alpha_{j}}\cdot \prod_k q_k^{\alpha_{k}}$ is equal to  $f_p \cdot \prod_l q_l^{\alpha_{l}}$ (where $f_p\in A$ is the image of $p$ via $\beta$). Since this sheaf is locally constant on $X_\log$, we have
$$
H^m(X_\log, L_\lambda^\Lambda)=H^m(\bZ(P), (L_\lambda^\Lambda)_y)
$$
where $y$ is a fixed point of $X_\log$, and the right term is group cohomology, with respect to the action of the fundamental group $\bZ(P)=\Hom(P,\bZ)$ of $X_\log$.

Now $(L_\lambda^\Lambda)_y$, as an abelian group, is a direct sum of copies of
$$
t^{-\lambda}\cdot  \prod_j q_j^{\alpha_{j}}\cdot A[\log(p_1),\hdots, \log(p_r)]\cong  A[T_1,\hdots, T_r]
$$
where the action of the generator $g_a\in \bZ(P)$ dual to the element $p_a\in P$ is given by $$g_a(T_{b})=e^{2\pi i (g_a)_\bR(\sum_j q_j\otimes \alpha_{j}-\lambda)} (T_{b}-2\pi i\delta_{a{b}})$$ where $\delta_{ab}$ is the Kronecker delta. Consequently, $H^m(\bZ(P), (L_\lambda^\Lambda)_y)$ is the direct sum of the cohomologies of these subgroups, and it suffices to show that these are all zero for $m>0$.

For $0\leq l\leq r$, let $\Gamma_l$ be the subgroup of $\bZ(P)$ generated by the elements $g_k$ for $l<k\leq r$, and let $R_l\subseteq A[T_1,\hdots, T_r]$ be the submodule $A[T_1,\hdots, T_l]$. We prove by descending induction on $l$ that $R\Gamma(\Gamma_l, A[T_1,\hdots, T_r])$ is either equal to $R_l$ in degree $0$ (with the induced action of the quotient group $\bZ(P)/\Gamma_l=\langle g_1,\hdots, g_l\rangle$) or to the zero complex. For $l=0$ this will give that $R\Gamma(\bZ(P), A[T_1,\hdots, T_r])$ is concentrated in degree zero, which finishes the proof.

For $j=r$ the statement is obvious, since $\Gamma_r$ is the trivial group. Now for the inductive step, assume that $R\Gamma(\Gamma_l, A[T_1,\hdots, T_r])=R_l$ or $0$. Then $$R\Gamma(\Gamma_{l-1},A[T_1,\hdots, T_r])=R\Gamma(\Gamma_{l-1}/\Gamma_{l}, R\Gamma(\Gamma_l, A[T_1,\hdots, T_r]))$$ so if $R\Gamma(\Gamma_l, A[T_1,\hdots, T_r])=0$ we are done. Otherwise, by inductive assumption and the fact that the quotient $\Gamma_{l-1}/\Gamma_{l}$ is cyclic generated by $g_l$, the right-hand side coincides with the complex $(g_l-1)\colon R_l\to R_l$. It is not hard to check that this map is surjective, and its kernel is either $R_{l-1}$ if $(g_l)_\bR(\sum_j q_j\otimes \alpha_{j}-\lambda)\in \bZ$, or otherwise $0$, as required.
\end{proof}

Back to the proof of Proposition \ref{prop:push.exact}, assuming that $X$ is the point $\Spec \bC$ equipped with the local ring $A=\cO_{X,x}$ and log structure $\beta\colon A^\times\oplus P\to A$, we have to check that if $F$ is a finitely presented sheaf of $\cO_\Lambda^\log$-modules on $X_\log$, then $H^1(X_\log, F)=0$. Since $F$ is finitely presented, we have a presentation
$$
\xymatrix{
\bigoplus_j L_{\lambda_j}^\Lambda \ar[r] & \bigoplus_i L_{\lambda_i}^\Lambda \ar[r]^f & F\ar[r] & 0,
}
$$
with finitely many summands on the whole $X_\log$ (the only non-empty open subset of $\Spec \bC$ is the whole space), and by Lemma \ref{lemma:loc.free.vanishing.coho} we have $H^i(X_\log,L_\lambda^\Lambda)=0$ for every $\lambda$ and $i>0$.

Let us consider the kernel $F_1$ of the map $f$ in the presentation above, fitting in a short exact sequence
$$
\xymatrix{
0\ar[r] & F_1  \ar[r] & \bigoplus_i L_{\lambda_i}^\Lambda \ar[r]^f & F\ar[r] & 0,
}
$$
and the induced long exact sequence in cohomology. Using the fact that $H^i(X_\log,L_\lambda^\Lambda)=0$ for $i>0$ and any $\lambda$, we see that $H^i(X_\log, F)\cong H^{i+1}(X_\log, F_1)$ for $i>0$. Now we claim that $F_1$ also has a presentation of the form
$$
\xymatrix{
 \bigoplus_k L_{\lambda_k}^\Lambda  \ar[r] & \bigoplus_j L_{\lambda_j}^\Lambda \ar[r] & F_1\ar[r] & 0,
}
$$
(where $\bigoplus_k L_{\lambda_k}^\Lambda$ might have infinitely many summands). This will allow us to iterate this process. Our objective therefore is now to prove the following lemma.

\begin{lemma}\label{lemma:surjection}
In the situation described above, the kernel of any morphism of sheaves $\bigoplus_j L_{\lambda_j}^\Lambda \to \bigoplus_i L_{\lambda_i}^\Lambda$ on $X_\log$ admits a surjection from some $\bigoplus_k L_{\lambda_k}^\Lambda$ (the index sets for $i,j$ and $k$ need not be finite).
\end{lemma}

\begin{proof}
Let us preliminarily note that 
\begin{equation}\label{eq:homs}
\Hom(L_\lambda^\Lambda,L_\mu^\Lambda)=\Hom(\cO_\Lambda^\log, L_{\mu-\lambda}^\Lambda)=\Gamma( L_{\mu-\lambda}^\Lambda)
\end{equation}
where $\Hom$ denotes the group of homomorphisms of sheaves of $\cO_\Lambda^\log$-modules on $X_\log$. Then, from the description in Section \ref{sec:local.description}, we see that $\Gamma( L_{\mu-\lambda}^\Lambda)$ is trivial unless the difference $\mu-\lambda\in \overline{M}^\gp=P^\gp\subseteq \Lambda^\gp$ (so that $g_\bR(\mu-\lambda)\in \bZ$ for every $g\in \bZ(P)=\Hom(P,\bZ)$). In this case, we have $\Gamma( L_{\mu-\lambda}^\Lambda)\cong \Gamma(X,L_{\mu-\lambda})\cong A$ (see Proposition \ref{lemma:push.structure} below for a generalization of this fact), and the homomorphism $L_{\lambda}^\Lambda\to L_{\mu}^\Lambda$ corresponding to $a\in A$ is described on local sections of $L_{\lambda}^\Lambda$ by $t^{-\lambda}\mapsto at^{-\mu}$.
Therefore, the map $\bigoplus_j L_{\lambda_j}^\Lambda \to \bigoplus_i L_{\lambda_i}^\Lambda$ is determined by a ``matrix'' $(a_{ij})$ of elements of $A$, and sends the local section $t^{-\lambda_j}$ to $\sum_i a_{ij} t^{-\lambda_i}$. Of course, even if the index sets are infinite, for every $j$ there is only a finite number of indices $i$ for which $a_{ij}\neq 0$.

To prove the statement we pass to the universal cover $\wti{X}_\log=\Hom(P^\gp, \bR)\to X_\log=\Hom(P^\gp,\bS^1)$. On $\wti{X}_\log$ the sheaves $\wti{L}_\lambda^\Lambda$ are constant sheaves. Recall the description of Example \ref{examp:sheaves.std.log.pt}: set
\begin{equation}\label{eq:B}
B=A[T_1,\hdots, T_r, S^{\lambda'} \mid \lambda'\in \Lambda^+]/(S^{\lambda'+\lambda''}-S^{\lambda'}\cdot S^{\lambda''} \,, \, S^p-f_p) \mid \lambda', \lambda''\in \Lambda^+, p\in P^+).
\end{equation}
where $f_p\in A$ is the image of $p$ via $\beta$.

Then $\wti{\cO}_\Lambda^\log$ is the constant sheaf $\underline{B}_{\wti{X}_\log}$, and the sheaf $\wti{L}_\lambda^\Lambda$ is the constant sheaf associated with the $B$-module $t^{-\lambda}\cdot B$ (recall that $t^{-\lambda}$ is a formal symbol, that keeps track of the $\bZ(P)$-equivariant structure). The pullback to $\wti{X}_\log$ of the map $\bigoplus_j L_{\lambda_j}^\Lambda \to \bigoplus_i L_{\lambda_i}^\Lambda$ that we are considering is then completely determined by a homomorphism of $B$-modules $\phi\colon\bigoplus_j t^{-\lambda_j}\cdot B \to \bigoplus_i t^{-\lambda_i}\cdot B$, that sends $t^{-\lambda_j}$ to $\sum_i a_{ij} t^{-\lambda_i}$ with $a_{ij}\in A$.

Let $K$ be the kernel of $\phi$. Then the kernel $\cK$ of the map of sheaves $\bigoplus_j \wti{L}_{\lambda_j}^\Lambda \to \bigoplus_i \wti{L}_{\lambda_i}^\Lambda$ is the constant sheaf on $\wti{X}_\log$ associated with the $B$-module $K$. Let us fix a point $y\in \wti{X}_\log$. We will prove that every element of the stalk $\cK_y\cong K$ is in the image of a morphism of sheaves $\bigoplus_k \wti{L}_{\lambda_k}^\Lambda\to\bigoplus_j \wti{L}_{\lambda_j}^\Lambda$ on the whole $\wti{X}_\log$ of the form $t^{-\lambda_k}\mapsto \sum_j c_{jk} t^{-\lambda_j}$ with $c_{jk}\in A$, that moreover lands entirely in $\cK$. The fact that $c_{jk}\in A$ assures that this morphism of sheaves will descend to a morphism $\bigoplus_k L_{\lambda_k}^\Lambda\to \bigoplus_j{L}_{\lambda_j}^\Lambda$ on $X_\log$. This will be enough to conclude the proof of the lemma, by taking a big direct sum indexed by all elements of the stalks of the sheaf $\cK$.

Let us fix an element $k=\sum_j b_j t^{-\lambda_j}$ in the stalk of the kernel $\cK_y \subseteq \bigoplus_j (\wti{L}_{\lambda_j}^\Lambda)_y\cong \bigoplus_j t^{-\lambda_j} \cdot B$ (so that $b_j\in B$). First, note that thanks to (\ref{eq:homs}) we can group together the terms of $k$ such that ${\lambda_j}\equiv\lambda_{j'} \; \mod \;  P^\gp$. Every resulting partial sum of terms will still be in $\cK_y$, so we can assume that ${\lambda_j}\equiv\lambda_{j'} \; \mod\; P^\gp$ for every $\lambda_j$ and $\lambda_{j'}$ appearing with non-zero coefficient in $k$.

The fact that $k$ is in the kernel of $\phi$ implies that
$$
\phi\left(\sum_j b_jt^{-\lambda_j}\right)=\sum_j b_j \left(\sum_{i} a_{ij} t^{-\lambda_i}\right)=\sum_i \left(\sum_j a_{ij}b_j\right)t^{-\lambda_i}=0
$$
in $\bigoplus_i t^{-\lambda_i} \cdot B$. Therefore for every $i$ we have $\sum_j a_{ij}b_j=0$ in $B$. Now using the description of $B$ given by (\ref{eq:B}), let us write $$b_j=\sum_{\underline{a},\mu}\gamma_{j,\underline{a},\mu} T^{\underline{a}} S^\mu \, ,$$ where $\underline{a}$ denotes a vector $(a_1,\hdots, a_r)$ of $r$ non-negative integers, $\mu\in \Lambda \setminus \langle P^+\rangle $, $T^{\underline{a}}=T_1^{a_1}\cdots T_r^{a_r}$ and $\gamma_{j,\underline{a},\mu}\in A$. The equation $\sum_j a_{ij}b_j=0$ gives then $\sum_{j,\underline{a},\mu} a_{ij}\gamma_{j,\underline{a},\mu} T^{\underline{a}} S^\mu=0$ for every $i$. This in turn implies that $\sum_j a_{ij}\gamma_{j,\underline{a},\mu}=0$ for every $i, \underline{a}$ and $\mu$.

This shows that for every fixed $\underline{a}, \mu$, the element $\sum_j \gamma_{j,\underline{a},\mu} t^{-\lambda_j}$ of $(\bigoplus_j \wti{L}_{\lambda_j}^\Lambda)_y$ is actually in $\cK_y$. Choose any $\lambda\in \Lambda^\gp$ such that $\lambda\equiv \text{ some } \lambda_j \;\mod \; P^\gp$. We can then define a morphism of sheaves $\wti{L}_{\lambda}^\Lambda \to \bigoplus_j \wti{L}_{\lambda_j}^\Lambda$ by sending $t^{-\lambda}$ to $\sum_j \gamma_{j,\underline{a},\mu} t^{-\lambda_j}$. This lands entirely in $\cK$, and at $y$ the element
$$k_{\underline{a},\mu}=\sum_j \gamma_{j,\underline{a},\mu}T^{\underline{a}}S^\mu t^{-\lambda_j}$$
is in the image of this map. Now observe that $k=\sum_{\underline{a},\mu} k_{\underline{a},\mu}$. This finally implies that the element $k$ is in the image of the map $\bigoplus_{\underline{a},\mu}\wti{L}_{\lambda}^\Lambda\to \cK\subseteq  \bigoplus_j \wti{L}_{\lambda_j}^\Lambda$ defined by $t^{-\lambda}_{\underline{a},\mu}\mapsto \sum_j \gamma_{j,\underline{a},\mu} t^{-\lambda_j}$. This concludes the proof of the lemma.
\end{proof}

Back to sheaf cohomology, we can now iterate the argument outlined before the last lemma, and obtain a chain of isomorphisms $$H^i(X_\log,F)\cong H^{i+1}(X_\log, F_1)\cong H^{i+2}(X_\log,F_2)\cong \cdots \cong H^{i+k}(X_\log, F_k)$$ for $i>0$. For $i=1$, as soon as $1+k>r$ (where $r$ is the rank of $P^\gp$), the last cohomology group has to vanish (as $X_\log \cong (\bS^1)^r$ is a manifold of dimension $r$), and hence we get $H^1(X_\log, F)=0$, as we wanted to prove.
\end{proof}

\begin{remark}\label{rmk:exact}
The previous proposition is the reason why it is important to bring the sheaf $\cO_X^\log$ into the picture: without tensoring the other sheaves by it, this proposition does not hold.

For example, consider the standard log point $(\Spec \bC,\bN)$, with Kato-Nakayama space $\tau\colon \bS^1\to \Spec \bC$. The structure sheaf downstairs is $\bC$, and its pullback is the constant sheaf $\underline{\bC}_{\bS^1}$. Clearly $$R^1\tau_*\tau^{-1}\cO_{\Spec \bC}=R^1\tau_*\underline{\bC}_{\bS^1}=H^1(\bS^1,\bC)\neq 0$$ in this case. On the other hand we do have $R^1\tau_*(\tau^{-1}\bC\otimes_{\tau^{-1}\bC} \cO_{\Spec \bC}^\log)=0$.

The heuristic here is that the non-trivial geometry that is introduced by the Kato-Nakayama construction obstructs the exactness of $\tau_*$, and tensoring with the sheaf $\cO_X^\log$ (which has sections that interact with this geometry) balances this out.

Without exactness, it might still be that part of the arguments go through, but for example it is not clear that the equivalence between parabolic sheaves and sheaves on $X_\log$ would respect exactness, something which is certainly desirable.
\end{remark}

Now we can deduce that $\tau_*$ respects quasi-coherence and finite presentation. We will need the following proposition.

\begin{proposition}\label{lemma:push.structure}
The natural morphism $\cO_X\to \tau_*\cO_\Lambda^\log$ is an isomorphism, and for every $\lambda\in \Lambda^\gp(X)$ the sheaf $\tau_*L_\lambda^\Lambda$ is finitely presented.
\end{proposition}

\begin{proof}
The first assertion was proven in \cite[Proposition 3.7]{illusie-kato-nakayama} in the case $\Lambda=\overline{M}$. In the general case the proof is similar.

For the second point, we can localize where $X$ and $\Lambda$ have charts $P\to \overline{M}(X)$ and $\Lambda_0\to \Lambda(X)$ with $P\subseteq \Lambda_0 \subseteq P_\bR$. We claim that for every $\lambda\in \Lambda_0^\gp$ there is an isomorphism
$$
\tau_*L_\lambda^{\Lambda_0}\cong \varinjlim_{P^\gp \, \ni \, p\leq \lambda}L_p
$$
where the map $L_p\to L_{p'}$ for $p\leq p'$, i.e. $p'=p+q$ for some $q\in P$ and $L_{p'}\cong L_p\otimes_{\cO_X} L_{q}$, is given by multiplication by the section $s_{q}\in \Gamma(L_{q})$. After we prove this, it is sufficient to note (because $P$ is finitely generated) that this is a finite colimit of coherent sheaves, hence coherent itself.

To prove the claim, note that there is a natural map $\tau^*L_p\to L_\lambda^{\Lambda_0}$ for every $p\leq \lambda$, that by adjunction gives $L_p\to \tau_*L_\lambda^{\Lambda_0}$. Taking the colimit we obtain a map $\varinjlim_{P^\gp \, \ni \, p\leq \lambda}L_p\to \tau_*L_\lambda^{\Lambda_0}$. We can check that this is an isomorphism on the stalks. Let us denote by  $p_1,\hdots, p_k\in P^\gp$ the elements of $P^\gp$ that are maximal among those such that $p\leq \lambda$, so that $\varinjlim_{P^\gp \, \ni \, p\leq \lambda}L_p$ is a quotient of $L_{p_1}\oplus \cdots\oplus L_{p_k}$.

For a point $x\in X$, set $A=\cO_{X,x}$. Note that the map $(\varinjlim_{P^\gp \, \ni \, p\leq \lambda}L_p)_x\to (\tau_*L_\lambda^{\Lambda_0})_x$ can be identified with the natural map
$$
\varinjlim_{P^\gp \, \ni \, p\leq \lambda}t^{-p}A \to (t^{-\lambda}B)^{\bZ(P)}
$$
where $B$ as in Example \ref{examp:sheaves.std.log.pt} is the ring
$$
A[T_1,\hdots, T_r, S^{\lambda'} \mid \lambda'\in{\Lambda_0}^+]/(S^{\lambda'+\lambda''}-S^{\lambda'}\cdot S^{\lambda''} \,, \, S^p-f_p \mid \lambda', \lambda''\in {\Lambda_0}^+, p\in P^+)
$$
(here $f_p$ is the image of $p$ via $\beta\colon A^\times\oplus P\to A$, the restriction of the log structure of $X$ to the local ring at $x$) and where $^{\bZ(P)}$ denotes taking invariants with respect to the the action of the group $\bZ(P)$ (see Example \ref{examp:sheaves.std.log.pt} for a description of the action). Let us check that this map is an isomorphism.

For injectivity, it suffices to check that every $t^{-{p_i}}A\to (t^{-\lambda}B)^{\bZ(P)}$ is injective, where $p_i\in P^\gp$ is one of the maximal elements with $p\leq \lambda$. In fact, this map is given by $a\cdot t^{-{p_i}}\mapsto aS^{\lambda'}\cdot t^{-\lambda}$, where $\lambda' \in {\Lambda_0}$ is $\lambda-p_i$. In fact, $\lambda'\in {\Lambda_0}\setminus \langle P^+\rangle$ (because of maximality of $p_i$), and therefore $aS^{\lambda'}=0$ implies $a=0$, showing injectivity.

Let us check now that the map is surjective. Take an element $s=\sum_{\underline{a},\mu} \gamma_{\underline{a},\mu} T^{\underline{a}}S^\mu\cdot t^{-\lambda} \in t^{-\lambda}B$, where the notation is as in the proof of Lemma \ref{lemma:surjection}: $\underline{a}=(a_1,\hdots, a_r)$ is a vector of non-negative integers, $T^{\underline{a}}=T_1^{a_1}\cdots T_r^{a_r}$, $\mu\in {\Lambda_0}\setminus \langle P^+\rangle$ and $\gamma_{\underline{a},\mu}\in A$.

By how $\bZ(P)$ acts on the $T_i$, it is clear that for $s$ to be invariant we need to have $\gamma_{\underline{a},\mu}=0$ for $\underline{a}\neq \underline{0}$. Assuming this, the action of $g\in \bZ(P)$ on $s$ is then given by
$$
g\cdot \left(\sum_{\mu} \gamma_{\mu} S^\mu\cdot t^{-\lambda}\right)=\sum_\mu e^{2\pi i g_\bR(\mu-\lambda)}\gamma_\mu S^\mu \cdot t^{-\lambda}
$$
and $s$ is invariant if and only if $g_\bR(\mu-\lambda)\in \bZ$ for every $g\in \bZ(P)$ (i.e. $\mu-\lambda\in P^\gp$) and for every $\mu$ with $\gamma_\mu\neq 0$.

Assuming this is satisfied, for every such $\mu$, let $p_\mu=\lambda-\mu$. Since $\mu\in {\Lambda_0}$ we have $p_\mu\leq \lambda$, and since $\mu\in{\Lambda_0}\setminus \langle P^+\rangle$, it follows that $p_\mu$ is one of the maximal elements $p_i$. Now for this fixed $\mu$ we have that $\gamma_\mu S^\mu \cdot t^{-\lambda}$ is in the image of $\varinjlim_{P^\gp \, \ni \, p\leq \lambda}t^{-p}A \to (t^{-\lambda}B)^{\bZ(P)}$, precisely it is the image of the element $\gamma_\mu\cdot t^{-p_\mu} \in t^{-p_\mu} A$. This immediately implies surjectivity of this map, and concludes the proof.
\end{proof}

\begin{proposition}\label{prop:qc.to.qc}
The functor $\tau_*\colon \Mod(\cO_\Lambda^\log)\to \Mod(\cO_X)$ sends $\Qcoh_\Lambda(X)$ into $\Qcoh(X)$ and $\FP_\Lambda(X)$ into $\FP(X)$.
\end{proposition}

\begin{proof}
First note that, since $\tau$ is a proper map of topological spaces, the functor $\tau_*$ commutes with filtered colimits.

Since $\tau_*\colon \Qcoh_\Lambda(X)\to \Mod(\cO_X)$ is exact (Proposition \ref{prop:push.exact}) and $\tau_*L_\lambda^\Lambda$ is finitely presented (Proposition \ref{lemma:push.structure}), a local presentation for $F \in \FP_\Lambda(X)$ gives a local presentation of $\tau_*F$ as a cokernel of a map of coherent sheaves of $\cO_X$-modules, so $\tau_*F\in \Coh(X)$. The fact that $\tau_*$ commutes with filtered colimits lets us conclude also that $\tau_*\Qcoh_\Lambda(X)\subseteq \Qcoh(X)$.
\end{proof}

To conclude this section, we point out that the projection formula holds for the map $\tau\colon X_\log\to X$ and quasi-coherent sheaves on the Kato-Nakayama space. This is in analogy with the projection formula for the root stacks of \cite[Proposition 2.2.10]{talpo} and \cite[Proposition 4.16]{TV}.

\begin{proposition}[Projection formula]\label{lemma:projection.formula}
We have:
\begin{itemize}
\item for $F\in \Qcoh_\Lambda(X)$ and $G\in \Qcoh(X)$, the natural map $$\tau_*F\otimes_{\cO_X} G\to \tau_*(F\otimes_{\cO_\Lambda^\log} \tau^*G)$$ is an isomorphism, and
\item for $G\in \Qcoh(X)$ the natural map $G\to \tau_*\tau^*G$ is an isomorphism.
\end{itemize}
\end{proposition}

The last item has been proven, in the case where $\Lambda=\overline{M}$, in \cite[Proposition 3.7 (3)]{illusie-kato-nakayama}.

\begin{proof}
This follows formally from exactness of $\tau_*$ (Proposition \ref{prop:push.exact}) and the fact that $\tau_*\cO_\Lambda^\log\cong \cO_X$ (Proposition \ref{lemma:push.structure}), for example as in \cite[Corollary 5.3]{olsson-starr}.
\end{proof}

\section{The correspondence}\label{sec:correspondence}

We are now ready to state and prove the main result of this paper.

\begin{theorem}\label{thm:main}
Let $X$ be a fine saturated log analytic space, with log structure $\alpha\colon M\to \cO_X$. Then for every $\overline{M}^\gp$-saturated quasi-coherent sheaf of monoids $\overline{M}\subseteq \Lambda\subseteq \overline{M}_\bR$ there is an equivalence of categories $\Phi\colon \Qcoh_\Lambda(X)\to \Par(X,\Lambda)$.  Moreover, the equivalence respects exactness, and restricts to the subcategories of finitely presented sheaves $\FP_\Lambda(X)$ and $\FP\Par(X,\Lambda)$.
\end{theorem}

The proof will be an adaptation of the one of \cite[Theorem 6.1]{borne-vistoli}, that we briefly recalled in (\ref{sec:reminder}) above.

\begin{example}
Before proceeding with the proof in the general case it is useful to sketch it for the standard log point.

Let therefore $X$ be the standard log point $(\Spec \bC, \bN)$, and assume we have a $\bZ$-saturated monoid $\bN\subseteq \Lambda\subseteq \bR_+$. For the Kato-Nakayama space and its universal cover we have natural homeomorphisms $X_\log\cong \bS^1$ and $\wti{X}_\log\cong \bR$. Recall from Example \ref{examp:sheaves.std.log.pt} that the sheaf $\wti{\cO}_\Lambda^\log$ is the constant sheaf associated with the ring
$$
B=\bC[T,S^\lambda\mid \lambda\in \Lambda^+]/(S^{\lambda+\lambda'}-S^\lambda\cdot S^{\lambda'}, S^n \mid n\in \bN)
$$
with the $\bZ$-equivariant structure given by
\begin{align*}
k\cdot T&=T-2\pi i k\\
k\cdot S^\lambda& =e^{2 \pi i k\lambda}S^\lambda
\end{align*}
for $k\in \bZ$.
Note that the ring $B$ can alternatively be described as $\bC[T,S^\lambda\mid \lambda\in \Lambda\cap (0,1)]$ where $S^\lambda\cdot S^{\lambda'}=S^{\lambda+\lambda'}$ if $\lambda+\lambda'<1$, and $0$ otherwise. In the same way, the line bundle $\wti{L}_\lambda^\Lambda$ on $\wti{X}_\log$ is the constant sheaf associated with the $B$-module $t^{-\lambda}\cdot B$, where the $\bZ(P)$-equivariant structure is ``twisted'' by $e^{2\pi i (-) \lambda}$.

Let us define the functor $\Phi\colon \Qcoh_\Lambda(X)\to \Par(X,\Lambda)$: for $\lambda\in \Lambda^\wt$ and a quasi-coherent sheaf $F$ of $\cO_\Lambda^\log$-modules, we set $\Phi(F)_\lambda=\tau_*(F\otimes_{\cO_\Lambda^\log} L_\lambda^\Lambda)$. For $\lambda\leq \lambda'$ (i.e. $\lambda'=\mu+\lambda$ for some $\mu\in \Lambda$) there is a map $\Phi(F)_\lambda\to \Phi(F)_{\lambda'}$, induced by the multiplication $L_\lambda^\Lambda\to L_\lambda^\Lambda\otimes_{\cO_\Lambda^\log} L_\mu^\Lambda\cong L_{\lambda'}^\Lambda$ by the section $s_\mu$ of $L_\mu^\Lambda$. The projection formula for $\tau_*$ gives the isomorphisms $\rho^{\Phi(F)}_{a,\lambda}$ required by Definition \ref{def:parabolic.sheaf}, and the action of the functor $\Phi$ on arrows is clear.

Let us describe the quasi-inverse $\Psi\colon\Par(X,\Lambda)\to \Qcoh_\Lambda(X)$. Starting from a quasi-coherent parabolic sheaf $E\colon \Lambda^\wt\to \Qcoh(X)$, we consider  the direct sum $$\wti{E}=\bigoplus_{\lambda\in \Lambda^\gp\cap [0,1)} \wti{\tau}^{-1}E_\lambda$$ on $\wti{X}_\log$, and the tensor product $\wti{\Psi}(E)=\wti{E}\otimes_{\wti{\tau}^{-1}\cO_X}\wti{\cO}_X^\log$. We are going to equip this sheaf with a structure of $\bZ$-equivariant sheaf of $\wti{\cO}_\Lambda^\log$-modules. This will give the quasi-coherent sheaf of $\cO_\Lambda^\log$-modules $\Psi(E)$ corresponding to $E$ by descent along $\wti{X}_\log\to X_\log$.

First, we define the action of $k\in \bZ$ on a section $f \in  \wti{\tau}^{-1}E_\lambda$ as $k\cdot f=e^{2\pi i k\lambda}f$, and on the section $T$ of $\wti{\cO}_X^\log$ (corresponding to the ``formal logarithm'' of the generator of the monoid $\bN$) by $k\cdot T=T-2\pi i k$ as usual. As for the structure of $\wti{\cO}_\Lambda^\log$-module, using the description of this sheaf of rings recalled above we let $T$ act naturally on the factor $\wti{\cO}_X^\log$ and trivially on $\wti{E}$. The variable $S^{\lambda'}$ on the other hand acts trivially on $\wti{\cO}_X^\log$, and acts on $f\in  \wti{\tau}^{-1}E_\lambda$ as the pullback of the map $E_{(\lambda,\lambda+\lambda')}\colon E_\lambda\to E_{\lambda+\lambda'}$ coming from the structure of parabolic sheaf if $\lambda+\lambda'<1$, otherwise we also compose with the isomorphism $E_{\lambda+\lambda'}\cong E_{\lambda+\lambda'-1}\otimes L_1\cong E_{\lambda+\lambda'-1}$. Here $L_1$ denotes the line bundle associated with $1\in \bN$ via the DF structure $L\colon \bN\to \Div(X)$ of $X$, and it is canonically trivialized, since $X$ has a Kato chart. One can check that the $\bZ$-equivariant structure is compatible with this $\wti{\cO}_\Lambda^\log$-module structure.

Let us sketch the proof that $\Psi$ is a quasi-inverse for $\Phi$: starting with a parabolic sheaf $E\in \Par(X,\Lambda)$ and for $\lambda\in \Lambda$, we have $\Phi(\Psi(E))_\lambda=\tau_*(\Psi(E)\otimes_{\cO_\Lambda^\log} L_\lambda^\Lambda)$. This can also be computed as $\wti{\tau}_*^\bZ(\wti{\Psi}(E)\otimes_{\wti{\cO}_\Lambda^\log} \wti{L}_\lambda^\Lambda)$, where $^\bZ$ denotes taking $\bZ$-invariants. We want to show that for every $\lambda$ we have $E_\lambda\cong \wti{\tau}_*^\bZ(\wti{\Psi}(E)\otimes_{\wti{\cO}_\Lambda^\log} \wti{L}_\lambda^\Lambda)$ (and these isomorphisms will also be compatible with the maps giving the parabolic structure).

If $\lambda\in \Lambda^\gp\cap [0,1)$, then the map $E_\lambda\to \wti{\tau}_*^\bZ(\wti{\Psi}(E)\otimes_{\wti{\cO}_\Lambda^\log} \wti{L}_\lambda^\Lambda)$ is defined by sending $f\in E_\lambda$ to the $\bZ$-invariant section $$(f\otimes 1)\otimes t^{-\lambda}\in \left(\Big(\bigoplus_{\lambda\in \Lambda^\gp\cap [0,1)}\wti{\tau}^{-1}E_\lambda \Big)\otimes_{\wti{\tau}^{-1}\cO_X}\wti{\cO}_X^\log\right)\otimes_{\wti{\cO}_\Lambda^\log} \wti{L}_\lambda^\Lambda.$$ This is clearly injective, and it is not hard to prove that it is surjective (see the coming proof of Theorem \ref{thm:main} for details).
For a general $\lambda\in \Lambda^\wt$, let $\lambda_0 \in [0,1)$ be the unique element such that $\lambda\equiv\lambda_0 \; \mod \; \bZ$. Then we have $E_{\lambda}\cong E_{\lambda_0}$ and $\wti{L}_{\lambda}^\Lambda\cong  \wti{L}_{\lambda_0}^\Lambda$ (note that as for $L_1$, all line bundles $L_a$ coming from the DF structure on $X$ for $a\in \bZ$ are canonically trivialized). The existence of the desired isomorphism for $\lambda$ follows.

On the other hand, given a quasi-coherent sheaf $F$ of $\cO_\Lambda^\log$-modules on $X_\log$, we want to check that $\Psi(\Phi(F))\cong F$. By passing to $\wti{X}_\log$, it suffices to check that $\wti{\Psi}(\Phi(F))\cong \pi^{-1}(F)$ (where $\pi\colon \wti{X}_\log\to X_\log$ is the projection). Note that for every $\lambda\in \Lambda^\gp\cap [0,1)$ we have a map $$\wti{\tau}^{-1}\Phi(F)_\lambda=\wti{\tau}^{-1}(\wti{\tau}_*^\bZ(F\otimes_{\wti{\cO}_\Lambda^\log} \wti{L}_\lambda^\Lambda))\to \pi^{-1}(F)$$ described by $f\otimes t^{-\lambda}\mapsto f$. This induces a map $\wti{\Phi(F)}=\bigoplus_{\lambda\in \Lambda^\gp\cap [0,1)} \wti{\tau}^{-1}\Phi(F)_\lambda\to \pi^{-1}F$, and then a map
$$
\eta\colon \wti{\Psi}(\Phi(F))=\left(\bigoplus_{\lambda\in \Lambda^\gp\cap [0,1)} \wti{\tau}^{-1}\Phi(F)_\lambda\right)\otimes_{\wti{\tau}^{-1}\cO_X} \wti{\cO}_X^\log \to  \pi^{-1}F.
$$
This map is an isomorphism.

We refer to the proof of the theorem for a complete justification of this claim. Here we just explain briefly why it is surjective, if $F$ is in addition finitely presented. In that case we have a surjective morphism $\bigoplus_i \wti{L}_{\lambda_i}^\Lambda \to \pi^{-1}F$ on $\wti{X}_\log$, so given a section of $\pi^{-1}F$, this is going to be the image of some element $\sum_i a_i t^{-\lambda_i}$, where $a_i$ is a section of $\cO_\Lambda^\log$ (since $X$ is a log point, all these sheaves are constant sheaves on $\wti{X}_\log$, and the description of Example \ref{examp:sheaves.std.log.pt} applies). Let $f_i \in \pi^{-1}F$ be the image of the section $t^{-\lambda_i}$. We argue that every $f_i$ is in the image of the map $\eta$ above: in fact, the action of $\bZ$ on $f_i$ is given by $k\cdot f_i=e^{-2\pi i k\lambda_i}f_i$, and hence $f_i\otimes t^{\lambda_i}$ is a section of $\wti{\tau}_*^\bZ(F\otimes_{\wti{\cO}_\Lambda^\log} \wti{L}_{-\lambda_i}^\Lambda)$. If $\mu_i$ is the unique element in $\Lambda^\gp\cap [0,1)$ with $\mu_i\equiv -\lambda_i \; \mod \; \bZ$, then $\wti{L}_{-\lambda_i}^\Lambda \cong \wti{L}_{\mu_i}^\Lambda$ (recall once again that all line bundles $L_a$ on $X$ for $a\in \bZ$ are canonically trivialized), and hence there is a corresponding element $f_i\otimes t^{-\mu_i} \in \wti{\tau}_*^\bZ(F\otimes_{\wti{\cO}_\Lambda^\log} \wti{L}_{\mu_i}^\Lambda)$, whose image under $\eta$ is exactly the section $f_i$. This justifies surjectivity of $\eta$.
\end{example}

The proof of Theorem \ref{thm:main} in the general case will be along the same lines as the one sketched in the previous example. The construction of the quasi-coherent sheaf $\Psi(E)$ on $X_\log$ corresponding to the parabolic sheaf $E$ will be slightly more complicated though, because in general there is no nice set of lifts of the elements of $\Lambda^\gp/\overline{M}^\gp$ to $\Lambda$, as it happens when $\bN\subseteq \Lambda\subseteq \bR_+$, where we can just take $\Lambda^\gp\cap [0,1)$.

\begin{proof}[Proof of Theorem \ref{thm:main}] We will proceed in a few steps.

\medskip

\emph{Construction of $\Phi$}:

\medskip

Let us construct the functor $\Phi$. Recall that on $X_\log$ we have a symmetric monoidal functor $L^\Lambda\colon \Lambda^\gp \to \Div_{(X_\log,\cO_\Lambda^\log)}$ (see (\ref{sec:rings.on.kn})). Assume that we are given a quasi-coherent sheaf $F\in  \Qcoh_\Lambda(X)$ on $X_\log$, and we want to produce a parabolic sheaf with weights in $\Lambda$. Suppose $U\subseteq X$ is open, and take an object $\lambda \in \Lambda^\wt(U)$. Set 
$$
\Phi(F)_\lambda:=\tau_*(F|_{\tau^{-1}U}\otimes_{\cO_\Lambda^\log|_{\tau^{-1}U}} L^\Lambda_\lambda).
$$
For an arrow $\lambda \to \lambda'$ in $\Lambda^\wt(U)$, corresponding to $\mu \in \Lambda(U)$ such that $\lambda'=\mu+\lambda$, we define $\Phi(F)_\lambda \to \Phi(F)_{\lambda'}$ to be the morphism induced by multiplication by the section $s_{\mu}$ from $L^\Lambda_\lambda$ to $L^\Lambda_\lambda\otimes_{\cO_\Lambda^\log|_{\tau^{-1}U}} L^\Lambda_\mu\cong L^\Lambda_{\lambda'}$, by tensoring by $F$ and pushing forward to $X$.

If $\lambda'=m+\lambda$ with $m\in\overline{M}^\gp(U)$, we have
\begin{align*}
\Phi(F)_{\lambda'}=\Phi(F)_{\lambda+m}& = \tau_*(F|_{\tau^{-1}U}\otimes_{\cO_\Lambda^\log|_{\tau^{-1}U}} L^\Lambda_{\lambda+m}) \\ & \cong  \tau_*(F|_{\tau^{-1}U}\otimes_{\cO_\Lambda^\log|_{\tau^{-1}U}} L^\Lambda_\lambda\otimes_{\cO_\Lambda^\log|_{\tau^{-1}U}} \tau^*L_m) \\ &  \cong  \tau_*(F|_{\tau^{-1}U}\otimes_{\cO_\Lambda^\log|_{\tau^{-1}U}} L^\Lambda_\lambda)\otimes_{\cO_U} L_m \\ &  =  \Phi(F)_\lambda\otimes_{\cO_U} L_m,
\end{align*}
where we used the projection formula for $\tau$.  This gives the required isomorphism $\rho^{\Phi(F)}_{m,\lambda}$, and the map $\Phi(F)_\lambda\to \Phi(F)_{\lambda'}$ corresponds to multiplication by the section $s_m$ of $L_m$.

If $V\subseteq U\subseteq X$, then it is clear that $\Phi(F)_\lambda|_V \cong \Phi(F|_{V_\log})_{\lambda|_V}$ canonically, and  this restriction is also compatible with the isomorphisms $\rho^{\Phi(F)}_{m, \lambda}$. The other conditions in the definition of a parabolic sheaf are easily verified.

Finally, we check that $\Phi(F)$ is a quasi-coherent parabolic sheaf. This is a local question on $X$, so we can assume that $F$ is a filtered colimit of finitely presented sheaves, as in Definition \ref{def:qcoh}. Assume for the time being that we have proven that $\Phi$ sends finitely presented sheaves to finitely presented parabolic sheaves (Definition \ref{def:qcoh.par}).
Then, since $\Phi$ respects filtered colimits (because tensor product does, and $\tau_*$ does as well since $\tau$ is a proper map), $\Phi(F)$ will be a filtered colimit of finitely presented parabolic sheaves, and hence quasi-coherent. We will verify the assertion about finitely presented sheaves at the end of the proof.

We leave the construction of the action of the functor $\Phi\colon \Qcoh_\Lambda(X)\to \Par(X,\Lambda)$ on arrows to the reader.

\medskip

\emph{Local construction of the quasi-inverse $\Psi$}:

\medskip

To prove that the functor $\Phi$ is an equivalence we will construct an inverse locally on $X$ (observe that both $\Qcoh_{\Lambda|_{-}}(-)$ and $\Par(-,\Lambda|_{-})$ are stacks on the classical site of $X$). So we may assume that $X$ has a Kato chart $X\to \bC(P)$ for a toric monoid $P$, and that $\Lambda$ has a compatible chart $ \Lambda_0\to \Lambda(X)$, with $P\subseteq \Lambda_0\subseteq P_\bR$.

In this case we have fairly explicit descriptions of the sheaves $\cA_\Lambda$ and $\cO_\Lambda$ in terms of their pullback to the covering space $\wti{X}_\log$ of $X_\log$. Recall from (\ref{sec:local.description}) that in this situation we have a cartesian diagram
$$
\xymatrix{
\widetilde{X}_\log\ar[r]\ar[d] & \bH(P)\ar[d]\\
X_\log\ar[r]\ar[d] & (\bR_{\geq 0}\times \bS^1)(P)\ar[d]\\
X\ar[r] & \bC(P)
}
$$
where the two top vertical maps are covering spaces for the group $\bZ(P)=\Hom(P,\bZ)$. Moreover, as in the previous example, all line bundles $L_a$ for $a\in P^\gp$ are canonically trivialized, although this will not be an important point in the present proof. Finally, recall from Proposition \ref{prop:parabolic.chart} that there is a natural equivalence $\Par(X,\Lambda)\cong \Par(X,\Lambda_0)$.

Let us assume that $E\colon \Lambda_0^\wt\to \Qcoh(X)$ is a parabolic sheaf for $\Lambda_0$. We will produce a sheaf of $\widetilde{\cO}_\Lambda^\log$-modules on $\widetilde{X}_\log$ equipped with a $\bZ(P)$-equivariant structure. This will give a sheaf of $\cO_\Lambda^\log$-modules on $X_\log$ by descent.

Recall that $\widetilde{\tau}\colon \widetilde{X}_\log\to X$ denotes the natural projection. Also, in this situation the sheaf $\wti{\cO}_\Lambda$ is a quotient of the sheaf
$$
\wti{\cO}_{\Lambda_0}=\wti{\tau}^{-1}{\cO}_X \otimes_{\bC[P]}\bC[\Lambda_0],
$$
where $\bC[P]\to \wti{\tau}^{-1}{\cO}_X$ is obtained from the map $X\to \bC(P)$. The kernel of $\wti{\cO}_{\Lambda_0}\to \wti{\cO}_\Lambda$ is locally generated by elements of the form $t^\lambda-1$, where $\lambda$ is in the kernel of the map of sheaves of monoids $(\Lambda_0)_X\to \Lambda$.

Starting from $E$, we consider the direct sum $\bigoplus_{\lambda\in \Lambda_0^\gp} E_\lambda$ as a sheaf of $\cO_X$-modules on $X$. 
We pull this back to $\widetilde{X}_\log$ and obtain
$$
\widetilde{E}:=\bigoplus_{\lambda\in \Lambda_0^\gp} \widetilde{\tau}^{-1}E_\lambda,
$$
which is a sheaf of $\wti{\tau}^{-1}{\cO}_X$-modules.

Consider the sheaf of $\wti{\tau}^{-1}\cO_X$-algebras $A:=\bigoplus_{a\in P^\gp}\wti{\tau}^{-1}L_a\cong \bigoplus_{a\in P^\gp} t^{-a}\cdot \wti{\tau}^{-1}\cO_X$, where $t^{-a}$ is just a placeholder variable. First note that, on top of its natural $\wti{\tau}^{-1}{\cO}_X$-module structure, $\wti{E}$ is also a sheaf of $A$-modules, via the map $\wti{E}\otimes_{\cO_X}A\to \wti{E}$ constructed as follows: we can define
$$
\wti{E}\otimes_{\wti{\tau}^{-1}{\cO}_X}A=\bigoplus_{a\in P^\gp, \; \lambda \in \Lambda_0^\gp} (\wti{\tau}^{-1}E_\lambda \otimes_{\wti{\tau}^{-1}\cO_X} \wti{\tau}^{-1}L_a )\to \widetilde{E}=\bigoplus_{\lambda \in \Lambda_0^\gp} \widetilde{\tau}^{-1}E_\lambda
$$
by using the pullback via $\wti{\tau}$ of the given isomorphisms $E_\lambda \otimes_{\cO_X} L_a \cong E_{\lambda+a}$ for $a\in P^\gp$ and $\lambda\in \Lambda_0^\gp$.

Moreover, the sheaf $\wti{E}$ on $\wti{X}_\log$ has a $\bZ(P)$-equivariant structure: if $g\in \bZ(P)$ and $f$ is a section of $\wti{\tau}^{-1}E_\lambda$, we define $g\cdot f=e^{2\pi i g_\bR(\lambda)} f$ as a section of $\wti{\tau}^{-1}E_\lambda$.  Moreover, $\wti{E}$ has an action of the constant sheaf $\bC[\Lambda_0]$ on $\wti{X}_\log$: for a section $x^\lambda$ of $\bC[\Lambda_0]$, we define the action on the piece $\wti{\tau}^{-1}E_\mu$ to be the pullback via $\wti{\tau}$ of the map $E_\mu\to E_{\mu+\lambda}$ coming from the structure of parabolic sheaf. This is compatible (by property (a) in Definition \ref{def:parabolic.sheaf}) with the action of $\bC[P]$, induced by $\bC[P]\to A$, where this map sends $x^p$ to the section $t^{-p}\cdot f_p \in t^{-p}\cdot \wti{\tau}^{-1}\cO_X\subseteq A$ (recall from Notation \ref{notation:variables}  that $f_p$ denotes the image of $p \in P$ in $\cO_X$).
This makes $\widetilde{E}$ into a $\bZ(P)$-equivariant sheaf of $A\otimes_{\bC[P]}\bC[\Lambda_0]$-modules. 

Now consider the morphism of sheaves of algebras $A\to \wti{\cO}_X^\log$ determined by sending each $t^{-a}$ with $a\in P^\gp$ to $1\in \wti{\tau}^{-1}\cO_X\subseteq \wti{\cO}_X^\log$. 
The tensor product
$$
\wti{\Psi}(E):=\wti{E}\otimes_{A}\wti{\cO}_{X}^\log
$$
has a structure of a $\bZ(P)$-equivariant sheaf of $\wti{\cO}_X^\log\otimes_{\bC[P]}\bC[\Lambda_0]=\wti{\cO}_{\Lambda_0}^\log$-modules. This last operation has the effect of imposing that the action of $P^\gp$ is trivial (i.e. it identifies $e_\lambda\in E_\lambda$ with the image $e_\lambda \otimes t^{-p} \in E_\lambda\otimes_{\cO_X}L_p\cong E_{\lambda+p}$ for $p\in P^\gp$), and of ``adding the sections of $\cO_X^\log$ as coefficients''. 

\begin{remark}
Imposing that the action of $P^\gp$ is trivial might look strange, but should be compared with the following situation for root stacks: as recalled in (\ref{sec:reminder}), given a parabolic sheaf $E$ with weights in the Kummer extension $P\to Q$, to obtain a quasi-coherent sheaf on the root stack $$\sqrt[Q]{X}=[(X\times_{\Spec\bZ[P]}\Spec \bZ[Q])/\mu_{Q/P}]\cong [\underline{\Spec}_X(\cO_X[P^\gp]\otimes_{\bZ[P]}\bZ[Q])/\widehat{Q}]$$
one forms the sheaf $\bigoplus_{q\in Q^\gp}E_q$ on $X$, and equips it with a structure of $Q^\gp$-graded $\cO_X[P^\gp]\otimes_{\bZ[P]}\bZ[Q]$-module to obtain a $\widehat{Q}$-equivariant sheaf on $\underline{\Spec}_X(\cO_X[P^\gp]\otimes_{\bZ[P]}\bZ[Q])$.

However, the presentation that we are using for the Kato-Nakayama space is closer to the first expression of $\radice[Q]{X}$ as a quotient stack, and the way to obtain a sheaf for that presentation is to pullback along the zero section $X\to \underline{\Spec}_X(\cO_X[P^\gp])=X\times \widehat{P}$, where $\widehat{P}$ is the Cartier dual of $P^\gp$, i.e. the algebraic torus $\Hom(P^\gp,\bG_m)$. This corresponds to forcing the action of $P^\gp$ to be trivial, since the sheaf of ideals of $X$ in $X\times \widehat{P}$ is exactly generated by the elements $t^p-1$ in $\cO_X[P^\gp]$ for $p\in P^\gp$.
\end{remark}

By descent along $\pi\colon \wti{X}_\log\to X_\log$, this gives us a sheaf  of $\cO_{\Lambda_0}^\log$-modules on $X_\log$. Now observe that the action of $\cO_{\Lambda_0}^\log$ factors through $\cO_\Lambda^\log$: it is not hard to check that if $\lambda$ is a local section in the kernel of $(\Lambda_0)_X\to \Lambda$, then the action of $t^\lambda \in \cO_{\Lambda_0}$ is given by the identity on the pieces of the parabolic sheaf $E$, and hence also on the sheaf $\wti{E}$.
Denote the resulting sheaf of $\cO_\Lambda^\log$-modules by $\Psi(E)$.

It is straightforward to define the action on arrows, so that $\Psi$ becomes a functor $\Par(X,\Lambda)\to \Mod(\cO_\Lambda^\log)$. Note that $\Psi$ respects filtered colimits: in fact, we can check that $\wti{\Psi}$ does so, and this is clear, because direct sums and tensor products commute with filtered colimits.
Therefore, again assuming that we have proven that $\Psi$ sends finitely presented parabolic sheaves to finitely presented sheaves, it follows that if $E$ is a quasi-coherent parabolic sheaf, then $\Psi(E)$ is a quasi-coherent sheaf of $\cO_\Lambda^\log$-modules.

This defines the quasi-inverse $\Psi\colon \Par(X,\Lambda)\to \Qcoh_\Lambda(X)$.

\medskip

\emph{$\Psi$ is a quasi-inverse}:

\medskip

Let us check that, still in the case where there is a Kato chart $X\to \bC(P)$, the functors $\Phi$ and $\Psi$ are quasi-inverses. Consider a parabolic sheaf $E \in \Par(X,\Lambda_0)$, and the parabolic sheaf $\Phi(\Psi(E))$. For every $\lambda\in \Lambda_0^\wt$, the sheaf $\Phi(\Psi(E))_\lambda$ is the pushforward $\tau_*(\Psi(E)\otimes_{\cO_{\Lambda_0}^\log} L_\lambda^{\Lambda_0})\in \Qcoh(X)$. We can also compute this as $\wti{\tau}_*^{\bZ(P)}(\wti{\Psi}(E)\otimes_{\wti{\cO}_{\Lambda_0}^\log} \wti{L}_\lambda^{\Lambda_0})$ (the superscript $^{\bZ(P)}$ denotes ${\bZ(P)}$-invariants). 

Note that there is a natural $\cO_X$-linear injective morphism $$E_\lambda\to \wti{\tau}_*^{\bZ(P)}(\wti{\Psi}(E)\otimes_{\wti{\cO}_{\Lambda_0}^\log} \wti{L}_\lambda^{\Lambda_0}),$$ that sends $f\in E_\lambda$ to the section
$$
(f\otimes 1)\otimes t^{-\lambda}\in \left( \Big( \bigoplus_{\lambda\in \Lambda_0^\gp}\wti{\tau}^{-1}E_\lambda\Big) \otimes_{A} \wti{\cO}_{X}^\log\right)\otimes_{\wti{\cO}_{\Lambda_0}^\log} \wti{L}_\lambda^{\Lambda_0}.
$$
We claim that this map is an isomorphism.

Let $s=\sum_i (a_i\otimes b_i) \otimes c_i$ be a section of $\wti{\tau}_*^{\bZ(P)}(\wti{\Psi}(E)\otimes_{\wti{\cO}_{\Lambda_0}^\log} \wti{L}_\lambda^{\Lambda_0})$,  seen as a $\bZ(P)$-invariant section of $\wti{\Psi}(E)\otimes_{\wti{\cO}_{\Lambda_0}^\log} \wti{L}_\lambda^{\Lambda_0}$ on $\wti{X}_\log$.
In particular $a_i$ are ``homogeneous'' sections of $\wti{E}$ (i.e. in some $E_{\lambda_i}$), $b_i$ are sections of $\wti{\cO}_{\Lambda_0}^\log$, and $c_i$ are sections of $\wti{L}_\lambda^{\Lambda_0}$. By bilinearity and by moving the coefficients to the other factors, we can assume that $c_i=t^{-\lambda}$ (the local generator of the line bundle) for every $i$. Moreover, it is clear that if $s$ is $\bZ(P)$-invariant, then $b_i$ has to be in $\wti{\tau}^{-1}{\cO}_X \subseteq \wti{\cO}_{X}^\log$ for every $i$ (recall that locally $\wti{\cO}_X^\log$ is a polynomial ring with coefficients in $\wti{\tau}^{-1}\cO_X$, and $\bZ(P)$ acts on each ``indeterminate'' by adding integer multiples of $2\pi i$). By moving coefficients to the first factor, we can assume that $b_i=1$ for every $i$.

Hence we are reduced to a section of the form $\sum_i( a_i \otimes 1)\otimes t^{-\lambda}$. By the explicit form of the action of $\bZ(P)$, it is clear that this is invariant if and only if $e^{2\pi i g_\bR(\lambda_i-\lambda)}=1$ for every $i$ and every $g\in \bZ(P)$ (where $a_i\in E_{\lambda_i}$), or equivalently if $\lambda_i-\lambda$ is zero in $\Lambda_0^\gp/P^\gp$ for all $i$, i.e. $\lambda_i\equiv \lambda \; \mod\; P^\gp$. Finally, we claim that each term $(a_i\otimes 1) \otimes t^{-\lambda}$ is equal to some $(d_i\otimes 1)\otimes t^{-\lambda}$ with $d_i\in E_\lambda$. Indeed, since $\lambda_i-\lambda\in P^\gp$, by acting on $a_i$ via $P^\gp$ we can obtain a section of $E_\lambda$. But by construction, the action of $P^\gp$ on $\wti{\Psi}(E)$ is the identity.

This gives an isomorphism $\Phi(\Psi(E))_\lambda=\wti{\tau}_*^{\bZ(P)}(\wti{E}\otimes_{\wti{\cO}_{\Lambda_0}^\log} \wti{L}_\lambda^{\Lambda_0})\cong E_\lambda$ for every $\lambda \in \Lambda_0^\wt$. By how the action of the section $s_\lambda$ of $L_\lambda^{\Lambda_0}$ is defined on $\Psi(E)$ it is also clear that the map $E_\lambda\cong \Phi(\Psi(E))_\lambda\to \Phi(\Psi(E))_{\lambda'}\cong E_{\lambda'}$ coincides with the given one $E_\lambda\to E_{\lambda'}$ for each $\lambda \leq \lambda'$ in $\Lambda_0^\wt$. After easily checking that these isomorphisms are compatible with restrictions to open subsets and functorial in the parabolic sheaf $E$, we conclude that there is a functorial isomorphism of parabolic sheaves $\Phi(\Psi(E))\cong E$.

Conversely, let us start from a quasi-coherent sheaf $F$ on $X_\log$, and show that there is a natural isomorphism $F\cong \Psi(\Phi(F))$. For this we can pull everything back along $\pi\colon \wti{X}_\log\to X_\log$, and check that $\pi^{-1}F\cong \wti{\Psi}(\Phi(F))$. Note that in fact there is a functorial morphism $\wti{\Psi}(\Phi(F))\to \pi^{-1}F$, obtained from the natural maps $$\wti{\tau}^{-1}\Phi(F)_\lambda\to \pi^{-1}F$$ defined by sending a section $f\otimes t^{-\lambda}$ of $\wti{\tau}^{-1}(\wti{\tau}_*^{\bZ(P)}(F\otimes_{\wti{\cO}_{\Lambda_0}^\log} \wti{L}_\lambda^{\Lambda_0}))$ to the section $f$ of $\pi^{-1}F$.
By further localizing on $X$ we can assume that $F$ is a filtered colimit of finitely presented sheaves, and thus it suffices to prove that claim for $F$ finitely presented.

Assume that $F$ has a presentation
$$
\xymatrix{
\bigoplus_{j} L_{\lambda_j}^{\Lambda_0} \ar[r] & \bigoplus_{i} L_{\lambda_i}^{\Lambda_0} \ar[r] & F\ar[r] & 0
}
$$
with finitely many summands, that we can pull back to $\wti{X}_\log$, obtaining a presentation
$$
\xymatrix{
\bigoplus_{j} \wti{L}_{\lambda_j}^{\Lambda_0} \ar[r] & \bigoplus_{i} \wti{L}_{\lambda_i}^{\Lambda_0} \ar[r] & \pi^{-1}F\ar[r] & 0.
}
$$
Recall moreover that $\wti{\Psi}(\Phi(F))=\left(\bigoplus_{\lambda\in \Lambda_0^\gp} \Phi(F)_\lambda\right) \otimes_{A}\wti{\cO}_{X}^\log$, and 
$\Phi(F)_\lambda=\tau_*(F\otimes_{\cO_{\Lambda_0}^\log}L_\lambda^{\Lambda_0})$.

From the exactness of the various functors we obtain a commutative diagram with exact rows
$$
\xymatrix{
\wti{\Psi}(\Phi(\bigoplus_{j} {L}_{\lambda_j}^{\Lambda_0}))\ar[d] \ar[r] & \ar[d] \wti{\Psi}(\Phi(\bigoplus_{i} {L}_{\lambda_i}^{\Lambda_0})) \ar[r] & \wti{\Psi}(\Phi(F))\ar[r] \ar[d]& 0 \\
\bigoplus_{j} \wti{L}_{\lambda_j}^{\Lambda_0} \ar[r] & \bigoplus_{i} \wti{L}_{\lambda_i}^{\Lambda_0} \ar[r] & \pi^{-1}F\ar[r] & 0 .
}
$$
Finally, it is not hard to check that two leftmost vertical maps are isomorphisms: one is reduced to checking the statement for a single sheaf $L_\lambda^\Lambda$, and in this case it is an explicit calculation similar to the one in the proof of proposition \ref{lemma:push.structure}. Hence also the rightmost map is an isomorphism, as we wanted to prove.

\medskip

\emph{Exactness and finite presentation}:

\medskip

It is clear from exactness of $\tau_*$ (Proposition \ref{prop:push.exact}) that $\Phi$ respects exactness. Let us show that it restricts to the subcategories of finitely presented sheaves on both sides.

Assume that $F$ is a finitely presented sheaf of $\cO_\Lambda^\log$-modules on $X_\log$, as in Definition \ref{def:qcoh}. By localizing on $X$ we can assume that we have a presentation
$$
\xymatrix{
\bigoplus_{j} L_{\lambda_j}^\Lambda \ar[r] & \bigoplus_{i} L_{\lambda_i}^\Lambda \ar[r] & F\ar[r] & 0
}
$$
with finitely many summands, and that we have charts $P\to \overline{M}(X)$ and $\Lambda_0\to \Lambda(X)$. By Proposition \ref{prop:qc.to.qc} all the pieces $\Phi(F)_\lambda$ are finitely presented sheaves on $X$. Moreover, consider the sub-weight system $R\subseteq \Lambda_0^\wt$ given by the orbits for the $P^\gp$-action of the (finitely many) elements $\lambda_i$ and $\lambda_j$ appearing in the presentation above. We claim that $\Phi(F)$ is the induction of a parabolic sheaf with weights in $R$.

Specifically, we claim that $\Phi(F)$ is isomorphic to $\ind_R^{\Lambda_0^\wt}(G)$, where $G$ is the sheaf $\res_R^{\Lambda_0^\wt}(\Phi(F))$. To verify this, it is enough to prove that for every $\lambda\in \Lambda_0^\wt$ the map
$$
\varinjlim_{R\, \ni \, r \leq \lambda}\Phi(F)_r \to \Phi(F)_\lambda
$$
is an isomorphism.

By applying the two functors to the presentation of $F$ above (which stays exact), we see that it is enough to check the statement for the sheaves  $\bigoplus_{i} L_{\lambda_i}^\Lambda$ and $\bigoplus_{j} L_{\lambda_j}^\Lambda$, for which it is an easy computation.

Conversely, assume that $E$ is a finitely presented parabolic sheaf on $X$, and let us show that $\Psi(E)$ is finitely presented on $X_\log$. As above, we can localize on $X$ where there are charts $P\to \overline{M}(X)$ and $\Lambda_0\to \Lambda(X)$, and where $E$ comes via induction from a finite sub-system $R\subseteq \Lambda_0^\wt$. We can also assume that each of the (finitely many) sheaves $E_r$ with $r\in \Lambda_0\setminus \langle P^+\rangle$ admits a presentation as
$$
\xymatrix{
\cO_X^{\oplus J_r} \ar[r]^{f_r} & \cO_X^{\oplus I_r} \ar[r] & E_r\ar[r] & 0
}
$$
where $I_r$ and $J_r$ are finite sets. It is easy to check that the sheaf $\wti{\Psi}(E)$ on $\wti{X}_\log$ has a presentation of the form
$$
\xymatrix@C=1.5cm{
\bigoplus_r (\wti{L}_{-r}^\Lambda)^{\oplus J_r} \ar[r]^{\oplus (f_r)_{-r}} & \bigoplus_r (\wti{L}_{-r}^\Lambda)^{\oplus I_r} \ar[r] & \wti{\Psi}(E)\ar[r] & 0
}
$$
which shows that $\Psi(E)$ is finitely presented. The map $\bigoplus_r (\wti{L}_{-r}^\Lambda)^{\oplus I_r}\to \wti{\Psi}(E)$ is defined by sending the generators of the factor $(\wti{L}_{-r}^\Lambda)^{\oplus I_r}$ to the images in $\wti{\tau}^{-1}E_r$ of the generators of $\cO_X^{\oplus I_r}$ via the map in the presentation of $E_r$ (recall that  $\wti{\Psi}(E)=(\oplus_{\lambda\in \Lambda_0^\gp} \wti{\tau}^{-1}E_\lambda)\otimes_{A} \wti{\cO}_{X}^\log$).

This concludes the proof.
\end{proof}

\begin{remark}
One can check, using the construction of $\Phi$, that for two $\overline{M}^\gp$-saturated quasi-coherent sheaves of monoids $\overline{M}\subseteq \Lambda\subseteq \Lambda'\subseteq \overline{M}_\bR$, with induced weight systems $W=\Lambda^\wt$ and $W'=(\Lambda')^\wt$, the restriction $\res_W^{W'}$ and induction $\ind_W^{W'}$ on parabolic sheaves correspond respectively to equipping a sheaf of $\cO_{\Lambda'}^\log$-modules with the structure of $\cO_\Lambda^\log$-module coming from the natural map $\cO_\Lambda^\log\to \cO_{\Lambda'}^\log$, and to taking the tensor product $-\otimes_{\cO_\Lambda^\log}\cO_{\Lambda'}^\log$.
\end{remark}

\subsection{Comparison with root stacks}\label{sec:comparison.root}

To conclude, we compare the equivalence of Theorem \ref{thm:main} to the ones between parabolic sheaves and sheaves on root stacks, of \cite{borne-vistoli} and \cite{TV}.

Let $X$ be a fine saturated log scheme locally of finite type over $\bC$, and assume also that $X$ is proper. With this assumption, because of Proposition \ref{prop:gaga} we can compare the equivalence of Theorem \ref{thm:main} (that involves analytic sheaves) with the ones of \cite{borne-vistoli} and \cite{TV} (that are formulated for parabolic sheaves on schemes).

For every $n$ there is a canonical morphism of topological stacks $\phi_n\colon X_\log\to \radice[n]{X}_\top$, coming for example from the fact that the projection $\pi_n\colon \radice[n]{X}\to X$ induces an isomorphism $(\radice[n]{X})_\log\xrightarrow{\cong} X_\log$ (see the proof of \cite[Proposition 4.6]{TVnew}). These are compatible for different indices, and induce a morphism $\phi_\infty\colon X_\log\to \radice[\infty]{X}_\top$ (see \cite[Proposition 4.1]{knvsroot} or \cite[Section 3.4]{TVnew}).

By \cite[Theorem 6.1]{borne-vistoli} and \cite[Theorem 7.3]{TV} we have compatible equivalences of abelian categories $\Phi_n\colon \Qcoh(\radice[n]{X})\to \Par(X, \frac{1}{n}\overline{M})$ and $\Phi_\infty\colon \Qcoh(\radice[\infty]{X})\to \Par(X,\bQ)$. Strictly speaking, these equivalences are formulated and proven for parabolic sheaves on schemes, but the same reasoning should apply for complex analytic spaces. Alternatively, one can rely on GAGA results for proper Deligne-Mumford stacks, for example as formulated, in a more general setting, in \cite{porta}.

We will prove that these equivalences are compatible with the equivalence $\Phi$ of Theorem \ref{thm:main}, in the following sense. We have several structures on the root stacks that we can pull back via the morphisms $\phi_n$.
The stack $\radice[n]{X}_\top$ has a structure sheaf $\cO_n$, and a tautological DF structure $L_n\colon \pi_n^{-1}\frac{1}{n}\overline{M}\to \Div_{\radice[n]{X}}$ (if $n=\infty$, then $\frac{1}{n}\overline{M}$ denotes $M_{\bQ}$). These are all compatible with respect to pullbacks along the projections $\radice[m]{X}\to \radice[n]{X}$.
In the case of the infinite root stack, it is better to think about $\cO_\infty$ as $\varinjlim_n \cO_n$, so that $\phi_\infty^{*}\cO_\infty$ on the space $X_\log$ is $\varinjlim_n \phi_n^{*}\cO_n$.

Here when we write $\phi_n^{*}F$ for a sheaf $F$ of $\cO_n$-modules on $\radice[n]{X}_\top$, we mean the sheaf $\phi_n^{-1}F\otimes_{\tau^{-1}\cO_X} \cO_X^\log$ on $X_\log$.

\begin{proposition}\label{prop:root.stack.sheaves}
There is a sequence of compatible isomorphisms of sheaves of rings $\phi_n^{*}\cO_n\cong \cO^\log_{\frac{1}{n}\overline{M}}$, where we are using the notation of (\ref{sec:rings.on.kn}) for the sheaf on the right hand side. Moreover, the pullback DF structure $\phi_n^{*}L_n\colon \phi_n^{-1}(\pi_n^{-1}\frac{1}{n}\overline{M})\to \Div_{(X_\log,{\cO^\log_{\frac{1}{n}\overline{M}}})}$ is canonically isomorphic to the DF structure described in (\ref{sec:rings.on.kn}).
\end{proposition}

\begin{proof}
We can reduce to checking the claim on log schemes of the form $X=\Spec \bC[P]$ for a toric monoid $P$. Moreover, to prove the first assertion it is enough to prove that there are compatible isomorphisms $\phi_n^{-1}\cO_n\cong \cO_{\frac{1}{n}\overline{M}}$.

As briefly explained in (\ref{sec:kn}), we have a diagram
$$
\xymatrix{
\wti{X}_\log\times \bZ(P)\ar[r]\ar@<4pt>[d] \ar@<-4pt>[d] & \bC(\frac{1}{n}P)\times \mu_n(P)\ar@<4pt>[d] \ar@<-4pt>[d]\\
\wti{X}_\log\ar[r]^{\wti{\phi}_n}\ar[d] & \bC(\frac{1}{n}P)\ar[d]\\
X_\log\ar[r]^{\phi_n} & \radice[n]{\bC(P)}
}
$$
where $\wti{\phi}_n\colon \wti{X}_\log=\Hom(P,\bH)\to \bC(\frac{1}{n}P)=\Hom(\frac{1}{n}P, \bC)$ is given by composing $\frac{1}{n}P\cong P \to \bH$ with $\bH\to \bH\to \bC$, where the first map is $(x,y)\mapsto (\radice[n]{x}, y/n)$ and the second map is $(x,y)\mapsto x\cdot e^{iy}$. The homomorphism $\bZ(P)=\Hom(P,\bZ)\to \mu_n(P)=\Hom(P,\bZ/n\bZ)$ is also defined by composing with $\bZ\to \bZ/n\bZ$, and $\wti{\phi}_n$ is equivariant with respect to this homomorphism.

Let us denote by $\wti{\cO}_n$ the pullback of $\cO_n$ to $\bC(\frac{1}{n}P)$, and by $\wti{\pi}_n$ the projection $\bC(\frac{1}{n}P)\to \radice[n]{\bC(P)}\to  X$. We can reduce to showing that $\wti{\phi}_n^{-1}\wti{\cO}_{n}\cong \wti{\cO}_{\frac{1}{n}\overline{M}}$ as $\bZ(P)$-equivariant sheaves. This is clear from the fact that $\wti{\cO}_{\frac{1}{n}\overline{M}}$ is the quotient of $\wti{\tau}^{-1}\cO_X\times_{\bC[P]}\bC[\frac{1}{n}P]$ by the ideal $I$ generated by local sections of the form $t^a-1$, where $a\in \underline{\frac{1}{n}P}$ maps to zero in $\overline{M}_\bR$: we have a natural map $ \wti{\tau}^{-1}\cO_X\times_{\bC[P]}\bC[\frac{1}{n}P]\to \wti{\phi}_n^{-1}\wti{\cO}_n$ induced by $\wti{\tau}^{-1}\cO_X=\wti{\phi}_n^{-1}\wti{\pi}_n^{-1}\cO_X\to \wti{\phi}_n^{-1}\wti{\cO}_n$ and $\bC[\frac{1}{n}P]\to \wti{\cO}_n$, which factors through the quotient by the sheaf of ideals $I$. One can verify that the resulting map is an isomorphism, for example by looking at the stalks.

The assertion about the DF structure is proved similarly. 
\end{proof}

\begin{remark}
The preceding discussion explains \cite[Remark 4.7]{TVnew}: if we consider for a log algebraic stack $X$ its Kato-Nakayama space $X_\log$ as a ringed topological stack, equipped with the sheaf of rings $\cO_X^\log$, then the isomorphism $(\radice[n]{X})_\log\cong X_\log$ is not an isomorphism of ringed topological stacks, since the structure sheaf of $(\radice[n]{X})_\log$ is identified with the sheaf $\cO_{\frac{1}{n}\overline{M}}^\log$ on $X_\log$.
\end{remark}

Let us now check that the equivalence $\Phi$ of Theorem \ref{thm:main} is compatible with the analogous equivalences on the root stacks.

\begin{proposition}\label{prop:finite.weights.compatible}
The following diagram of functors is 2-commutative.
$$
\xymatrix@=.5cm{
\Qcoh(\radice[n]{X})\ar[rr]^{\phi_n^*} \ar[rd]_{\Phi_n} & &  \Qcoh_{\frac{1}{n}\overline{M}}(X_\log)\ar[ld]^{\Phi} \\
&\Par(X,\frac{1}{n}\overline{M}) & 
}
$$
Moreover all the functors restrict to the subcategories of finitely presented sheaves, and the diagrams for different $n$ are compatible with respect to pushforward and pullback along projections between root stacks, and induction and restriction between the categories of parabolic sheaves.

In particular, for every $n$ the pullback functor $\phi_n^*$ is an equivalence.
\end{proposition}

\begin{proof}
We can assume that $X$ has a global Kato chart $X\to \Spec\bC[P]$. Fix a quasi-coherent sheaf $F\in \Qcoh(\radice[n]{X})$. We have to check that $\Phi(\phi_n^*F)$ and $\Phi_n$ are the same parabolic sheaf on $X$ with weights in $\frac{1}{n}P$.

For an element $\frac{1}{n}a\in \frac{1}{n}P^\gp$, we have $$\Phi_n(F)_{\frac{1}{n}a}=(\pi_n)_*(F\otimes_{\cO_n} L_{\frac{1}{n}a})\, ,$$ where as above $\pi_n \colon \radice[n]{X}\to X$ is the projection and $L\colon \frac{1}{n}P^\gp\to \Pic({\radice[n]{X}})$ is the symmetric monoidal functor corresponding to the universal DF structure on the root stack. On the other hand $$\Phi(\phi_n^*F)_{\frac{1}{n}a}=\tau_*(\phi_n^*F\otimes_{\cO_{\frac{1}{n}\overline{M}}^\log} L_{\frac{1}{n}a}^{\frac{1}{n}\overline{M}})\, ,$$ where $L^{\frac{1}{n}\overline{M}}\colon  \frac{1}{n}P^\gp\to \Pic{(X_\log,\cO_{\frac{1}{n}\overline{M}}^\log)}$ is the analogous symmetric monoidal functor on the ringed space $(X_\log,\cO_{\frac{1}{n}\overline{M}}^\log)$.

Now note that $L_{\frac{1}{n}a}^{\frac{1}{n}\overline{M}}\cong \phi_n^*L_{\frac{1}{n}a}$, so that we have
$$
\Phi(\phi_n^*F)_{\frac{1}{n}a}=\tau_*(\phi_n^*(F\otimes_{\cO_n} L_{\frac{1}{n}a})).
$$
Finally, since $(X_\log, \cO_{\frac{1}{n}\overline{M}}^\log)=((\radice[n]{X})_\log, \cO_{\radice[n]{X}}^\log) \to \radice[n]{X}$ can be seen as the projection from the Kato-Nakayama space of $\radice[n]{X}$ and  $\tau_*=(\pi_n)_*\circ (\phi_n)_*$, the analogue of Proposition \ref{lemma:projection.formula} (whose proof we leave to the reader) implies that $(\phi_n)_*\circ \phi_n^*\cong \id$, and hence $\Phi_n(F)_{\frac{1}{n}a}\cong \Phi(\phi_n^*F)_{\frac{1}{n}a}$.

The remaining assertions are routinely checked.
\end{proof}

\begin{remark}
Let us also briefly comment on the relationship between our setup and the similar one of \cite{illusie-kato-nakayama}. In Section 3 of that paper, the authors consider the Kummer-\'etale topos $X_\ket$, which is equipped with a natural morphism $X_\ket\to X_\an$, and construct a morphism of topoi $\tau_\ket\colon X_\log\to X_\ket$ factoring $\tau\colon X_\log\to X_\an$. They consider a sheaf of rings $\cO_X^\klog$ on $X_\log$, defined as $\tau_\ket^{-1}\cO_X^\ket \otimes_{\tau^{-1}\cO_X} \cO_X^\log$, where 
$\cO_X^\ket$ is the structure sheaf of the Kummer-\'etale topos. In (3.2), the authors give a  description of $\cO_X^\klog$ that coincides with the one of Remark \ref{rmk:sheaves.description}, seeing it as being generated over $\cO_X$ by formal logarithms (which correspond to $\cO_X^\log$) and formal $n$-th roots for every $n$ (which correspond to $\cO_X^\ket$), of sections of $\overline{M}$.

The ringed topos $(X_\ket,\cO_X^\ket)$ is equivalent to the infinite root stack, in the following sense. The natural functor that sends a Kummer-\'etale morphism $U\to X$ to the map between infinite root stacks $\radice[\infty]{U}\to \radice[\infty]{X}$ induces a morphism of sites from an opportunely defined small \'etale site of $\radice[\infty]{X}$ to the Kummer-\'etale site of $X$, which gives a morphism of topoi $\rho \colon \radice[\infty]{X}\to X_\ket$ (where we identify the stack $\radice[\infty]{X}$ with its small \'etale topos), which is proved to be an equivalence in \cite[Theorem 6.21]{TV}.

Finally, as mentioned above, there is a natural isomorphism $\cO_X^\klog\cong \cO_{\overline{M}_\bQ}^\log$, and, by Proposition \ref{prop:root.stack.sheaves}, the last sheaf can also be seen as the pullback $\phi_\infty^*\cO_\infty$ along $\phi_\infty\colon X_\log\to \radice[\infty]{X}$.
\end{remark}

\bibliographystyle{plain}
\bibliography{biblio}

\begin{thebibliography}{10}

\bibitem{biswas}
Indranil Biswas.
\newblock Parabolic bundles as orbifold bundles.
\newblock {\em Duke Math. J.}, 88(2):305--325, 1997.

\bibitem{biswas-dhillon}
Indranil Biswas and Ajneet Dhillon.
\newblock Semistability criterion for parabolic vector bundles on curves.
\newblock {\em J. Math. Sci. Univ. Tokyo}, 18(2):181--191, 2011.

\bibitem{boden}
Hans~U. Boden and Yi~Hu.
\newblock Variations of moduli of parabolic bundles.
\newblock {\em Math. Ann.}, 301(3):539--559, 1995.

\bibitem{borne}
Niels Borne.
\newblock Sur les repr\'esentations du groupe fondamental d'une vari\'et\'e
  priv\'ee d'un diviseur \`a croisements normaux simples.
\newblock {\em Indiana Univ. Math. J.}, 58(1):137--180, 2009.

\bibitem{borne-vistoli}
Niels Borne and Angelo Vistoli.
\newblock Parabolic sheaves on logarithmic schemes.
\newblock {\em Advances in Mathematics}, 231(3-4):1327--1363, {O}ct 2012.

\bibitem{knvsroot}
D.~Carchedi, S.~Scherotzke, N.~Sibilla, and M.~Talpo.
\newblock Kato-{N}akayama spaces, infinite root stacks, and the profinite
  homotopy type of log schemes.
\newblock {\em Geom. Topol.}, 21(5):3093--3158, 2017.

\bibitem{conrad}
Brian Conrad.
\newblock Relative ampleness in rigid geometry.
\newblock {\em Ann. Inst. Fourier (Grenoble)}, 56(4):1049--1126, 2006.

\bibitem{illusie-kato-nakayama}
L.~Illusie, K.~Kato, and C.~Nakayama.
\newblock Quasi-unipotent logarithmic {R}iemann-{H}ilbert correspondences.
\newblock {\em J. Math. Sci. Univ. Tokyo}, 12:1--66, 2005.

\bibitem{iyer-simpson}
Jaya~N. Iyer and Carlos~T. Simpson.
\newblock The {C}hern character of a parabolic bundle, and a parabolic
  corollary of {R}eznikov's theorem.
\newblock In {\em Geometry and dynamics of groups and spaces}, volume 265 of
  {\em Progr. Math.}, pages 439--485. Birkh\"auser, Basel, 2008.

\bibitem{iyer-simpson1}
Jaya N.~N. Iyer and Carlos~T. Simpson.
\newblock A relation between the parabolic {C}hern characters of the de {R}ham
  bundles.
\newblock {\em Math. Ann.}, 338(2):347--383, 2007.

\bibitem{kashiwara-shapira}
Masaki Kashiwara and Pierre Schapira.
\newblock {\em Sheaves on manifolds}, volume 292 of {\em Grundlehren der
  Mathematischen Wissenschaften [Fundamental Principles of Mathematical
  Sciences]}.
\newblock Springer-Verlag, Berlin, 1994.
\newblock With a chapter in French by Christian Houzel, Corrected reprint of
  the 1990 original.

\bibitem{KN}
Kazuya Kato and Chikara Nakayama.
\newblock Log {B}etti cohomology, log \'etale cohomology, and log de {R}ham
  cohomology of log schemes over {${\bf C}$}.
\newblock {\em Kodai Math. J.}, 22(2):161--186, 1999.

\bibitem{kato-usui}
Kazuya Kato and Sampei Usui.
\newblock {\em Classifying spaces of degenerating polarized {H}odge
  structures}, volume 169 of {\em Annals of Mathematics Studies}.
\newblock Princeton University Press, Princeton, NJ, 2009.

\bibitem{lorenzon}
Pierre Lorenzon.
\newblock Indexed algebras associated to a log structure and a theorem of
  {$p$}-descent on log schemes.
\newblock {\em Manuscripta Math.}, 101(3):271--299, 2000.

\bibitem{maruyama}
M.~Maruyama and K.~Yokogawa.
\newblock Moduli of parabolic stable sheaves.
\newblock {\em Math. Ann.}, 293:77--99, 1992.

\bibitem{matsubara}
Toshiharu Matsubara.
\newblock On log {H}odge structures of higher direct images.
\newblock {\em Kodai Math. J.}, 21(2):81--101, 1998.

\bibitem{metha-seshadri}
V.B. Mehta and C.S. Seshadri.
\newblock Moduli of vector bundles on curves with parabolic structures.
\newblock {\em Math. Ann.}, 248:205--239, 1980.

\bibitem{mochizuki}
Takuro Mochizuki.
\newblock Kobayashi-{H}itchin correspondence for tame harmonic bundles and an
  application.
\newblock {\em Ast\'erisque}, (309):viii+117, 2006.

\bibitem{ogus}
Arthur Ogus.
\newblock On the logarithmic {R}iemann-{H}ilbert correspondence.
\newblock {\em Doc. Math}, 655, 2003.

\bibitem{olsson-starr}
Martin Olsson and Jason Starr.
\newblock Quot functors for {D}eligne-{M}umford stacks.
\newblock {\em Comm. Algebra}, 31(8):4069--4096, 2003.
\newblock Special issue in honor of Steven L. Kleiman.

\bibitem{Ols}
Martin~C. Olsson.
\newblock Logarithmic geometry and algebraic stacks.
\newblock {\em Ann. Sci. \'Ecole Norm. Sup. (4)}, 36(5):747--791, 2003.

\bibitem{porta}
Mauro Porta.
\newblock Derived complex analytic geometry {I}: {G}{A}{G}{A} theorems.
\newblock Preprint, arXiv:1506.09042.

\bibitem{mckay}
S.~Scherotzke, N.~Sibilla, and M.~Talpo.
\newblock On a logarithmic version of the derived {M}c{K}ay correspondence.
\newblock Preprint, arXiv:1612.08961. To appear in \emph{Compos. Math.}

\bibitem{talpo}
Mattia Talpo.
\newblock Moduli of parabolic sheaves on a polarized logarithmic scheme.
\newblock {\em Trans. Amer. Math. Soc.}, 369(5):3483--3545, 2017.

\bibitem{logstr}
Mattia Talpo and Angelo Vistoli.
\newblock A general formalism for logarithmic structures.
\newblock Preprint, arXiv:1703.02663, pulished online in
  \href{https://link.springer.com/article/10.1007\%2Fs40574-017-0149-6}{Boll.
  Unione Mat. Ital.}

\bibitem{TVnew}
Mattia Talpo and Angelo Vistoli.
\newblock The {K}ato-{N}akayama space as a transcendental root stack.
\newblock arXiv:1611.04041, pulished online in
  \href{http://academic.oup.com//imrn/article/doi/10.1093/imrn/rnx079/3770480/The-KatoNakayama-Space-as-a-Transcendental-Root?guestAccessKey=96fed519-1ad1-4a30-ab12-343c93a119ce}{Int.
  Math. Res. Not.}, 2017.

\bibitem{TV}
Mattia Talpo and Angelo Vistoli.
\newblock Infinite root stacks and quasi-coherent sheaves on logarithmic
  schemes.
\newblock {\em Proc. Lond. Math. Soc.}, 117(5):1187--1243, 2018.

\bibitem{thaddeus}
Michael Thaddeus.
\newblock Variation of moduli of parabolic {H}iggs bundles.
\newblock {\em J. Reine Angew. Math.}, 547:1--14, 2002.

\end{thebibliography}

\end{document}